\documentclass[a4paper,11pt]{amsart}

\usepackage[utf8]{inputenc}

\usepackage[T1]{fontenc}
\usepackage{amsmath}
\usepackage{amsfonts}
\usepackage{amssymb}
\usepackage{amsthm}
\usepackage{mathrsfs}
\usepackage{calrsfs}
\usepackage{enumitem}
\usepackage{footnote}
\usepackage{hyperref,tikz}

\allowdisplaybreaks[2] 

\renewcommand{\geq}{\geqslant}
\renewcommand{\leq}{\leqslant}
\renewcommand{\epsilon}{\varepsilon}

\DeclareMathOperator{\Li}{Li}

\hypersetup{colorlinks=true,    linkcolor=MyDarkBlue,
citecolor=BrickRed,       filecolor=BrickRed,       urlcolor=darkgreen
}
\usepackage{color}
\definecolor{darkgreen}{rgb}{0,0.4,0}
\definecolor{MyDarkBlue}{rgb}{0,0.08,0.85}
\definecolor{BrickRed}{rgb}{0.8,0.08,0}

\newtheorem{theorem}{Theorem}
\newtheorem{lemma}{Lemma}
\newtheorem{remark}{Remark}

\newtheorem{proposition}[lemma]{Proposition}
\newtheorem{cor}[lemma]{Corollary}
\newtheorem{definition}{Definition}
\newtheorem{conj}{Conjecture}

\numberwithin{equation}{section}

\title{Constructing discrete harmonic functions in wedges}

\author{Viet Hung Hoang}
\address{CNRS \and Institut Denis Poisson, UMR CNRS 7013, Universit\'e de Tours et Universit\'e d'Orl\'eans, Parc de Grandmont, 37200 Tours, France}
\email{viet.hung-hoang@lmpt.univ-tours.fr}

\author{Kilian Raschel}
\address{CNRS \and Laboratoire Angevin de Recherche en Math\'ematiques, UMR CNRS 6093, Universit\'e d'Angers, 2 Boulevard Lavoisier, 49000 Angers, France}
\email{raschel@math.cnrs.fr}
\thanks{This project has received funding from the European Research Council (ERC) under the European Union's Horizon 2020 research and innovation programme under the Grant Agreement No.\ 759702.}

\author{Pierre Tarrago}
\address{Laboratoire de Probabilit\'es, Statistique et Mod\'elisation, UMR CNRS 8001, Sorbonne Universit\'e, 4 Place Jussieu, 75005 Paris, France}
\email{pierre.tarrago@sorbonne-universite.fr}


\date{\today}

\begin{document}

\begin{abstract}
We propose a systematic construction of signed harmonic functions for discrete Laplacian operators with Dirichlet conditions in the quarter plane. In particular, we prove that the set of harmonic functions is an algebra generated by a single element, which conjecturally corresponds to the unique positive harmonic function.
\end{abstract}

\maketitle

\setcounter{tocdepth}{1}
\tableofcontents

\section{Introduction and main results}
\label{sec:introduction}
\nocite{BoBMMe-18}

\subsection*{Harmonic functions for the standard Laplacian}
Let us first briefly recall some facts around the elliptic operator in $\mathbb R^2$
\begin{equation}
\label{eq:continuous_Laplacian}
   \Delta=\sigma_1\frac{\partial^2}{\partial x^2}+2\sigma_{1,2}\frac{\partial^2}{\partial x\partial y}+\sigma_2\frac{\partial^2}{\partial y^2},
\end{equation}
which we will call continuous or classical Laplacian operator.
Given a domain $D\subset \mathbb R^2$ with boundary $\partial D$, solving the Dirichlet problem on $D$ for the operator $\Delta$ amounts to find the set $H(D)$ of functions $h:D\rightarrow \mathbb{R}$ which are harmonic on $D$, continuous on the closure $\overline{D}$ of $D$, and zero on $\partial D$. This problem admits the unique solution $h=0$ when $D$ is bounded, but the situation is much richer when $D$ is unbounded. 

Assume first that $\Delta$ is the standard Laplacian (i.e., $\sigma_1=\sigma_2=1$ and $\sigma_{1,2}=0$ in \eqref{eq:continuous_Laplacian}) and $D$ is the upper half-plane $\mathcal{H}=\{x+iy:y>0\}$, and look at the associated Dirichlet problem. Defining $h(z)=-h(\overline{z})$ for $z$ in the lower half-plane, the classical Schwarz reflection principle implies that $h$ can be extended to a harmonic function on $\mathbb{C}$. Hence $h$ is solution to this Dirichlet problem if and only if there exists an analytic function $f$ on $\mathbb{C}$ such that $h=\Im f$ (with $\Im$ denoting the imaginary part), and the Schwarz reflection implies that $f$ is self-conjugate, meaning that $f(\overline{z})=\overline{f(z)}$ on $\mathbb C$. The set $H(\mathcal{H})$ therefore admits the explicit description
\begin{align*}
   H(\mathcal{H})&=\bigl\{\Im f_{\vert\mathcal{H}}: f \text{ is analytic on $\mathbb C$ and } f(\overline{z})=\overline{f(z)}\bigr\},\\
   &=\bigl\{\sum_{n\geq 1} a_n \Im(x^{n}): a_n\in\mathbb{R} \text{ and } \vert a_n\vert^{1/n}\rightarrow 0\bigr\}.
\end{align*}

A similar description holds for other cones (not necessarily the half-space $\mathcal{H}$) and Laplacian operators \eqref{eq:continuous_Laplacian} with general covariance matrix 
\begin{equation}
\label{eq:covariance_matrix}
   \sigma=\left(\begin{array}{ll}\sigma_1&\sigma_{1,2}\\\sigma_{1,2}&\sigma_{2}\end{array}\right).
\end{equation}
In particular, for a future use, when $D$ is the positive quadrant $\mathcal Q$, we have
\begin{equation}
\label{eq:basis_continuous_case}
   H(\mathcal Q)=\bigl\{\sum_{n\geq1} a_n h_n^\sigma: a_n\in\mathbb{R} \text{ and } \vert a_n\vert^{1/n}\rightarrow 0\bigr\},
\end{equation}
where the functions $h_n^\sigma$ will be introduced later (see \eqref{eq:h_n^sigma}) and $\sigma$ is given by \eqref{eq:covariance_matrix}.

\subsection*{A glimpse of our results}
Our main objective in the present paper is to prove that in the discrete setting, a surprisingly simple equality of set holds, analogue to \eqref{eq:basis_continuous_case}; see Theorem \ref{thm:main_intro-2} for a precise statement.

A function $h:\mathbb Z^2\to\mathbb R$ is discrete harmonic (an equivalent terminology is preharmonic) in a domain $D\subset\mathbb Z^2$ with respect to the discrete Laplacian operator 
\begin{equation}
\label{eq:def_Laplacian}
    \Delta h(i,j) = \sum_{k,\ell} p_{k,\ell} h(i+k,j+\ell)- h(i,j),
\end{equation}
the set of weights $\{p_{k,\ell}\}$ being fixed, if $\Delta h(i,j)=0$ for all $(i,j)\in D$. In this article, we start from the analytic approach of \cite{CoBo-83} and propose a systematic construction of discrete harmonic functions in the quarter plane (i.e., $D=\mathbb N^2=\{1,2,3,\ldots\}^2$) which vanish on the boundary axes. We go beyond the existing literature, in the sense that our construction works:
\begin{enumerate}[label=(\roman{*}),ref={\rm(\roman{*})}]
     \item\label{item:neg_jumps}for walks with arbitrary big (negative) jumps (see Figure \ref{fig:step_sets}), 
     \item\label{item:neg_values}not only for positive, but also for signed discrete harmonic functions.
\end{enumerate}
The two above features illustrate the robustness of our theory. The constructive aspect will follow from that we will obtain exact expressions for the generating functions of harmonic functions in terms of certain conformal mappings.


\subsection*{Construction of preharmonic functions}

First of all, we would like to review some results in the literature dedicated to constructions of discrete harmonic functions. Let us first mention elementary constructions of preharmonic functions on $\mathbb Z^d$. Discrete polynomials and discrete exponential functions are constructed (mostly iteratively) in \cite{Fe-44,He-49,Is-52,Du-56}; see also \cite{Mu-64} for a construction of preharmonic polynomials in terms of well-signed multinomials. 

Further examples arise when preharmonic functions are defined on sets having certain rigid structures. Picardello and Woess \cite{PiWo-92} prove that discrete harmonic functions for Cartesian products of Markov chains have a product form. In the case of Weyl chambers of type A, Eichelsbacher and K\"onig \cite{EiKo-08} prove that preharmonic functions can be expressed in terms of Vandermonde determinants. K\"onig and Schmid \cite{KoSc-10} demonstrate similar results for Weyl chambers of other types. Still in presence of a Weyl chamber structure, Biane, Bougerol and O'Connell \cite{BiBoOC-05} compute the (continuous) harmonic function (namely, the survival probability) with the help of the reflection principle.

Using the theory of analytic combinatorics in several variables \cite{PeWi-13}, Courtiel et al.\ \cite{CoMeMiRa-17} study simple random walks with drift in various (two-dimensional) wedges and show that the harmonic functions obey to a rigid construction: namely, they are all obtained from a single function (again related to the reflection principle, see \cite[Eq.~(20)]{CoMeMiRa-17}), after some elementary operations (differentiation, division, evaluation). However, this technique seems to work only for random walks satisfying to (a certain algebraic version of) the reflection principle.

Roughly speaking, the Martin boundary theory aims at describing the set of all positive harmonic functions, given a Laplacian operator. See \cite{KuMa-98,IR-08,IRLo-10,Ra-11,Ra-14,LeRa-16} for examples when the Laplacian is related to random walks killed on the boundary of a half-plane or quadrant. This theory only rarely yields a construction of harmonic functions. In the previously cited articles, the computation of Martin boundary relies on an asymptotic study of quotients of Green functions. Let us underline, however, that Ignatiouk-Robert and Loree \cite{IR-08,IRLo-10} find an explicit formula for the exponential growth of all positive preharmonic functions in the Martin boundary.

In another direction, it is natural to ask whether discrete and classical harmonic functions are related, see for instance the papers \cite{Va-09,Va-15} by Varopoulos. In \cite{DeWa-15}, Denisov and Wachtel prove that a certain natural preharmonic function (appearing as a prefactor in the persistence probability asymptotics)\ can be constructed by compensating the classical harmonic function, see \cite[Eq.~(5)]{DeWa-15}. 

Representation theory provides us with many examples of Markov processes for which explicit expressions of harmonic functions exist. For instance, preharmonic functions are expressed in terms of dimensions of irreducible representations in \cite{Bi-91,Bi-92}. Such results are quite powerful, but intrinsically limited to a few step sets. 

We mention here a very last framework, which will be at the heart of the present work, and which relies on complex analysis techniques. As we shall see, the generating functions of preharmonic functions satisfy certain functional equations and, after further analysis, boundary value problems (BVPs). Solving these BVPs eventually yields exact expressions for the generating functions in terms of conformal mappings. The link between harmonic functions and conformal mappings is already illustrated in \cite{Ra-14,LeRa-16}, but we shall go much further here, by considering random walks with arbitrary big negative jumps (see item \ref{item:neg_jumps} above), and by constructing not only positive preharmonic functions (item \ref{item:neg_values}); recall that in \cite{Ra-14,LeRa-16}, the attention is restricted to positive harmonic functions for small step walks.

Obviously, the list of constructive approaches to preharmonic functions could be continued, as the latter are all-present in probability theory and statistical physics, see the book \cite{Pr-12} for recent illustrations.

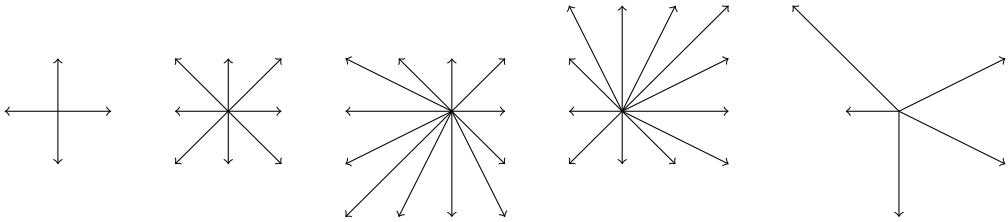
\begin{figure}
\begin{center}
\begin{tikzpicture}[scale=.7] 
    \draw[->,white] (1,2) -- (-1,-2);
    \draw[->,white] (1,-2) -- (-1,1);
    \draw[->] (0,0) -- (0,-1);
    \draw[->] (0,0) -- (0,1);
    \draw[->] (0,0) -- (-1,0);
    \draw[->] (0,0) -- (1,0);
  \end{tikzpicture}
\begin{tikzpicture}[scale=.7] 
    \draw[->,white] (2,2) -- (-2,-2);
    \draw[->,white] (2,-2) -- (-2,2);
    \draw[->] (0,0) -- (0,-1);
    \draw[->] (0,0) -- (0,1);
    \draw[->] (0,0) -- (1,1);
    \draw[->] (0,0) -- (-1,-1);
    \draw[->] (0,0) -- (1,-1);
    \draw[->] (0,0) -- (-1,1);
    \draw[->] (0,0) -- (-1,0);
    \draw[->] (0,0) -- (1,0);
  \end{tikzpicture}
   \begin{tikzpicture}[scale=.7] 
  \draw[->,white] (1.5,2) -- (-2,-2);
    \draw[->,white] (1.5,-2) -- (-2,2);
    \draw[->] (0,0) -- (1,1);
    \draw[->] (0,0) -- (0,1);
    \draw[->] (0,0) -- (1,0);
    \draw[->] (0,0) -- (-2,-1);
    \draw[->] (0,0) -- (-2,-2);
    \draw[->] (0,0) -- (-2,0);
    \draw[->] (0,0) -- (-1,-2);
    \draw[->] (0,0) -- (0,-2);
    \draw[->] (0,0) -- (1,-2);
    \draw[->] (0,0) -- (-2,1);
    \draw[->] (0,0) -- (-1,1);
    \draw[->] (0,0) -- (1,-1);
  \end{tikzpicture}
     \begin{tikzpicture}[scale=.7] 
  \draw[->,white] (2.5,2) -- (-1.5,-2);
    \draw[->,white] (2,-2) -- (-1.5,2);
    \draw[->] (0,0) -- (-1,-1);
    \draw[->] (0,0) -- (0,-1);
    \draw[->] (0,0) -- (-1,0);
    \draw[->] (0,0) -- (2,1);
    \draw[->] (0,0) -- (2,2);
    \draw[->] (0,0) -- (2,0);
    \draw[->] (0,0) -- (1,2);
    \draw[->] (0,0) -- (0,2);
    \draw[->] (0,0) -- (-1,2);
    \draw[->] (0,0) -- (2,-1);
    \draw[->] (0,0) -- (1,-1);
    \draw[->] (0,0) -- (-1,1);
  \end{tikzpicture}
    \begin{tikzpicture}[scale=.7] 
  \draw[->,white] (2,2) -- (-2.5,-2);
    \draw[->,white] (2,-2) -- (-2,2);
    \draw[->] (0,0) -- (0,-2);
    \draw[->] (0,0) -- (-2,2);
    \draw[->] (0,0) -- (2,1);
    \draw[->] (0,0) -- (-1,0);
    \draw[->] (0,0) -- (2,-1);
  \end{tikzpicture}
  \end{center}
  \label{fig:step_sets}
  \caption{Various step sets. From left to right: the simple random walk; the king walk (with jumps to all eight nearest neighbours); an example with big negative jumps; an example with big positive jumps; an arbitrary model.}
\end{figure}

\subsection*{Various analytic approaches} 
As we have seen above, some results on harmonic functions (structure of the Martin boundary, exact expressions, see \cite{Ra-14,LeRa-16})\ have already been obtained using the analytic approach of \cite{FaIaMa-17}. Let us now say a few words about this method. It was initially developed by Malyshev in \cite{Ma-72}, Fayolle and Iasnogorodski \cite{FaIa-79}, to study the stationary distribution of random walks in the quarter plane reflected at the boundary. The stationary distribution generating functions are shown to satisfy BVPs.
From this, explicit expressions (typically, contour integral involving special functions)\ are deduced. See \cite[Chap.~5]{FaIaMa-17} for full details.

Since then, the analytic approach of \cite{FaIaMa-17} has been applied to various contexts: queueing networks \cite{KuSu-03}, potential theory \cite{Ra-11,Ra-14,LeRa-16}, enumerative combinatorics (counting walks in the quarter plane)\ \cite{BMMi-10,KuRa-12,DrHaRoSi-18}. This approach is thus particularly fruitful and leads to precise results (both exact and asymptotic results). 

However, the construction in \cite{FaIaMa-17} is restricted to the case of random walks with jumps to the eight nearest neighbours (see Figure \ref{fig:step_sets}). Its generalization to bigger jumps would require the precise understanding of the location of the branch points on a certain Riemann surface, see \cite{FaRa-15}. This is certainly possible on a few given examples, but the variety of possible behaviors makes us pessimistic to pursue in this direction to develop a general theory.

We will prefer here the alternative analytic approach \cite{CoBo-83} by Cohen and Boxma (see also \cite{Co-92} by Cohen). The starting point is essentially the same (writing functional equations and BVPs for the generating functions) but remarkably, the construction can be done for arbitrary big (negative) jumps without increasing the level of complexity. The analytic approaches of \cite{FaIaMa-17} and \cite{CoBo-83} will be compared in Section \ref{sec:examples}.

\subsection*{Signed harmonic functions}
In most of the literature cited above, the focus is put on positive harmonic functions, for clear probabilistic reasons: for example in relation with the concept of Doob transform, or because many probabilistic estimates use positive harmonic functions. In this paper, we go further and look at signed harmonic functions. Our motivation is fourfold: first, this will allow us to study the structure of the set of harmonic functions (for instance, we shall see that the vector space of harmonic functions having a bounded polynomial growth is finite-dimensional and will give a basis).

Our second motivation is that signed harmonic functions appear in various complete asymptotic expansions of relevant probabilistic or combinatorial quantities. To give a concrete example (related to this paper), let $K$ be a given cone of $\mathbb R^d$ and $\{Z(n)\}_{n\geq0}$ be a zero-mean random walk (with some moment assumptions). Then it is shown in \cite{DeWa-15} that, as $n\to \infty$, the survival probability is asymptotically equivalent to
\begin{equation}
\label{eq:one-term_asymp}
   \mathbb P_x(\tau_K>n)\sim h(x) n^{-\alpha},
\end{equation}
where $x\in K$ is the starting point of the random walk, $\tau_K$ is the first exit time of $\{Z(n)\}_{n\geq0}$ from the cone $K$, $h(x)$ is a harmonic function and $\alpha>0$ is a critical exponent. Assuming that \eqref{eq:one-term_asymp} may be refined as
\begin{equation}
\label{eq:many-terms_asymp}
   \mathbb P_x(\tau_K>n)\sim \sum_{i}h_i(x) n^{-\alpha_i},
\end{equation}
it is shown in \cite{ChFuRa-20} that the $h_i(x)$ may be constructed from certain polyharmonic functions as well as from signed harmonic functions.

Our third motivation comes from potential theory, where sign changing discrete harmonic functions turn out to be more difficult to study, compared to positive harmonic functions. Indeed, such classical tools as Harnack inequality (heavily used in \cite{Mu-06,Va-99}, for instance) do not hold anymore. Any construction of such functions becomes more relevant.

Finally, there might be some interesting features regarding the nodal domains of these sign changing harmonic functions. Indeed, within the connected components of the nodal lines, harmonic functions take a constant sign (by definition). Take an example: the function $h(i,j)=ij(i-j)(i+j)$ is discrete harmonic for the usual Laplacian in dimension $2$ (probability $\frac{1}{4}$ to the four nearest neighbours):
\begin{equation}
\label{eq:def_Laplacian_usual}
    \Delta h(i,j) = \frac{1}{4}\bigl(h(i+1,j)+h(i,j-1)+h(i-1,j)+h(i,j+1)\bigr)- h(i,j).
\end{equation}
Its nodal lines are given by the two axes, the diagonal $\{i-j=0\}$ and the anti-diagonal $\{i+j=0\}$, see Figure \ref{fig:nodal_lines}. The same function $h$ is also a positive harmonic function with Dirichlet boundary condition in the octant $\{(i,j)\in\mathbb Z^2 : 0\leq i\leq j\}$. What would be the general structure of these nodal lines?

\begin{figure}[ht]
\includegraphics[width=0.2\textwidth]{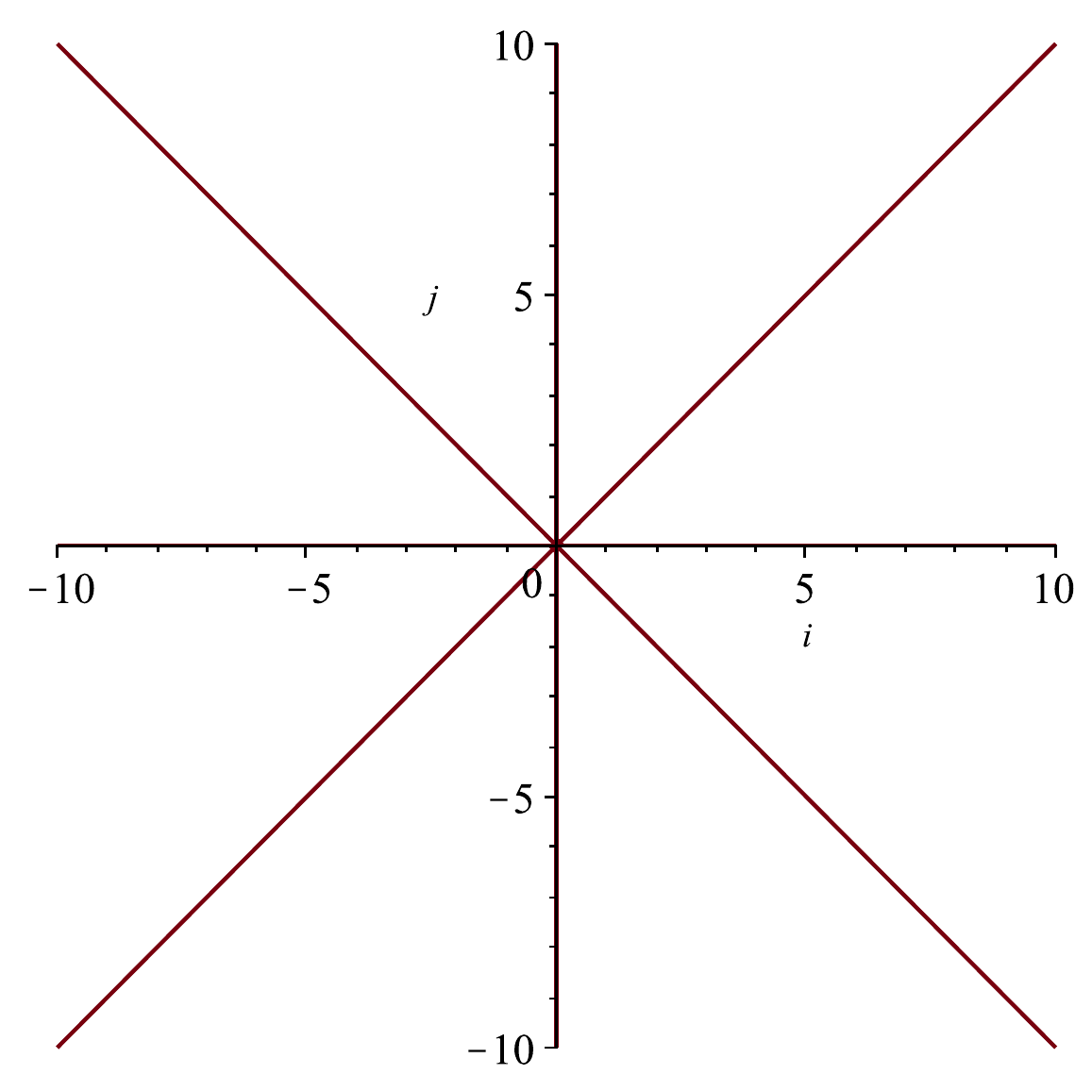}\qquad\qquad
\includegraphics[width=0.2\textwidth]{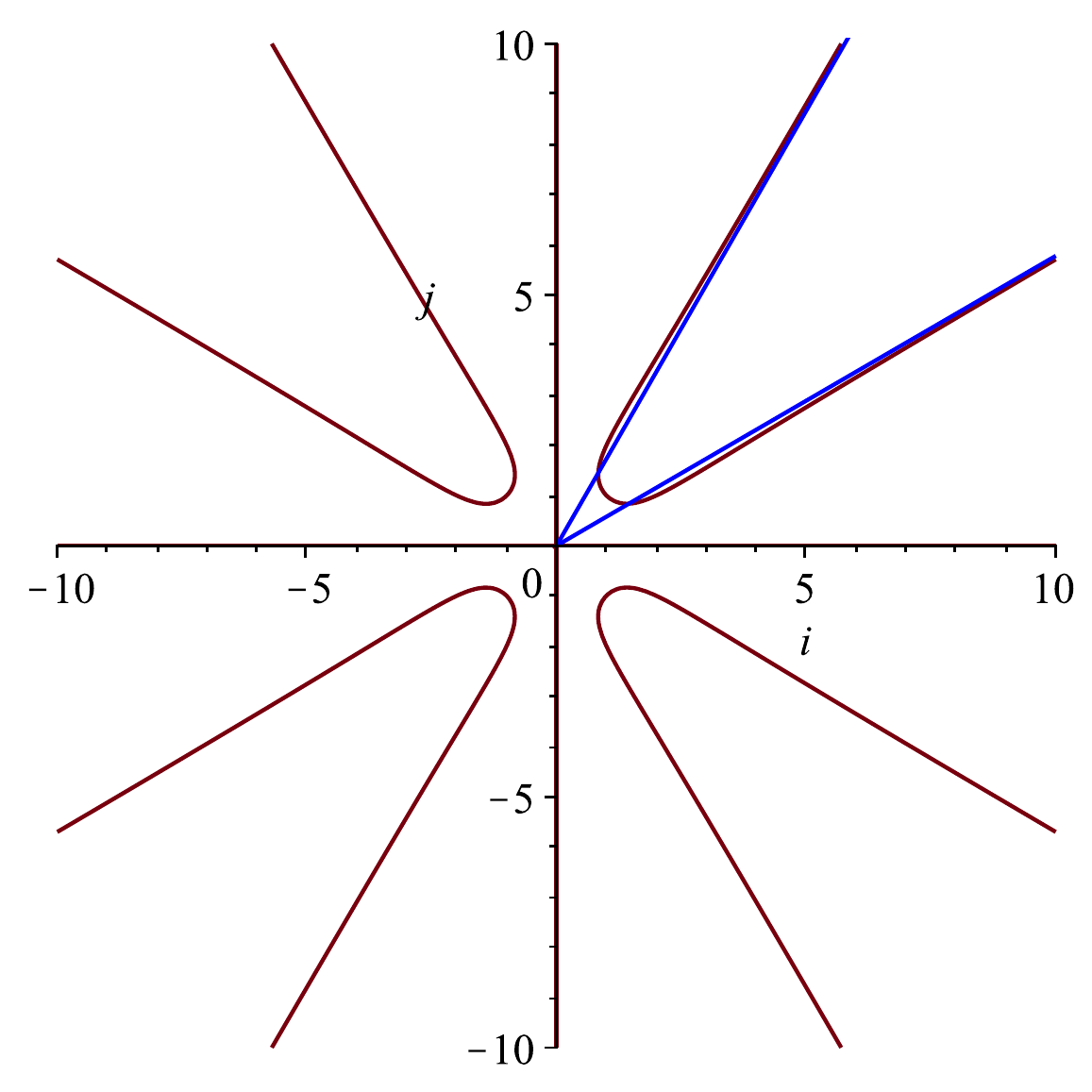}
\caption{Nodal domains associated to the functions $ij(i-j)(i+j)$ and $ij(3i^4-10i^2j^2+3j^4-5i^2-5j^2+14)$, which are discrete harmonic for the operator \eqref{eq:def_Laplacian_usual}. On the right picture, the blue cone is tangent to the nodal lines at infinity.}
\label{fig:nodal_lines}
\end{figure}

\subsection*{Preharmonic functions and their generating functions}

Let $\{p_{k,\ell}\}_{(k,\ell)\in\mathbb Z^2}$ be non-negative weights (or transition probabilities) summing to $1$, such that:
\begin{enumerate}[label=(H\arabic{*}),ref={\rm (H\arabic{*})}]
   \item\label{H1:jumps}The weights are symmetric, i.e., for all $k$ and $\ell$, $p_{k,\ell}=p_{\ell,k}$;
   \item\label{H2:jumps}Weights in the positive directions should be small, i.e., $p_{k,\ell}=0$ if $k\geq2$ or $\ell\geq2$, while weights in the negative directions may be arbitrary large, see Figure \ref{fig:step_sets};
   \item\label{H3:jumps}If $\big\vert \sum p_{k,\ell}x^ky^\ell \big\vert=1$ and $\vert x\vert=\vert y\vert=1$, then $x=y=1$ (equivalently, the random walk on $\mathbb Z^2$ with increment distribution given by the $p_{k,\ell}$ is irreducible);
   \item\label{H4:jumps}The $p_{k,\ell}$ admit moments of order $2$, i.e., $\sum (k^2+\ell^2) p_{k,\ell}<\infty$;
   \item\label{H5:jumps}The drift is zero, meaning that $\sum k p_{k,\ell}=\sum \ell p_{k,\ell}=0$.
\end{enumerate}
Consider now the associated discrete Laplacian operator \eqref{eq:def_Laplacian}, acting on complex-valued functions $h=\{h(i,j)\}_{(i,j)\in\mathbb Z^2}$. Our aim is to describe the functions $h$ which satisfy to:
\begin{enumerate}[label=(H\arabic{*}),ref={\rm (H\arabic{*})}]\setcounter{enumi}{5}
    \item\label{H1:harmonic}For all $(i,j)\in\mathbb Z^2$ with $i\leq0$ and/or $j\leq0$, $h(i,j)=0$;
    \item\label{H2:harmonic}For all $i,j\geq 1$, $\Delta h(i,j)=0$.
\end{enumerate}
Such functions are harmonic for the random walk on $\mathbb Z^2$ killed when exiting the quadrant.

The generating function of the harmonic function $h$ is
\begin{equation}
\label{eq:generating_functions_harmonic_functions}
     H(x,y)=  \sum_{i,j\geq 1} h(i,j) x^{i-1}y^{j-1},
\end{equation}
and its sections are
\begin{equation}
\label{eq:generating_functions_harmonic_functions_uni}
     H(x,0)= \sum_{i\geq 1} h(i,1) x^{i-1}
     \quad\text{and}\quad H(0,y)= \sum_{j\geq 1} h(1,j) y^{j-1}.
\end{equation}
Finally, the kernel is
\begin{equation}
\label{eq:def_K}
    K(x,y)= xy\left(1-\sum p_{k,\ell}x^{-k}y^{-\ell}\right) =xy-\sum p_{k,\ell}x^{-k+1}y^{-\ell+1}.
\end{equation}
The kernel is a bivariate power series due to our hypothesis \ref{H2:jumps} (even a polynomial if the jumps are bounded) and is obviously fully characterized by the jumps $\{p_{k,\ell}\}$. The function $H(x,y)$ satisfies the functional equation (which simply reflects the harmonicity relations)
\begin{equation}
\label{eq:functional_equation}
     K(x,y)H(x,y) = K(x,0) H(x,0) +K(0,y) H(0,y)-K(0,0) H(0,0).
\end{equation}

\subsection*{Main results}
Equation \eqref{eq:functional_equation} implies that any generating series $H(x,y)$ of a harmonic function has the form 
\begin{equation}
\label{eq:form_H}
   H(x,y)=\frac{F(x)+G(y)}{K(x,y)},
\end{equation}
for some power series $F,G\in\mathbb{C}[[t]]$ (here and throughout, given a field $\mathbb{K}$, $\mathbb{K}[[t]]$ will denote the set of power series in $t$ with coefficients in $\mathbb{K}$). On the other hand, not any power series $F$ and $G$ are such that the right-hand side of \eqref{eq:form_H} defines a bivariate power series (indeed, in the case $p_{1,1}=0$, notice that $K(0,0)=0$). From that point of view, we will answer the following questions:
\begin{itemize}
   \item For which power series $F$ and $G$ is the function $H$ in \eqref{eq:form_H} a power series? (Observe that in the case $p_{1,1}\neq 0$, this is always the case, as $K(0,0)\neq0$.)
   \item In case the function $H$ is a bivariate power series, is it analytic in a neighbourhood of $(0,0)$? What is the associated radius of convergence?
   \item Which choice of $F$ and $G$ guarantees that the generating function $H$ has positive coefficients?
\end{itemize}

In order to state our main result, we need to introduce a certain curve as well as two related conformal mappings. This step is the basis of our entire analysis and is inspired by the books \cite{CoBo-83,Co-92}. To do so, we first introduce the domain
\begin{equation}
\label{eq:main_domain}
   \mathcal{K}=\{(x,y)\in\mathbb C^2 : K(x,y)=0 \text{ and } \vert x\vert = \vert y\vert \leq 1\}. 
\end{equation}
As it turns out (more details are to come in Section \ref{sec:func_eq}), its projection along the first variable defines a curve $\mathcal S_1$ which is closed, non-intersecting, symmetric with respect to the real axis and contains $1$; see Figure \ref{fig:some_curves} for a few examples. The bounded (resp.\ unbounded) domain whose boundary is $\mathcal S_1$ is denoted by $\mathcal S_1^+$ (resp.\ $\mathcal S_1^-$). Let also $\mathcal{H}^+$ (resp.\ $\mathcal{H}^-$) denote the interior of the right (resp.\ left) half-plane. Let $\psi_{1}$ be a conformal mapping from $\mathcal S_1^+$ to $\mathcal{H}^+$, with the conditions $\psi_{1}(0)=p_{1,1}$ and $\psi_{1}'(0)>0$. Finally, introduce for any $n\geq 0$ the polynomial
\begin{equation}
\label{eq:def_P_n}
   P_{n}=\bigl(X^2-p_{1,1}^2\bigr)^{\lfloor n/2\rfloor}X^{n\,[2]},
\end{equation}
where $n\,[2]$ stands for $n$ modulo $2$. Obviously, the family $\{ P_{n}\}_{n\geq 0}$ is a basis of the set of real polynomials. The first few polynomials $P_n$ are:
\begin{equation*}
  P_0=1,\quad P_1=X,\quad P_2=X^2-p_{1,1}^2,\quad P_{3}=X^3-p_{1,1}^2X,\quad \text{etc.}
\end{equation*}

\begin{figure}
\begin{center}
\includegraphics[height=4.0cm]{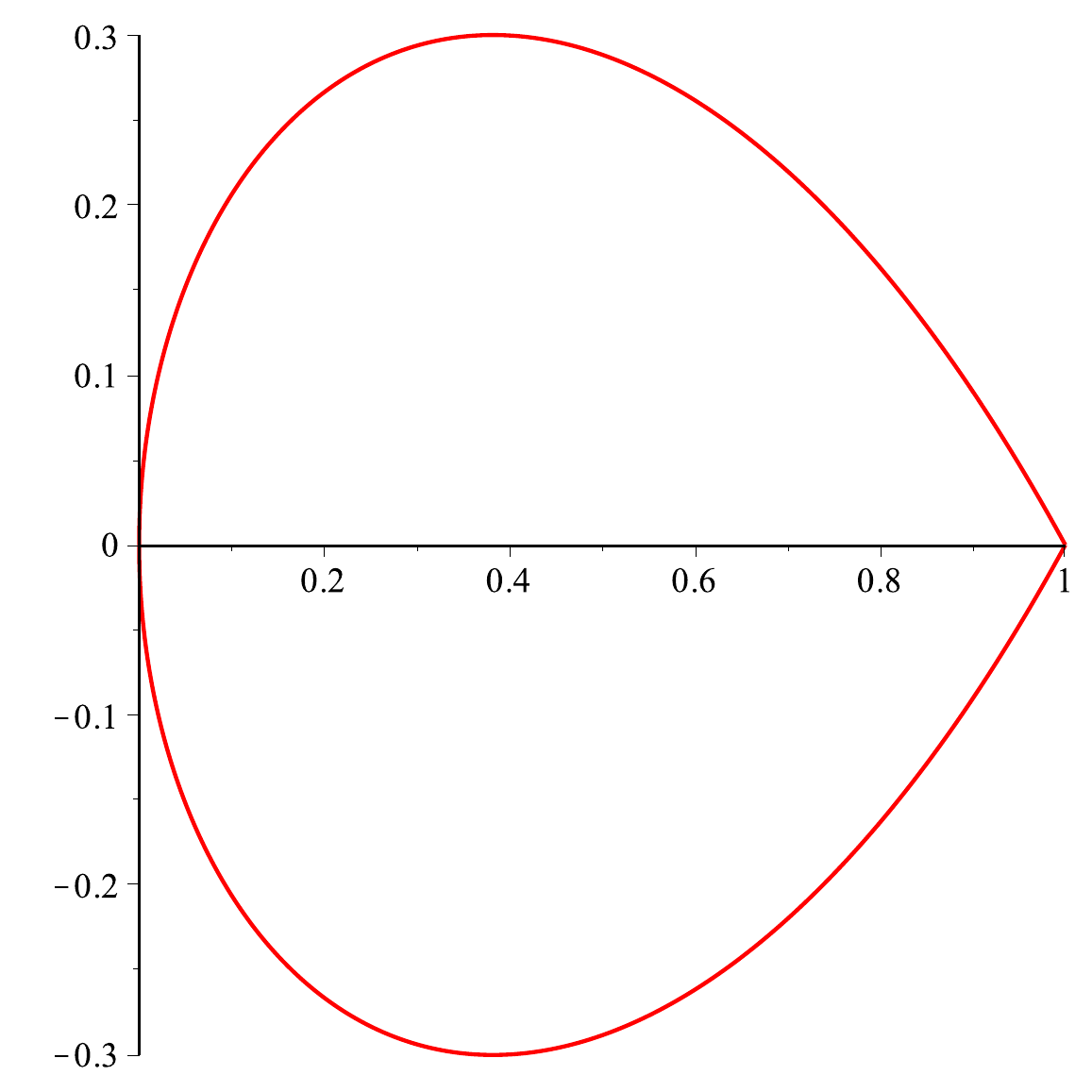}\quad
\includegraphics[height=4.0cm]{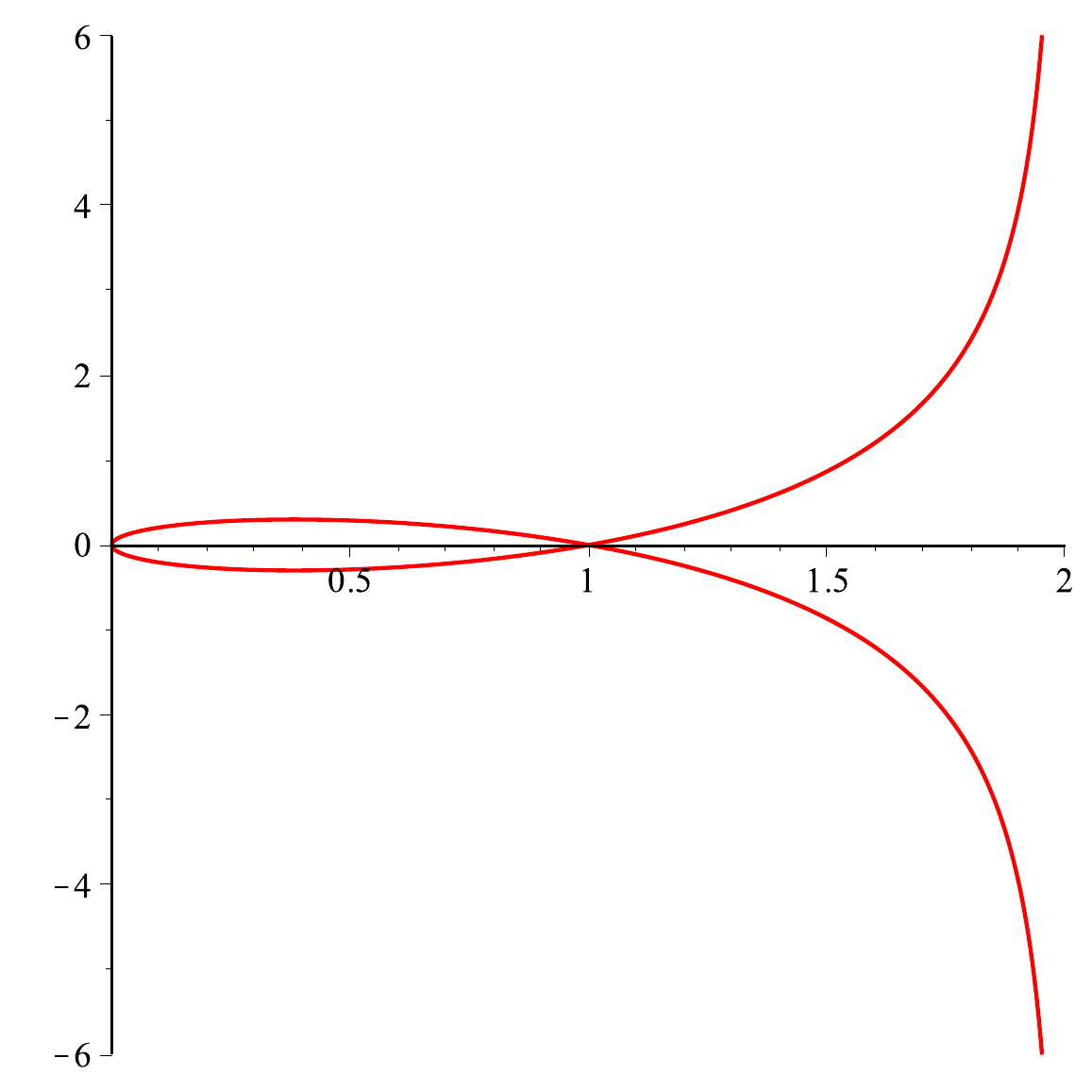}\quad
\includegraphics[height=4.0cm]{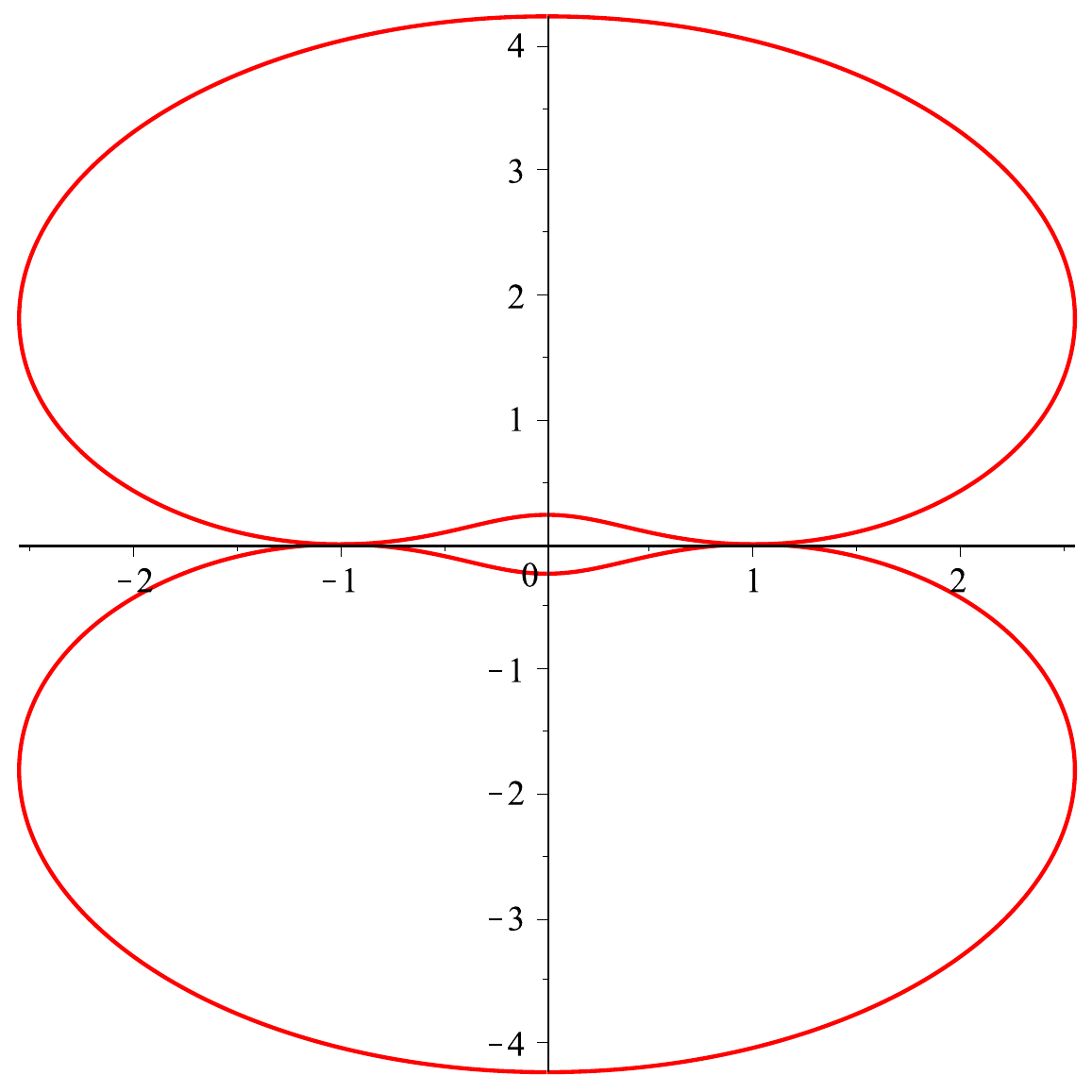}
\end{center}
\caption{The curve $\mathcal S_1$ for the simple random walk (left). It is included in the parametric curve \eqref{eq:expression_S1_SRW} with self-intersection (middle). On the right, the parametric curve for the (non-symmetric) walk with jumps $(-1,1),(1,0),(0,-1)$. The curve $\mathcal S_1$ is the intersection with the unit disk.}
\label{fig:some_curves}
\end{figure}

\begin{theorem}
\label{thm:main_intro-1}
For any $n\geq 1$, the function $\frac{P_n(\psi_{1}(x))-P_n(-\psi_{1}(y))}{K(x,y)}$ defines a bivariate power series
\begin{equation}
\label{eq:expression_H_n}
   H_{n}(x,y)=\sum_{i,j\geq 1}h_n(i,j) x^{i-1}y^{j-1}=\frac{P_{n}(\psi_{1}(x))-P_{n}(-\psi_{1}(y))}{K(x,y)},
\end{equation}
which satisfies the following properties:
\begin{enumerate}[label=\textnormal{(\arabic{*})},ref=\textnormal{(\arabic{*})}]
   \item\label{thm:main_intro-1:it1}Its Taylor coefficients $\{h_n(i,j)\}_{i,j\geq 1}$ form a discrete harmonic function for the Laplacian operator \eqref{eq:def_Laplacian}.
   \item\label{thm:main_intro-1:it2}$H_n$ is analytic for $x,y\in \mathcal{S}_1^+$.
   \item\label{thm:main_intro-1:it3}Denote by $h_{n}^{\sigma}$ the continuous harmonic function as in \eqref{eq:h_n^sigma}, with $\sigma$ being the covariance matrix \eqref{eq:covariance_matrix} associated to the transition probabilities $p_{k,\ell}$, and set 
   	\begin{equation}
	\label{eq:angle_at_1}
		\theta = \arccos\frac{-\sigma_{12}}{\sqrt{\sigma_{1}\sigma_{2}}}= \arccos\frac{-\sum k\ell p_{k,\ell}}{\sqrt{\sum k^2p_{k,\ell}}\sqrt{\sum \ell^2p_{k,\ell}}}.
	\end{equation}
For $x,y>0$, one has $\frac{1}{m^{n\pi/\theta+1}}h_{n}(\lfloor mx\rfloor ,\lfloor my\rfloor)\to  h_{n}^{\sigma}(x,y)$ as $m$ goes to infinity, in the sense of their Laplace transforms.
\end{enumerate}
\end{theorem}

Our second main result shows that the harmonic functions $h_n$ introduced in Theorem~\ref{thm:main_intro-1} actually describe the whole space of discrete harmonic functions. 
\begin{theorem}
\label{thm:main_intro-2}
The space of discrete harmonic functions is isomorphic to the vector space of formal power series $\mathbb{R}_{0}[[t]]$ with vanishing constant term, through the isomorphism 
\begin{equation*}
   \Phi:\left\lbrace\begin{matrix}
\mathbb{R}_0[[t]]&\longrightarrow&H(\mathbb N^2)\\
\sum_{n\geq 1}a_{n}t^n&\longmapsto& \sum_{n\geq 1}a_{n}h_n
\end{matrix}\right.
\end{equation*}
\end{theorem}

Various examples illustrating Theorems \ref{thm:main_intro-1} and \ref{thm:main_intro-2} will be given in Section \ref{sec:examples}. Let us now describe the main features of these theorems:
\begin{itemize}
   \item In the case $p_{1,1}= 0$, Theorems \ref{thm:main_intro-1} and \ref{thm:main_intro-2} imply that a function $\{h(i,j)\}_{i,j\geq 1}$ with generating function $H(x,y)$ is harmonic if and only if there exists a formal power series $F=\Phi^{-1}(h)$ such that
   \begin{equation}
   \label{eq:main-intro-2}
       H(x,y)=\frac{F(\psi_{1}(x))-F(-\psi_{1}(y))}{K(x,y)}.
   \end{equation}
   In the case $p_{1,1}\neq 0$, the expression of $H$ in terms of $\Phi^{-1}(h)$ is not so direct, and \eqref{eq:main-intro-2} should be replaced by \eqref{eq:form_H} for some functions $F$ and $G$ depending on $\Phi^{-1}(h)$ in a non-trivial way.
   
   The different behavior of the functional equation \eqref{eq:functional_equation} for a vanishing $p_{1,1}$ is not only of technical nature. Indeed, as we will see in Lemmas \ref{lemma:consequence_triangle_lemma} and \ref{lemma:consequence_rectangle_lemma}, specifying the values of a harmonic function on the horizontal axis $\{(i,1):i\geq 1\}$ uniquely determines the harmonic function in the $p_{1,1}=0$ case, whereas in the $p_{1,1}\not=0$ case, one also needs to give the values of the function on the vertical axis $\{(1,j):j\geq 1\}$ to fully characterize it.
   
   However, Theorem \ref{thm:main_intro-1} looks the same in both cases, because we choose to identify in the case $p_{1,1}\not=0$ a particular set of harmonic functions that behave like in the vanishing case (with respect to the assertion \ref{thm:main_intro-1:it3} of Theorem~\ref{thm:main_intro-1}, for example). Then, the use of power series in $h_n$ in Theorem \ref{thm:main_intro-2} allows us to reconstruct in the case $p_{1,1}\not=0$ all the harmonic functions from this particular set of harmonic functions.
   
   In the framework of this article, we say that a function $F$ characterizes a harmonic function $h$ if its generating function can be written as \eqref{eq:main-intro-2}.
   \item Part of the results of Theorem \ref{thm:main_intro-2} is that for any formal power series $F$ (even with a zero radius of convergence), the formula \eqref{eq:main-intro-2} defines formal power series (see our Lemmas \ref{lemma:triangle_lemma} and \ref{lemma:rectangle_lemma}) and eventually discrete harmonic functions.
   \item This correspondence between harmonic functions and formal power series may be refined as follows: writing $h(i,j)=\frac{h(i,j)+h(j,i)}{2}+\frac{h(i,j)-h(j,i)}{2}$, one may decompose a harmonic function as a sum of a symmetric harmonic function and another anti-symmetric harmonic function. Then in Theorem \ref{thm:main_intro-2}, symmetric (resp.\ anti-symmetric) harmonic functions correspond to even (resp.\ odd) power series $F(t)=\sum_{n\geq 1}a_{n}t^n$, see Section~\ref{sec:sym_antisym}.
   \item In Equation \eqref{eq:main-intro-2}, the only dependence in the model $\{p_{k,\ell}\}$ relies in the conformal mapping $\psi_{1}$. The fact that we may take any power series $F$ is independent of the model and therefore appears as a universal feature.
   \item It is known from \cite{BoMuSi-15,IR-20,DuRaTaWa-20} that there exists a unique function which is both positive and harmonic for the Laplacian operator \eqref{eq:def_Laplacian}, as soon as the walk admits moments of order high enough. We conjecture (and bring some evidence)\ that this unique positive harmonic function corresponds to $F(t)=t$ in \eqref{eq:main-intro-2}. In this sense, positivity of the harmonic function corresponds to minimality in terms of polynomial degree.

\begin{conj}
\label{conj:pos}
The unique (up to multiplicative factors) positive harmonic function associated to the Laplacian operator \ref{H1:jumps}--\ref{H5:jumps} is $h_1$ in Theorem \ref{thm:main_intro-1}.
\end{conj}
   \item The polynomial growth of the unique positive harmonic function can be read off on the corner point of the curve $\mathcal S_1$ at $1$ (see Figure \ref{fig:some_curves}), or equivalently on the singularity of the conformal mapping $\psi_{1}$ at $1$.
   \item In the article \cite{DeWa-15}, optimal moment conditions are given, under which one can fully describe the asymptotic behavior of the random walk in the quarter plane. The authors of \cite{DeWa-15} further show by a counter-example the optimality of their moment conditions: namely, one needs moments of order at least $\max\{2+\epsilon,\pi/\theta\}$, where $\epsilon >0$ and $\theta$ is as in \eqref{eq:angle_at_1}. In our paper, such moment conditions are not required: we only need a moment of order $1$ to impose a zero drift on the random walk and a moment of order $2$ to be able to speak about the covariance and the angle $\theta$. Remark that the latter condition is not even required if we are not interested in the Assertion \ref{thm:main_intro-1:it3} of Theorem \ref{thm:main_intro-1} (see Section \ref{subsec:example_less_moments} for an example without $2+\epsilon$-moment and one without $\pi/\theta$-moment).
   
   However, there are two main differences between our setting and the one of \cite{DeWa-15}: first, we only admit positive jumps of size $1$. This is a strong assumption, and counter-examples exhibited in \cite{DeWa-15} use positive jumps of unbounded size. Second, we only construct harmonic functions and do not describe any asymptotic behavior of the random walk. It would be very interesting to see whether this analytical approach allows to tackle asymptotic analysis with moment assumptions smaller than $2$.
   \item The vector space of harmonic functions with growth bounded by some constant $g$ is finite-dimensional, and amounts to taking $F\in \mathbb C_n[t]$ for some $n$ related to $g$.
   \item A much weaker statement of Theorem \ref{thm:main_intro-1} (small step random walks and only $F(t)=t$) is given in \cite{Ra-14}; see Section \ref{sec:uniformization} for a more detailed comparison between our approach and the one in \cite{Ra-14}.
   \item The conformal mapping $\psi_{1}$ is uniquely characterized by the Riemann mapping theorem. We show how to obtain its explicit expression in the small step case, as well as in a few models with larger steps. However, in general, despite the particular structure of the set \eqref{eq:main_domain}, there is little hope to derive concrete expressions for these conformal mappings.
   \item We will further explain how, given a (finite or infinite)\ number of boundary values, we may construct (and give a formula for)\ a harmonic function $h$ having these values.
   \item Interestingly, by Theorem \ref{thm:main_intro-2} one can transfer the algebra structure of $\mathbb{R}_0[[t]]$ into an algebra structure on the space of harmonic functions $H(\mathbb{N}^2)$. Formally, the multiplication is defined on $H(\mathbb{N}^2)$ by the formula 
\begin{equation*}
   h\cdot h'=\Phi\bigl(\Phi^{-1}(h)\Phi^{-1}(h')\bigr),\quad \forall h,h'\in H(\mathbb{N}^2).
\end{equation*}
This implies in particular that $h_n\cdot h_m=h_{n+m}$, and that $H(\mathbb{N}^2)$ is generated as an algebra by the (conjecturally) unique positive harmonic function $h_1$. 
\end{itemize}

\subsection*{Acknowledgments} We would like to thank \'Eric Fusy for enlightening discussions concerning Section~\ref{sec:SRW}. We further thank Gerold Alsmeyer, Onno Boxma and Fran\c cois Chapon for interesting discussions. We warmly thank an anonymous referee for extremely useful comments: she/he made us realize that our moment assumptions could be considerably generalized.

\section{Study of the kernel}
\label{sec:func_eq}

In this section, the main objective is the description of the set $\mathcal{K}$ introduced in \eqref{eq:main_domain}. To that purpose, in Section~\ref{subsec:preliminary}, we first recall from \cite{CoBo-83} some key properties of the two-variable function $K$ defined in \eqref{eq:def_K}. We also give a few elementary properties of the curve \eqref{eq:main_domain}. In Section~\ref{subsec:corner}, we prove a new result on the existence of a corner point of the curve \eqref{eq:main_domain} and compute the associated angle. This angle turns out to be very important, as we will prove later that it is strongly related to the growth of harmonic functions, as stated in Theorem \ref{thm:main_intro-1} \ref{thm:main_intro-1:it2}.

Throughout the manuscript, $\mathcal{C}$ will denote the unit circle and $\mathcal{C}^+$ (resp.\ $\mathcal{C}^-$, $\overline{\mathcal{C}^{+}}$) will stand for the interior (resp.\ exterior, resp.\ closed interior) domain, i.e., the open unit disk (resp.\ the complex plane deprived from the closed unit disk, resp.\ the closed unit disk). We shall also denote the bidisk by $\mathcal{U}=\mathcal{C}^+\times\mathcal{C}^+=\lbrace (x,y)\in \mathbb{C}^2:\vert x\vert< 1,\vert y\vert< 1\rbrace$.

\subsection{Preliminary results on the solutions of the kernel equation}
\label{subsec:preliminary}

To describe the domain $\mathcal K$ in \eqref{eq:main_domain}, remark that by \eqref{eq:def_K} we have
\begin{equation}
\label{eq:kernel_eq_eta_s}
    K(\eta s,\eta s^{-1}) = \eta ^2-\sum p_{k,\ell}\eta ^{-k-\ell+2}s^{-k+\ell}
\end{equation}
when $s\not=0$. For $(x,y)\in \mathcal{C^+}$ with $\vert x\vert=\vert y\vert$, we write
\begin{equation}
\label{eq:x_y_eta_s}
    x = \eta s\quad \text{and}\quad y = \eta s^{-1},
\end{equation}
with $\eta\in\mathbb{C}$ and $\vert s\vert=1$.

The lemma hereafter presents some properties of the roots of the kernel \eqref{eq:kernel_eq_eta_s}. Remark that in Lemma \ref{lemma:kernel} below, and in many other statements as well, two cases will be considered separately, according to whether $p_{1,1}=0$ or not.

\begin{lemma}
\label{lemma:kernel}
	Assume \ref{H1:jumps}--\ref{H5:jumps}. For $\vert s\vert=1$, the equation $K(\eta s ,\eta s^{-1})=0$ admits exactly two solutions in $\overline{\mathcal{C}^+}$. Moreover, the following assertions hold:
	\begin{enumerate}[label=\textnormal{(\roman{*})},ref=\textnormal{(\roman{*})}]
		\item\label{item:kernel_p11=0}If $p_{1,1}= 0$, one solution is identically zero and the other one, denoted by $\eta(s)$, is real and varies in $[-1,1]$. Furthermore, $\eta(1)=1$, $\eta(-1)=-1$ and $\vert\eta(s)\vert<1$ for all $\vert s\vert=1$ with $s\neq \pm 1$.
		\item\label{item:kernel_p11_not=0} If $p_{1,1}\neq 0$, the two solutions, denoted by $\eta_1(s)$ and $\eta_2(s)$, are real and satisfy, for all $\vert s\vert =1$, $\eta_2(s) = -\eta_1(-s)$. Further, $\eta_1(1)=1$ and $\eta_1(s)\in(0,1)$ for all $\vert s\vert=1$ with $s\neq 1$.
	\end{enumerate}
\end{lemma}

Although the proof of Lemma \ref{lemma:kernel} may be found in the book \cite{CoBo-83} (see in particular \cite[Lem.~2.1]{CoBo-83}), we recall here the details, for the sake of completeness.  
\begin{proof}
	We first present the proof of \ref{item:kernel_p11=0}. When $p_{1,1}=0$, a term $\eta$ may be factorised out of the kernel \eqref{eq:kernel_eq_eta_s}, so one of the roots is identically equal to zero. More precisely, we have $K(\eta s,\eta s^{-1})=\eta \widetilde{K}(\eta,s)$, where
	\begin{equation}
	\label{eq:K-tilde}
   	\widetilde{K}(\eta,s)=\eta -\sum p_{k,\ell}\eta^{-k-\ell+1}s^{-k+\ell}.
	\end{equation}
	By Assumption \ref{H3:jumps}, one knows that if $(\eta s,\eta s^{-1})$ is a root of ${K}$ with $\vert\eta\vert=\vert s\vert =1$, then $\eta=s=1$ or $\eta=s=-1$. Therefore, for all $\vert s\vert =1$ with $s\neq \pm 1$, one has the following strict inequality:
	\begin{equation}
	\label{eq:before_Rouch\'e}
	\left\vert\sum  p_{k,\ell} \eta^{-k-\ell+1} s^{-k+\ell}\right\vert < 1 = \vert \eta \vert.
	\end{equation}
	It readily follows from Rouch\'e's theorem (we use its statement as in \cite[Thm~6.2]{Ev-66}) that, as a function of $\eta$, $\widetilde{K}(\eta,s)$ has a unique root in $\mathcal{C}^+$, denoted by $\eta(s)$. Since $\widetilde{K}(1,s)\geq 0$ and $\widetilde{K}(-1,s)\leq 0$, then $\eta(s)\in[-1,1]$, hence the root is real.

	The previous argument does not work for $s=\pm 1$, because \eqref{eq:before_Rouch\'e} becomes an equality at this point. We now consider $s=1$. Fix $r\in (0,1)$ and consider the series
	\begin{equation*}
    \widetilde{K}_r(\eta)=\eta -r\sum p_{k,\ell}\eta^{-k-\ell+1}.
	\end{equation*}
	Rouch\'e's theorem yields that $\widetilde{K}_r$ has a unique root $\eta_r(1)\in\mathcal{C}^+$. Further, Assumption~\ref{H5:jumps} ensures that $\widetilde{K}_r$ is increasing on $[0,1]$, with $\widetilde{K}_r(0)=-(p_{1,0}+p_{0,1})<0$ and $\widetilde{K}_r(1)=1-r>0$. Moreover, $\eta_r(1)\in(0,1)$ and converges to $1$ as $r\to 1$. By continuity arguments, one has $\eta(1)=1$. At the point $s=-1$, notice that for all $\eta$ and $s$, one has $\widetilde{K}(-\eta,-s) = -\widetilde{K}(\eta,s)$. This implies that $\eta(-1)=-1$.

	We move to the proof of \ref{item:kernel_p11_not=0}. Fix $\vert s\vert=1$ with $s\neq 1$. Since for all $\vert\eta\vert=1$, one has
	\begin{equation*}
   	\left\vert\sum p_{k,\ell} \eta^{-k-\ell+2} s^{-k+\ell} \right\vert < 1 = \vert \eta^2 \vert,
	\end{equation*}
	then by Rouch\'e's theorem, $K(\eta s,\eta s^{-1})$ admits two roots $\eta_1(s)$ and $\eta_2(s)$ in $\mathcal{C}^+$. These quantities are real, as $K(\eta s,\eta s^{-1})$ is non-negative at $1$ and $-1$, but negative at $0$. So one of them (say $\eta_1(s)$) is positive, and the other one ($\eta_2(s)$) is negative.

	Looking now at the point $s=1$, we set for $r\in(0,1)$
	\begin{equation*}
	K_r(\eta) = \eta^2 - r\sum p_{k,\ell}\eta^{-k-\ell+2}.	
	\end{equation*}
	Rouch\'e's theorem implies that $K_r$ has two roots in $\mathcal{C}^+$, say $\eta_{1,r}(1)$ and $\eta_{2,r}(1)$. Notice that $K_r$ is positive at $\pm 1$, negative at $0$, so $\eta_{1,r}(1)\in (0,1)$ and $\eta_{2,r}(1)\in (-1,0)$.

	Since $K_r'$ is concave on $[0,1]$, with $K_r'(0)=-r(p_{1,0}+p_{0,1})\leq 0$ and $K_r'(1)=2(1-r)>0$, there exists $\eta_0\in[0,1)$ such that $K_r'(\eta_0)=0$. Hence, $K_r$ is decreasing on $[0,\eta_0]$, increasing on $[\eta_0,1]$, and $\eta_{1,r}(1)\in (\eta_0,1)$. As $r\to 1$, then $\eta_{1,r}(1)\to 1$, while $\eta_{2,r}(1)$ goes to a point in $(-1,0)$. Hence, $K_1$ has two solutions in $\overline{\mathcal{C}^+}$, which are $\eta_1(1)=1$ and $\eta_2(1)\in (-1,0)$.

	To conclude, we prove that for all $\vert s\vert=1$, $\eta_2(s)=-\eta_1(-s)$. Equation $K(\eta s,\eta s^{-1})=0$ may be rewritten as
	\begin{equation}
	\label{eq:quadratic_eq_kernel}
    \left(1-\sum_{k+\ell\leq 0}p_{k,\ell}\eta^{-k-\ell}s^{-k+\ell}\right)\eta^2-(p_{1,0}s^{-1}+p_{0,1}s)\eta - p_{1,1}=0.
	\end{equation}
	Notice that if $\eta(s)$ is a solution of \eqref{eq:quadratic_eq_kernel}, then $-\eta(-s)$ is the other solution. Hence, our claim follows.
\end{proof}

It is worth mentioning that in the case $p_{1,1}\neq 0$, the domain \eqref{eq:main_domain} does not depend on the choice of $\eta_1(s)$ and $\eta_2(s)$. Hereafter, we shall write $\eta(s)$ for $\eta_1(s)$ in the case $p_{1,1}\not=0$. Based on Lemma \ref{lemma:kernel} and using the notation just introduced, we may now properly define the domain \eqref{eq:main_domain}:
\begin{equation*}
     \mathcal{K}=\{ (\eta(s)s,\eta(s)s^{-1}):\vert s\vert=1 \}=\{ (\eta(s)s,\overline{\eta(s)s}):\vert s\vert=1 \},
\end{equation*}
with the bar denoting the complex conjugation. Introduce the curves
\begin{equation*}
   \mathcal{S}_1=\{\eta(s)s : \vert s\vert=1\} \quad \text{and}\quad \mathcal{S}_2=\{\eta(s)s^{-1}:\vert s\vert=1\},
\end{equation*}
which are obtained by taking the projection of $\mathcal{K}$ along the first and second coordinates, respectively. See Figure \ref{fig:some_curves} for a few examples. Obviously, an equivalent description of $\mathcal{S}_1$ (as a parametrized curve) is
\begin{equation*}
   \mathcal{S}_1=\{(\eta(e^{it})\cos t,\eta(e^{it})\sin t) : t\in[0,2\pi)\},
\end{equation*}
where we recall that $\eta(e^{it})$ is real, and that the interval $[0,2\pi)$ may be replaced by $[0,\pi)$ in the case $p_{1,1}= 0$. For example, in the case of the simple random walk, one will have
\begin{equation}
\label{eq:expression_S1_SRW}
   \mathcal{S}_1 = \{(1-\sin t,(1-\sin t)\tan t): t\in [0,\pi)\},
\end{equation}
see Figures \ref{fig:some_curves} and \ref{fig:def_theta}. We also emphasize that if the support of the $p_{k,\ell}$ is bounded, then the function $\eta(s)$ is an algebraic function of $s$.

\begin{lemma}
	\label{lemma: properties of S1 S2'}
	Assume \ref{H1:jumps}--\ref{H5:jumps}. The following assertions hold:
	\begin{enumerate}[label=\textnormal{(\roman{*})},ref=\textnormal{(\roman{*})}]
	\item\label{it:S1-i}If $p_{1,1}=0$, then $0\in\mathcal{S}_1$ and for all $\vert s\vert=1$, $\eta(-\overline{s})=-\eta(s)$. The maps $s\mapsto\eta(s)s$ and $s\mapsto \eta(s)s^{-1}$ are two-to-one from $\mathcal{C}$ to $\mathcal{S}_1$ and $\mathcal{S}_2$, respectively.
	\item\label{it:S1-ii}If $p_{1,1}\neq 0$, then $0\in\mathcal{S}_1^+$ and for all $\vert s\vert=1$, $\eta(s)>0$ and $\eta(\overline{s})=\eta(s)$. The maps $s\mapsto\eta(s)s$ and $s\mapsto \eta(s)s^{-1}$ are one-to-one from $\mathcal{C}$ to $\mathcal{S}_1$ and $\mathcal{S}_2$, respectively. 
	\item\label{it:S1-iii}The curves $\mathcal{S}_1$ and $\mathcal{S}_2$ are equal, non-self-intersecting and symmetric with respect to the real axis. If $s$ traverses $\mathcal{C}$ counterclockwise, then $\eta(s)s$ traverses $\mathcal{S}_1$ counterclockwise and $\eta(s)s^{-1}$ traverses $\mathcal{S}_2$ clockwise.
	\end{enumerate}
\end{lemma}

\begin{proof}
	Item \ref{it:S1-iii} is a direct consequence of \ref{it:S1-i} and \ref{it:S1-ii}, and we start with the proof of \ref{it:S1-i}. By Assumption \ref{H1:jumps} and our notation \eqref{eq:K-tilde}, we have for $s={e}^{{i}\lambda}$:
	\begin{equation*}
	\widetilde{K}(\eta, {e}^{{i}\lambda}) = \eta-\sum p_{k,\ell}\eta^{-k-\ell+1}\cos((-k+\ell)\lambda).
	\end{equation*}
	It is seen that if $\eta=\eta({e}^{{i}\lambda})$ is a root of the above function, then so is $-\eta(-{e}^{{-i}\lambda})$. By uniqueness of the solution, we must have $\eta(s)=-\eta(-\overline{s})$.

	Moving to the proof of \ref{it:S1-ii}, we first have
	\begin{equation}
	\label{eq:formula_K_ii}
	{K}(\eta{e}^{{i}\lambda},\eta{e}^{{-i}\lambda})= \eta^2-\sum p_{k,\ell}\eta^{-k-\ell+2}\cos((-k+\ell)\lambda),
	\end{equation}
	see \eqref{eq:kernel_eq_eta_s}. It is seen that if $\eta({e}^{{i}\lambda})$ is a root of \eqref{eq:formula_K_ii}, then so is $\eta({e}^{{-i}\lambda})$. Hence, $\eta(s)=\eta(\overline{s})$, being both positive (see Lemma \ref{lemma:kernel} and its proof).
\end{proof}

\subsection{Corner point of the curve $\mathcal{S}_1$ at $1$}
\label{subsec:corner}

Since it will be crucial in the next sections, we need to study the precise shape of the curve at the point $1$. It should be noted that this single point $1$ contains all the information about the growth of harmonic functions. 

Throughout the paper, we write $V[K]$ for the set of zeros of $K$ in $\mathbb C^2$. We also recall here that a point of $V[K]$ is singular if and only if $\frac{\partial K}{\partial x}$ and $\frac{\partial K}{\partial y}$ simultaneously vanish at that point.
\begin{lemma}
	\label{lemma:angle_at_1}
	Assume \ref{H1:jumps}--\ref{H5:jumps}. The curve $\mathcal{S}_1$ admits a corner point at $1$ with interior angle $\theta$ given by \eqref{eq:angle_at_1}, and is smooth elsewhere. Moreover, $\mathcal{K}\setminus\{(1,1)\}$ consists of non-singular points of $V[K]$. 
\end{lemma}
In other words, the value of the angle at the corner point is given by the arccosine of the opposite of the correlation coefficient.  See Figure~\ref{fig:def_theta} for an illustration of Lemma \ref{lemma:angle_at_1}.

\begin{figure}
\hspace{-10mm}
\begin{tikzpicture}[scale=2.5]
\draw[gray,very thin] (-1.1,-1.1) grid (1.1,1.1)
	 [step=0.25cm] (-1,-1) grid (1,1);
\draw[->] (-1.1,0) -- (1.1,0)  node[right] {$1$};
\draw[->] (0,-1.1) -- (0,1.1)  node[above] {$1$};

  \textcolor{red}{\qbezier(59,-10)(47,0)(59,10)}

\draw[thick,variable=\t,domain=0:180,samples=50,blue]
  plot ({1-sin(\t)},{tan(\t)-sin(\t)*tan(\t)});
  

\put(45,5){$\theta$}  

\put(30,25){$\mathcal S_1$} 
  
\end{tikzpicture}
\begin{tikzpicture}[scale=2.5]
\draw[gray,very thin] (-1.1,-1.1) grid (1.1,1.1)
	 [step=0.25cm] (-1,-1) grid (1,1);
\draw[->] (-1.1,0) -- (1.1,0)  node[right] {$1$};
\draw[->] (0,-1.1) -- (0,1.1)  node[above] {$1$};

  \textcolor{red}{\qbezier(59,-10)(47,0)(59,10)}

  \draw[thick,variable=\t,domain=0:180,samples=50,blue]
  plot ({(-cos(\t)+sqrt(12-3*cos(\t)*cos(\t))-sqrt((-cos(\t)+sqrt(12-3*cos(\t)*cos(\t)))*(-cos(\t)+sqrt(12-3*cos(\t)*cos(\t)))-4))*cos(\t)/2},{(-cos(\t)+sqrt(12-3*cos(\t)*cos(\t))-sqrt((-cos(\t)+sqrt(12-3*cos(\t)*cos(\t)))*(-cos(\t)+sqrt(12-3*cos(\t)*cos(\t)))-4))*sin(\t)/2});
  \draw[thick,variable=\t,domain=0:180,samples=50,blue]
  plot ({(-cos(\t)+sqrt(12-3*cos(\t)*cos(\t))-sqrt((-cos(\t)+sqrt(12-3*cos(\t)*cos(\t)))*(-cos(\t)+sqrt(12-3*cos(\t)*cos(\t)))-4))*cos(\t)/2},{-(-cos(\t)+sqrt(12-3*cos(\t)*cos(\t))-sqrt((-cos(\t)+sqrt(12-3*cos(\t)*cos(\t)))*(-cos(\t)+sqrt(12-3*cos(\t)*cos(\t)))-4))*sin(\t)/2});
  

\put(45,5){$\theta$}  

\put(30,25){$\mathcal S_1$} 
  
\end{tikzpicture}
\caption{The curve $\mathcal S_1$ for the simple walk (left) and for the king walk (right), and the definition of the angle $\theta$. For the two models above, $\theta=\frac{\pi}{2}$.}
\label{fig:def_theta}
\end{figure}
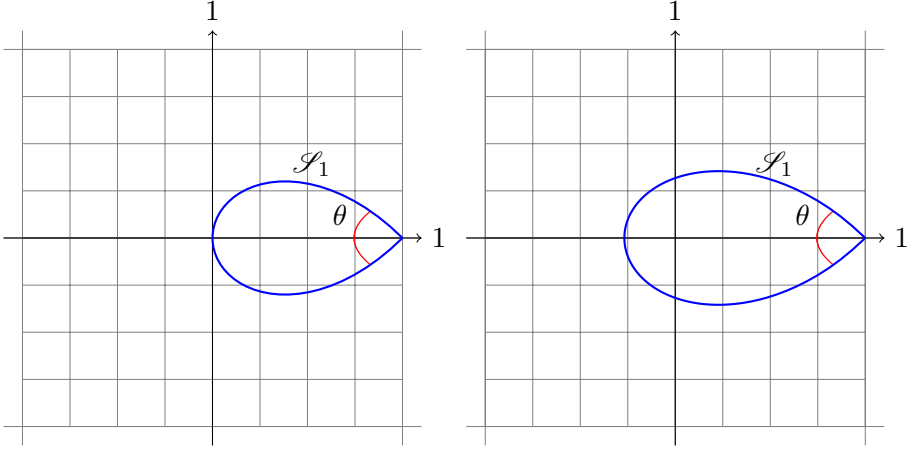

\begin{proof}
	We start with the case $p_{1,1}=0$ and first prove that $\mathcal{K}\setminus\{(1,1)\}$ consists of smooth points of $V[K]$. Suppose that $\widetilde{K}(\eta_0,s_0)=0$, with $\vert\eta_0\vert<1$, $\vert s_0\vert=1$ and $s_0\not=\pm 1$. Then, by Weierstrass preparation theorem for analytic functions in several variables (whose statement is recalled in Appendix~\ref{sec:app_W}, see also \cite{Gu-65}), there exist $r\geq 1$, a neighbourhood $V$ around $(\eta_0,s_0)$ (whose projection along the second coordinate is denoted by $P_s(V)$), $r$ analytic functions $g_0,\ldots,g_{r-1}$ on $P_s(V)$ vanishing only at $s_0$ and a non-vanishing analytic function $h$ on $V$ such that
	\begin{equation*}
    \widetilde{K}(\eta,s)=h(\eta,s)\bigl((\eta-\eta_0)^r+g_{r-1}(s)(\eta-\eta_0)^{r-1}+\cdots+g_{0}(s)\bigr).
	\end{equation*}
	Hence, for all $s$ close to but different from $s_0$, there are $r$ distinct solutions to the equation $\widetilde{K}(\eta,s)=0$. Since for $s\in\mathcal{C}$ in a neighbourhood of $s_0$, there is a unique solution to the equation $\widetilde{K}(\eta,s)=0$, we must have $r=1$. In particular,
	\begin{equation*}
    \frac{\partial\widetilde{K}}{\partial \eta}(\eta_0,s_0)=h(\eta_0,s_0)\not=0,
	\end{equation*}
	and for $\eta_0\neq 0$,
	\begin{equation*}
    \frac{\partial K}{\partial \eta}(\eta_0s_0,\eta_0s_0^{-1})=\eta_0\frac{\partial\widetilde{K}}{\partial \eta}(\eta_0,s_0) \neq 0.
	\end{equation*}
	Since 
	\begin{equation*}
    \frac{\partial K}{\partial \eta}=\frac{\partial x}{\partial \eta}\frac{\partial K}{\partial x}+\frac{\partial y}{\partial \eta}\frac{\partial K}{\partial y} = s\frac{\partial K}{\partial x}+s^{-1}\frac{\partial K}{\partial y},
	\end{equation*} 
	we must have either $\frac{\partial K}{\partial x}$ or $\frac{\partial K}{\partial y}$ non-zero on $\mathcal{K}\setminus \{(1,1)\}$. For $\eta_0=0$, it is easily seen that $\frac{\partial K}{\partial x}(0,0)\neq 0$. The claim then follows.

	We move to the proof of the smoothness of the curve $\mathcal{S}_1\setminus\{1\}$. Differentiating the identity $\widetilde{K}(\eta(s),s)=0$, see \eqref{eq:K-tilde}, we obtain
	\begin{equation}
	\label{eq:derivative_kernel_1}
    \eta'(s)\frac{\partial\widetilde{K}}{\partial\eta}(\eta(s),s) - \frac{\partial\widetilde{K}}{\partial s}(\eta(s),s) = 0.
	\end{equation}
	It is seen that for $\vert s\vert =1$ and $s\neq \pm 1$, $\eta'(s)$ exists and is finite. Now we set $s={e}^{{i}\lambda}$, with $\lambda\in (0,\pi)$. Since both $\lambda$ and $\eta({e}^{{i}\lambda})$ are real, then so is $(\eta({e}^{{i}\lambda}))'$. Moreover, $\eta({e}^{{i}\lambda})$ and $(\eta({e}^{{i}\lambda}))'$ are not simultaneously zero, which leads to
	\begin{equation*}
    (\eta({e}^{{i}\lambda}){e}^{{i}\lambda})'= {e}^{{i}\lambda}\bigl( (\eta({e}^{{i}\lambda}))' + {i}\eta({e}^{{i}\lambda}) \bigr)\neq 0.
	\end{equation*}
	In other words, $\mathcal{S}_1\setminus\{1\}$ is smooth.

	We now deal with the point $(1,1)$. We differentiate again \eqref{eq:derivative_kernel_1} and evaluate the new equation at $s=1$, which leads to
	\begin{equation}
	\label{eq:derivative_kernel_3}
	\left(\sum p_{k,\ell}(k+\ell)^2\right)\eta'(1)^2 + \sum p_{k,\ell}(k-\ell)^2 =0.
	\end{equation}
	Since the equation \eqref{eq:derivative_kernel_3} has two distinct solutions, then $\eta(s)$ is semi-differentiable at $s=1$. Further, the solutions of \eqref{eq:derivative_kernel_3} represent the left and right derivatives of $\eta$ at $1$. Let $\partial_+\eta(1)$ denote the right derivative of $\eta$ at $1$, i.e.,
	\begin{equation*}
	\partial_+\eta(1) = \lim_{\lambda\to 0^+} \frac{\eta({e}^{{i}\lambda})-\eta(1)}{{e}^{{i}\lambda}-1}.
	\end{equation*}
	From Lemma \ref{lemma: properties of S1 S2'}, we know that $\eta({e}^{{i}\lambda})\to 1^-$ as $\lambda\to 0^+$. Hence, by \eqref{eq:derivative_kernel_3}, one has
	\begin{equation*}
	\partial_+\eta(1) = {i}\sqrt{\frac{\sum p_{k,\ell}(k-\ell)^2}{\sum p_{k,\ell}(k+\ell)^2}}.
	\end{equation*}
	We then have
	\begin{equation*}
	\bigl(\partial_+\eta(s)s\bigr)\vert_{s=1}=\partial_+\eta(1) + \eta(1) = 1+{i}\sqrt{\frac{\sum p_{k,\ell}(k-\ell)^2}{\sum p_{k,\ell}(k+\ell)^2}}.
	\end{equation*}
	This allows us to derive the interior angle $\theta$ of $\mathcal{S}_1$ at $1$:
	\begin{equation*}
	\cos \theta = 2\sin^2\bigl((\arg \partial_+\eta(s)s)\vert_{s=1}\bigr)-1=\frac{-\sum k\ell p_{k,\ell}}{\sqrt{\sum k^2p_{k,\ell}}\sqrt{\sum \ell^2p_{k,\ell}}}.
	\end{equation*}    
	This completes the proof in the case $p_{1,1}=0$.

	In the case $p_{1,1}\neq 0$, one may apply the same arguments directly on $K$ (instead of $\widetilde{K}$) and derive the smoothness of the curves. Differentiating the identity $K(\eta(s)s,\eta(s)s^{-1})=0$, we obtain	
	\begin{multline*}
		\eta'(s) \left(2\eta(s)- \sum p_{k,\ell}(-k-\ell+2)\eta(s)^{-k-\ell+1}s^{-k+\ell}\right) \\
		- \left(\sum p_{k,\ell}(-k+\ell)\eta(s)^{-k-\ell+2}s^{-k+\ell-1}\right)=0.
	\end{multline*}
	Equation \eqref{eq:derivative_kernel_3} and the rest of the proof follow in a similar way.
\end{proof}

\subsection{Non-existence of solutions to the kernel equation in $\mathcal{S}_1^+\times\mathcal{S}_1^+$}

In this part, we show that $K(x,y)$ does not have any zero in the domain $\mathcal{S}_1^+\times\mathcal{S}_1^+$. This will allow us to deduce some regularity properties of $H(x,y)$ in the same domain, which will be used importantly in the next sections.

\begin{lemma}
	\label{lemma:no_solution}
	Assume \ref{H1:jumps}--\ref{H5:jumps}.
	If $(x,y)\in V[K]$ are such that $\vert y\vert<\vert x\vert <1$ (resp.\ $\vert x\vert<\vert y\vert <1$), then $x\notin \mathcal{S}_1^+$ (resp.\ $y\notin \mathcal{S}_1^+$). As a consequence, $K(x,y)$ does not have any root in $\mathcal{S}_1^+\times\mathcal{S}_1^+$.
\end{lemma}

\begin{proof}
	We first consider the case $p_{1,1}\not=0$ and set
	\begin{equation*}
	   \mathcal{V}_x=\{(x,y)\in V[K]: \vert y\vert<\vert x\vert<1\}.
	\end{equation*}
	Reasoning by contradiction, let us suppose that there exists $(x,y)\in\mathcal{V}_x$ with $x\in \mathcal S_1^+$. Let $\gamma$ be a path from $x$ to $0$ which avoids $\mathcal S_1$ and any point of $V[K]$ such that $\frac{\partial K}{\partial x}$ vanishes (there are only finitely many such points). Then, there exists a path $\gamma':[0,1]\rightarrow V[K]$ such that $\gamma'(0)=(x,y)$ and $P_{x}\circ \gamma'=\gamma$, with $P_x$ denoting the projection along the first variable. Hence, since $p_{1,1}\not=0$, $\gamma'(1)=(0,y')$ for some $y'\in\mathbb{C}\setminus \{0\}$. Since $P_x\circ\gamma'=\gamma$ never meets $\mathcal S_1$, $\gamma'$ never meets $\mathcal{K}$ and thus $\gamma'([0,1])\subset \mathcal{V}_x$. We should thus have $\vert y'\vert<0$, which is a contradiction. 

	We move to the case $p_{1,1}=0$, for which $(0,0)\in \mathcal{K}$. By Weierstrass preparation theorem (recalled here in Appendix~\ref{sec:app_W}) and coefficient identifications, we have 
	\begin{equation*}
   	K(x,y)=h(x,y)\bigl(y+f(x)\bigr),
	\end{equation*}
	for $(x,y)$ in a neighbourhood $V$ of $(0,0)$, with 
	\begin{equation*}
   	f(x)=x+\frac{1+2p_{2,0}-p_{0,1}}{p_{1,0}}x^2+o(x^2).
	\end{equation*}
	Since $\frac{1+2p_{2,0}-p_{0,1}}{p_{1,0}}>0$, for $x$ small enough with $x>0$, there is a unique $y\in\mathbb{C}$ such that $(x,y)\in V[K]\cap V$, and moreover for $x$ small enough
	\begin{equation*}
   \vert y\vert=\vert f(x)\vert>\vert x\vert.
	\end{equation*}
	Hence, suppose that $(x,y)\in\mathcal{V}_x$ with $x\in \mathcal S_1^+$, and let $\gamma$ be a path from $x$ to $0$ in $\mathcal S_1^+$ avoiding any singular point of $V[K]$ and such that $\gamma(t)$ is real for $t$ close to $1$. Then, there exists a path $\gamma'(x)$ such that $P_{x}\circ \gamma'=\gamma$. By the previous reasoning, for $t$ close to $1$ we have $\vert P_y\circ\gamma'(t)\vert>\vert P_x\circ\gamma'(t)\vert$. Since $\gamma'$ is continuous, and since for $t$ close to $0$ we have $\vert P_y\circ\gamma'(t)\vert<\vert P_x\circ\gamma'(t)\vert$, there exists $t\in[0,1)$ such that $\gamma'(t)\in\mathcal{K}$. This contradicts the fact that $P_x\circ\gamma'$ avoids $\mathcal S_1$.
\end{proof}

\section{Boundary value problems for the generating functions}
\label{sec:BVP}

The main objective of this section is to prove Proposition \ref{prop:BVP}, which states a BVP for the sectional generating functions $H(x,0)$ and $H(0,y)$ defined in \eqref{eq:generating_functions_harmonic_functions_uni}. The polynomial solutions to this BVP will be analyzed in Corollary \ref{cor:solutions-poly_BVP}. We first need to introduce conformal mappings for the bounded domain delimitated by $\mathcal S_1$.

We begin with introducing some notation. Let $\mathcal A$ denote a non-intersecting curve separating the complex plane into two domains, $\mathcal A^+$ and $\mathcal A^-$. Let also $f$ be meromorphic in $\mathbb C\setminus\mathcal A$. Then for $t$ on the curve $\mathcal A$, $f^+(t)$ will denote $\lim_{x\to t, x\in\mathcal A^+} f(x)$, provided it exists. The notation $f^-(t)$ is defined similarly.

\subsection{Conformal mappings}

\begin{lemma}\label{lem:pi_{1}}
    For any $a\in\mathcal{S}_1^+\cap\mathbb{R}$, there exists a unique conformal mapping $\pi_{1}$ from $\mathcal{S}_1^+$ onto $\mathcal{C}^+$ such that $\pi_{1}(a)=0$ and $\pi_{1}'(a)>0$. Moreover, $\pi_{1}^+(1)=1$ and
    \begin{enumerate}[label=\textnormal{(\roman{*})},ref=\textnormal{(\roman{*})}]
        \item\label{lem:pi_{1}-1} In the case $p_{1,1}=0$, $\pi_{1}^+(0)=-1$;
        \item\label{lem:pi_{1}-2} In the case $p_{1,1}\neq 0$, $\pi_{1}(0)\in(-1,1)$.
    \end{enumerate}
\end{lemma}

\begin{proof}
    The existence and uniqueness of $\pi_{1}$ is an immediate consequence of the classical Riemann mapping theorem.
    
    We now prove that $\pi_{1}(x) = \overline{\pi_{1}(\overline{x})}$. Let $\chi(x)$ denote $\overline{\pi_{1}(\overline{x})}$ and $a_0$ denote the other intersection (not $1$) of $\mathcal{S}_1$ and $\mathbb{R}$. Since $\pi_{1}(x)$ maps one-to-one from $\mathcal{S}_1^+$ onto $\mathcal{C}^+$, which are both symmetric w.r.t.\ the real axis, then so is $\chi(x)$. Moreover, since $\chi(a)=\pi_{1}(a)$ and $\chi'(a)=\overline{\pi_{1}'(\overline{a})}>0$ (since $a$ is real and $\pi_{1}'(a)>0$), then $\chi(x)$ maps conformally $\mathcal{S}_1^+$ onto $\mathcal{C}^+$ and $\chi = \pi_{1}$. This implies that $\pi(x)\in (-1,1)$ for all $x\in (a_0,1)$. Since $\pi_{1}$ is univalent and $\pi_{1}'(a)>0$, then $\pi_{1}$ maps one-to-one from $(a_0,1)$ onto $(-1,1)$ in the same direction. Hence, $\pi_{1}^+(1)=1$ and $\pi_{1}^+(a_0)=-1$. Since $a_0=0$ in the case $p_{1,1}=0$ and $0\in (a_0,1)$ in the case $p_{1,1}\neq 0$, then \ref{lem:pi_{1}-1} and \ref{lem:pi_{1}-2} follow.
\end{proof}

With $\pi_{1}$ as in Lemma \ref{lem:pi_{1}}, we introduce
\begin{equation}
\label{eq:pi_{2}}
	\pi_{2}(z)=\frac{1}{\pi_{1}(z)},
\end{equation}
which maps $\mathcal{S}_1^+$ conformally onto $\mathcal{C}^-$, the exterior of the unit disk.
Denote by $\mathcal{H}^+$ (resp.\ $\mathcal{H}^-$) the interior of the right (resp.\ left) half-plane and define the following mapping
\begin{equation}
\label{eq:expression_phi}
    \phi(z)=-\frac{z+1}{z-1}.
\end{equation}
It is well known that $\phi$ maps conformally $\mathcal{C}^+$ (resp.\ $\mathcal{C}^-$) onto $\mathcal{H}^+$ (resp.\ $\mathcal{H}^-$). We finally introduce a few further conformal mappings:
\begin{itemize}
    \item $\psi_{1}=\phi\circ\pi_{1}:\mathcal S_1^+\to \mathcal{H}^+$;
    \item $\psi_{2}=\phi\circ\pi_{2}:\mathcal S_1^+\to \mathcal{H}^-$;
    \item $\pi_{10},\pi_{20},\phi_{0},\psi_{10},\psi_{20}$ denote respectively the inverses of $\pi_{1},\pi_{2},\phi,\psi_{1},\psi_{2}$.
\end{itemize}

The following lemma presents some crucial properties of $\psi_{10}$ and $\psi_{20}$.
\begin{lemma}\label{lem:psi_{10}-psi_{20}}
We have:
\begin{enumerate}[label=\textnormal{(\roman{*})},ref=\textnormal{(\roman{*})}]
    \item\label{item:psi-0}$\psi_{10}(\infty)=1$ and $\psi_{20}(\infty)=1$;
    \item\label{item:psi-1}$\psi_{20}(t) = \psi_{10}(-t)$ for all $t\in\mathcal{H}^-$;
    \item\label{item:psi-2}$\psi_{10}(\overline{t})=\overline{\psi_{10}(t)}$ for all $t\in\mathcal{H}^+$;
    \item\label{item:psi-3}$\psi_{20}(\overline{t})=\overline{\psi_{20}(t)}$ for all $t\in\mathcal{H}^-$;
    \item\label{item:psi-4}$\psi_{20}^-(t)=\overline{\psi_{10}^+(t)}$ for all $t\in {i}\mathbb{R}$, i.e., $\psi_{10}^+\times\psi_{20}^-({i}\mathbb{R})=\mathcal{K}$.
    \end{enumerate}
\end{lemma}

\begin{proof}
Item \ref{item:psi-0} follows by construction: $\psi_{10}(\infty)=\pi_{10}\circ \phi_{0}(\infty)=\pi_{10}(1)=1$, see \eqref{eq:expression_phi} and Lemma \ref{lem:pi_{1}}. Similar computations hold for $\psi_{20}(\infty)$.

In order to prove \ref{item:psi-1}, we show equivalently that for all $x\in\mathcal{S}_1^+$, $\psi_{2}(x) =- \psi_{1}(x)$. This comes from the construction of $\pi_{2}$ and $\phi$. Indeed, for all $x\in\mathcal{S}_1^+$,
\begin{equation}
\label{eq:psi_{2}=-psi_{1}}
    \psi_{2}(x) = \phi(
    \pi_{2}(x))= \phi\left(\frac{1}{\pi_{1}(x)}\right) = - \phi(\pi_{1}(x)) = -\psi_{1}(x),
\end{equation}
see \eqref{eq:pi_{2}} and \eqref{eq:expression_phi}.

We conclude by proving \ref{item:psi-4}. We show equivalently that for all $x\in\mathcal{S}_1$, $\psi_{2}^+(x)=\psi_{1}^+(\overline{x})$. Notice that for $x\in\mathcal{S}_1$, since $\pi_{2}^+(x)=\frac{1}{\pi_{1}^+(x)}=\pi_{1}^+(\overline{x})$, then $\psi_{2}^+(x)=\psi_{1}^+(\overline{x})$. The proof is complete.
\end{proof}

\begin{lemma}
\label{lem:psi_{1}(0)}
The following assertions hold true:
\begin{enumerate}[label=\textnormal{(\roman{*})},ref=\textnormal{(\roman{*})}]
	\item\label{lemma:psi_{1}(0)-0}The asymptotic behavior of $\psi_{1}$ around $1$ is
	\begin{equation*}
		\psi_{1}(x)\sim \frac{c}{(1-x)^{\pi/\theta}},
	\end{equation*}
	for some non-zero constant $c$ and $\theta$ as in \eqref{eq:angle_at_1}. 
		\item\label{lemma:psi_{1}(0)-1} In the case $p_{1,1}=0$, $\psi_{1}$ can be extended analytically around $0$, such that $\psi_{1}(0)=0$ and $\psi_{1}'(0)\neq 0$.
		\item\label{lemma:psi_{1}(0)-2} In the case $p_{1,1}\neq 0$, $\psi_{1}(0)>0$. Without loss of generality, we will assume that $\psi_{1}(0)=p_{1,1}$.
	\end{enumerate}
\end{lemma}

\begin{proof}
We first prove \ref{lemma:psi_{1}(0)-0} about the asymptotic behavior of $\psi_{1}$ around $1$. Recall that $\pi_{10}$ maps conformally $\mathcal{C}^+$ onto $\mathcal{S}_1^+$ and the interior angle of $\mathcal{S}_1^+$ at $1$ is $\theta$. Then by \cite[Thm~3.11]{Po-92}, there exists a non-zero constant $c_1$ such that as $z\to 1$,
\begin{equation*}
   \pi_{10}(z)= \pi_{10}(1) + (1-z)^{\theta/\pi}(c_1+o(1)).
\end{equation*}
Hence as $x\to1$, there exists $c\neq 0$ such that
\begin{equation*}
   \pi_{1}(x)=1 + (1-x)^{\pi/\theta}(1/c_1+o(1)) 
   \quad \text{and} \quad \psi_{1}(x) = \phi\circ \pi_{1}(x) \sim \frac{c}{ (1-x)^{\pi/\theta}}.
\end{equation*}
	 

We now prove \ref{lemma:psi_{1}(0)-1}. At the root $(0,0)$ of $K(x,y)$, since $\partial_y K(0,0)=-p_{1,0}\not= 0$, then by the implicit function theorem (\cite[Sec.~B.4]{Gu-65}), there exists a unique function $Y(x)$ analytic in a neighbourhood $V$ of $0$ such that $K(x,Y(x))=0$, for all $x\in V$. One can further point out that $Y(x)$ is a conformal mapping with $V$ small enough, since $Y'(0)=-\partial_x K(0,0)/\partial_y K(0,0) = -1\not= 0$. By the description of $\mathcal{K}$, one knows that $Y(x)=\overline{x}$ for all $x\in\mathcal{S}_1\cap V$, and thus the image of $\mathcal{S}_1^-\cap V$ under $Y(x)$ is contained in $\mathcal{S}_1^+$ (by the principle of corresponding boundaries, \cite[p.~109]{Ev-66}). Since the function
\begin{equation}
\begin{cases}
   \psi_1(x)& \quad\text{if}\quad x\in V\cap S_1^+\\
   \psi_2(Y(x))& \quad\text{if}\quad x\in V\cap S_1^-
\end{cases}
\end{equation}
is continuous and sectionally analytic on $V$, then by Morera's theorem, $\psi_2(Y(x))$ is an analytic continuation of $\psi_1(x)$. Moreover, by \cite[Thm~3.9]{Po-92}, $\psi_1'(0)\not= 0$.
	 
We finally prove \ref{lemma:psi_{1}(0)-2}. Since $\pi_{1}(0)\in(-1,1)$, then $\psi_{1}(0)>0$. It is seen that if in Lemma~\ref{lem:pi_{1}} we choose two different points $a_1$ and $a_2$, and consider the associated mappings $\psi_{1}$ and ${\widetilde\psi}_{10}$, then there exists a constant $c>0$ such that $\psi_{1}=c\widetilde\psi_{1}$. Indeed, since $\phi$ is a conformal mapping from $\mathcal{C}^+$ onto $\mathcal{H}^+$, so is $c\phi$ for any $c>0$. Hence, $c\widetilde\psi_{1}$ is a conformal mapping from $\mathcal{S}_1^+$ onto $\mathcal{H}^+$. One may choose $c=\frac{\psi_{1}(a_2)}{\widetilde\psi_{1}(a_2)}=\psi_{1}(a_2)>0$ and the proof is complete.
\end{proof}



\subsection{Boundary value problems}

We now have collected enough material to construct a BVP, whose solutions relate to the generating functions $H(x,0)$ and $H(0,y)$. For the sake of brevity, we shall denote $KH(x,0)=K(x,0)\times H(x,0)$ and similarly $KH(0,y)=K(0,y)\times H(0,y)$. Define a function $F$ on $\mathbb C\setminus i\mathbb R$ by
\begin{equation}
\label{eq:def_f_BVP}
     F(t) = \left\{
     \begin{array}{lc}
     KH(\psi_{10}(t),0)-\frac{KH(0,0)}{2} & \text{if }  t\in\mathcal{H}^+,\medskip\\
     -KH(0,\psi_{20}(t))+\frac{KH(0,0)}{2} & \text{if }  t\in\mathcal{H}^-,
     \end{array}\right.
\end{equation}
and recall that the notation $F^\pm(t)$ has been introduced at the beginning of Section~\ref{sec:BVP}.
\begin{proposition}
\label{prop:BVP}
Assume \ref{H1:jumps}--\ref{H2:harmonic}, and assume in addition that the radii of convergence of $H(x,0)$ and $H(0,y)$ are greater than or equal to one. Then $F$ in \eqref{eq:def_f_BVP} is sectionally analytic on $\mathbb{C} \setminus i\mathbb R$ and more specifically, $F$ satisfies the following BVP:
\begin{enumerate}[label=\textnormal{(\roman{*})},ref=\textnormal{(\roman{*})}]
    \item\label{item:BVP-1}$F$ is analytic on $\mathcal{H}^+$ and continuous on $\mathcal{H}^+\cup{i}\mathbb{R}$.
	\item\label{item:BVP-2} $F$ is analytic on $\mathcal{H}^-$ and continuous on $\mathcal{H}^-\cup{i}\mathbb{R}$. 
	\item\label{item:BVP-3}For all $t\in{i}\mathbb{R}$, 
\begin{equation}
\label{eq:boundary_condition}
	F^+(t) - F^-(t)=0.
\end{equation}
	\item\label{item:BVP-4}If the associated harmonic function $h$ is non-zero, then $F(\infty)=\infty$.
	\end{enumerate}
\end{proposition}

\begin{proof}
We start with the proof of \ref{item:BVP-1}. Since $H(x,0)$ is assumed to be analytic in the unit disk, which contains $\mathcal{S}_1^+$, then $F$ is analytic on $\mathcal H^+$. Moreover, with the exception of the point $1$, the curve $\mathcal{S}_1$ is contained in the open unit disk, so the continuity on $\mathcal{H}^+\cup {i}\mathbb{R}$ follows. Item \ref{item:BVP-2} would be proved along the same lines.

We now prove \ref{item:BVP-3}. By Lemma \ref{lem:psi_{10}-psi_{20}} \ref{item:psi-4}, we have $K(\psi_{10}^+(t),\psi_{20}^-(t))=0$ for all $t\in {i}\mathbb{R}$. So the identity \eqref{eq:boundary_condition} is just a consequence of the functional equation \eqref{eq:functional_equation}.

It remains to prove \ref{item:BVP-4}. If $F$ is bounded at infinity, then by Lemma \ref{lem:constant_functions} below, the function $F$ should be constant, and actually even  identically zero by the fact that if $x=y=0$, then $KH(x,0)=KH(0,y)=KH(0,0)$. 
\end{proof}

\begin{lemma}
\label{lem:constant_functions}
Let $F$ be sectionally analytic on $\mathbb C\setminus i\mathbb R$, satisfying to \ref{item:BVP-1}, \ref{item:BVP-2} and \ref{item:BVP-3} of Proposition~\ref{prop:BVP}. If in addition $F$ is bounded at infinity, then $F$ is a constant function.
\end{lemma}

\begin{proof}
Since $F$ is sectionally analytic on $\mathbb{C}\setminus i\mathbb{R}$ and continuous on $\mathbb{C}$ ($F^+(t)=F^-(t)$ on ${i}\mathbb{R}$ by \eqref{eq:boundary_condition}), then $F$ is an entire function. Moreover, since $F$ is bounded, then by Liouville's theorem, $F$ is a constant function.
\end{proof}

\subsection{Polynomial solutions to the boundary value problem}  
\label{sec:polynomial_solutions}

As shown in~Proposition~\ref{prop:BVP}~\ref{item:BVP-4}, the solutions $F$ to the BVP cannot be bounded at infinity for non-trivial harmonic functions. It is thus natural to look at functions $F$ of (polynomial) order $n$ at infinity, i.e., such that $F(t)=O(t^n)$, for some $n\geq 1$. Such functions will be called polynomial solutions and are studied in this section. 


\begin{cor}
\label{cor:solutions-poly_BVP}
Given that $F$ in \eqref{eq:def_f_BVP} has a pole of order $n>0$ at infinity, then $F$ is a polynomial of degree $n$ satisfying to the following conditions:
    \begin{align}
    \label{eq:condition_p11} &F(p_{1,1})=-F(-p_{1,1})=\frac{KH(0,0)}{2}=-\frac{1}{2}p_{1,1}h(1,1).
    \end{align}
    We then have, for $x,y\in\mathcal{S}_1^+$,
    \begin{align}\label{eq: KH(x,0),KH(0,y)}
        KH(x,0) &= F(\psi_{1}(x)) + \frac{KH(0,0)}{2},\\
        KH(0,y) &= -F(-\psi_{1}(y)) + \frac{KH(0,0)}{2},\\
        \label{eq: KH(x,y)}H(x,y) &= \frac{F(\psi_{1}(x))-F(-\psi_{1}(y))}{K(x,y)}.
    \end{align}
\end{cor}

\begin{proof}
With a proof similar to that of Lemma \ref{lem:constant_functions}, we derive that $F$ is an entire function. Since $F$ has a pole of order $n$ at infinity, then by the extended version of Liouville's theorem, $F$ is a polynomial.

The identity \eqref{eq:condition_p11} is derived very naturally. First, in the case $p_{1,1}=0$, since $\psi_{10}^+(0)=0$ (see Lemma \ref{lem:psi_{1}(0)}) and $KH(0,0)=0$, one must have $F(0) = KH(\psi_{10}^+(0),0)=0$. In the case $p_{1,1}\neq 0$, since $\psi_{10}(p_{1,1})=\psi_{20}(-p_{1,1})=0$ (see again Lemma \ref{lem:psi_{1}(0)}), one has
\begin{equation*}
   F(p_{1,1}) = -F(-p_{1,1})=KH(\psi_{10}(p_{1,1}),0)-\frac{KH(0,0)}{2}=\frac{KH(0,0)}{2}.
\end{equation*}
Applying into the solutions the inverses of $\psi_{10}$ and $\psi_{20}$, which are $\psi_{1}$ and $\psi_{2}$, we derive \eqref{eq: KH(x,0),KH(0,y)}. Equation \eqref{eq: KH(x,y)} is deduced from \eqref{eq: KH(x,0),KH(0,y)} together with the main functional equation \eqref{eq:functional_equation}.
\end{proof}



\section{Proof of our main results (Theorems \ref{thm:main_intro-1} and \ref{thm:main_intro-2})}
\label{sec:solution_BVP}

This part is structured as follows: we successively prove Theorem~\ref{thm:main_intro-1}~\ref{thm:main_intro-1:it1}, Theorem~\ref{thm:main_intro-1}~\ref{thm:main_intro-1:it2}, Theorem~\ref{thm:main_intro-1}~\ref{thm:main_intro-1:it3} and Theorem~\ref{thm:main_intro-2}. Finally, in the independent Section~\ref{sec:sym_antisym}, we study some features of symmetric and anti-symmetric harmonic functions.

\subsection{Proof of Theorem \ref{thm:main_intro-1} \ref{thm:main_intro-1:it1} and \ref{thm:main_intro-1:it2}}

\begin{proof}
Let $P_n$ be the family of polynomials introduced in \eqref{eq:def_P_n} and $H_n$ the associated bivariate function \eqref{eq:expression_H_n}. 
Let us first prove that $H_n$ defines a bivariate power series. This is obvious in the case $p_{1,1}\neq 0$, since $K(0,0)\neq 0$. We therefore assume that $p_{1,1}=0$. In this case $P_{n}(t)=t^n$, and one can rewrite \eqref{eq:expression_H_n} as
\begin{equation*}
	H_n(x,y) = \frac{\psi_{1}(x)^n-(-\psi_{1}(y))^n}{K(x,y)} = \frac{\psi_{1}(x)+\psi_{1}(y)}{K(x,y)}\sum_{k=0}^{n-1}\psi_{1}(x)^k(-\psi_{1}(y))^{n-1-k}.
\end{equation*}
By Lemma \ref{lem:psi_{1}(0)}, $\psi_{1}$ is defined on a neighbourhood of $0$. Since $\psi_{1}(x)=-\psi_{1}(y)$ on a neighbourhood of $(0,0)$ in the one-dimensional real variety $\mathcal{K}$ in \eqref{eq:main_domain}
(see Lemma \ref{lem:psi_{10}-psi_{20}} \ref{item:psi-4} and its proof), by analyticity of $\psi_{1}$ we have $\psi_{1}(x)=-\psi_{1}(y)$ on a neighbourhood of $(0,0)$ in the one-dimensional complex variety where $K(x,y)=0$.

Now recall that on a neighbourhood of $(0,0)$,
\begin{equation*}
   \frac{\partial}{\partial y}\bigl(\psi_{1}(x)+\psi_{1}(y)\bigr)\neq 0 \quad \text{and}\quad \frac{\partial}{\partial y}K(x,y)\neq 0.
\end{equation*}
Hence, by the Weierstrass preparation theorem for analytic functions in several variables (see Appendix~\ref{sec:app_W}), there exist two functions $u(x,y)$ and $v(x,y)$ analytic in a neighbourhood of $(0,0)$ and not vanishing at $(0,0)$, as well as a function $g(x)$ analytic in a neighbourhood of $0$ and vanishing at $0$, such that 
\begin{equation*}
    \frac{\psi_{1}(x)+\psi_{1}(y)}{K(x,y)} = \frac{u(x,y)(y-g(x))}{v(x,y)(y-g(x))} = \frac{u(x,y)}{v(x,y)}.
\end{equation*}
This implies that $H_n(x,y)$ is analytic in a neighbourhood of $(0,0)$, i.e., $H_n(x,y)$ is the generating function of a function $h_n$ on the quadrant. Since $H_{n}$ satisfies the main functional equation \eqref{eq:functional_equation}, $h_n$ is a harmonic function. 

Since $\psi_{1}$ is defined on $\mathcal{S}_1^+$ and $K(x,y)$ does not have any solution in $\mathcal{S}_1^+\times\mathcal{S}_1^+$ by Lemma~\ref{lemma:no_solution}, then $H_n(x,y)$ is analytic in $\mathcal{S}_1^+\times\mathcal{S}_1^+$. The proof is complete.
\end{proof}

\subsection{Proof of Theorem \ref{thm:main_intro-1} \ref{thm:main_intro-1:it3}}

Our main result in this part will be stated under Proposition~\ref{prop:main_intro-1:it3} and appears as a refinement of Theorem~\ref{thm:main_intro-1}~\ref{thm:main_intro-1:it3}. We first introduce some necessary notation. The Laplace transform of a discrete function $f$ on $\mathbb{N}^2$ is defined as
\begin{equation*}
	\mathcal{L}f(x,y) = \sum_{u,v=0}^\infty f(u,v){e}^{-(ux+vy)}.
\end{equation*}
For a measurable function $f$ on the quadrant $\mathcal{Q}$, its Laplace transform is defined by 
\begin{equation*}
	\mathcal{L}f(x,y)=\iint_{[0,\infty)^2} f(u,v)e^{-(ux+vy)}dudv.
\end{equation*}
Observe that the above Laplace transforms are well defined (and analytic) on $\mathcal{H}^+\times \mathcal{H}^+$ as soon as the growth of $f$ at infinity is at most polynomial.

Finally, we introduce
\begin{equation}
	\label{eq:h_n^sigma}
	h_n^\sigma (x,y) = \Im \bigl((x/\sin\theta+y\cot\theta+{i}y)^{n\pi/\theta}\bigr)= g_n\bigl(x/\sin\theta+y\cot\theta ,y\bigr),
\end{equation}
with $g_n(x,y) = \Im\bigl((x+{i}y)^{n\pi/\theta}\bigr)$. It is easily checked that setting
\begin{equation}
\label{eq:continuous_Laplacian_theta}
	\Delta = \frac{\partial^2}{\partial x^2} - 2\cos\theta\frac{\partial^2}{\partial x\partial y} +\frac{\partial^2}{\partial y^2},
\end{equation}
one has $\Delta h_n^\sigma=0$. Notice that \eqref{eq:continuous_Laplacian_theta} is exactly \eqref{eq:continuous_Laplacian} with $\sigma_1=\sigma_2$ and $\theta = \arccos\frac{-\sigma_{12}}{\sqrt{\sigma_{1}\sigma_{2}}}$, see \eqref{eq:angle_at_1}.

\begin{proposition}
\label{prop:main_intro-1:it3}
Let $h_n$ be the harmonic function with generating function $H_n$ defined by \eqref{eq:expression_H_n}, and $h_n^\sigma$ be the continuous harmonic function defined in \eqref{eq:h_n^sigma}. Then there exists a positive constant $c$ such that for all $x,y>0$,
\begin{equation*}
    \lim_{m\to\infty} \frac{c}{m^{n\pi/\theta +1}}\mathcal{L}h_n( mx, my) = \mathcal{L}h_n^\sigma (x,y).
\end{equation*}
\end{proposition}

Before embarking into the proof of Proposition  \ref{prop:main_intro-1:it3}, we provide a few remarks on the construction of the function $h_n^\sigma$ as introduced in \eqref{eq:h_n^sigma}.
We first consider a Dirichlet problem on the cone ${D}=\{(r\cos t,r\sin t):r> 0, t\in(0,\theta)\}$. Recall from our introduction that the associated set of harmonic functions may be described as
\begin{equation*}
	H(D)= \bigl\{\sum_{n\geq 1} a_ng_n(x,y): a_n\in\mathbb{R}
	\text{ and }\vert a_n\vert^{1/n}\to 0 \bigr\},
\end{equation*}
with $g_n$ as above. By the linear transformation $(x,y)\mapsto (x/\sin\theta+y\cot\theta ,y)$ from the positive quadrant $\mathcal{Q}$ onto the cone $D$, one can transform the above problem into a problem on $\mathcal{Q}$ with corresponding Laplacian \eqref{eq:continuous_Laplacian_theta}. Naturally, the set of underlying continuous harmonic functions associated is described by \eqref{eq:basis_continuous_case}.


\begin{proof}[Proof of Proposition  \ref{prop:main_intro-1:it3}]
By definition of Laplace transforms and generating functions, one has
\begin{multline}
\label{comp_h_n}
    \mathcal{L}h_n(mx,my) = \frac{{e}^{-(x+y)/m}}{m} H_n({e}^{-x/m},{e}^{-y/m})\\
     = \frac{{e}^{-(x+y)/m}}{m}   \frac{P\bigl(\psi_{1}({e}^{-x/m})\bigr)-P\bigl( -\psi_{1}({e}^{-y/m}) \bigr)}{K({e}^{-x/m},{e}^{-y/m})}.
\end{multline}
Using now Lemma~\ref{lem:psi_{1}(0)}~\ref{lemma:psi_{1}(0)-0}, one deduces that as $m\to\infty$, 
\begin{align*}
	\psi_{1}({e}^{-x/m})&\sim \frac{c_0}{(1-{e}^{-x/m})^{\pi/\theta}}\sim \frac{c_0}{(x/m)^{\pi/\theta}},\\
	P\left(\psi_{1}({e}^{-\frac{x}{m}})\right) &\sim c_1\left(\left(\frac{m}{x}\right)^{\frac{2\pi}{\theta}}-p_{1,1}\right)^{\left\lfloor \frac{n}{2}\right\rfloor} \left(\frac{m}{x}\right)^{n\,[2]}\sim c_1\left(\frac{m}{x}\right)^{\frac{n\pi}{\theta}},
\end{align*}
where $c_0$ and $c_1$ are non-zero constants. 

On the other hand, observe that 
\begin{equation*}
   \sum_{k,\ell} p_{k,\ell}e^{-\bigl(u(-k+1)+v(-\ell+1)\bigr)}
\end{equation*}
is the Laplace transform $\mathcal{L}\mu(u,v)$ of the probability measure $\mu=\sum_{k,\ell}p_{k,\ell}\delta_{-k+1,-\ell+1}$ supported on $\mathbb{N}\times \mathbb{N}$. Since $\mu$ admits second moments, we have as $u$ and $v$ go to $0$ (see \cite[Ch.~XIII.2, (2.5)]{Fel-08})
\begin{equation*}
   \mathcal{L}\mu(u,v)=1-\mu(X)u-\mu(Y)v+\frac{1}{2}\cdot\bigl(\mu(X^2)u^2+2\mu(XY)uv+\mu(Y^2)v^2\bigr)+o(u^2+v^2),
\end{equation*}
 where, using that $\sum k p_{k,\ell}=\sum \ell p_{k,\ell}=0$, $\mu(X)=\sum p_{k,\ell}(-k+1)=1$ is the first moment of the first coordinate according to $\mu$, similarly $\mu(Y)=1$, and 
\begin{equation*}
   \mu(X^2)=\sum p_{k,\ell}(-k+1)^2=1+\sum p_{k,\ell}k^2,
\end{equation*}
$\mu(Y^2)=1+\sum p_{k,l}\ell^2$ and $\mu(XY)=1+\sum p_{k,\ell}k\ell$. Hence, we have as $u$ and $v$ go to $0$
\begin{multline*}
\mathcal{L}\mu(u,v)\\
=1-(u+v)+\frac{(u+v)^2}{2}+\frac{u^2\sum p_{k,\ell}k^2}{2}+uv\sum p_{k,\ell}k\ell+\frac{v^2\sum p_{k,\ell}\ell^2}{2}+o(u^2+v^2).
\end{multline*}
Thus, as $u$ and $v$ go to $0$, 
\begin{align*}
{e}^{-(u+v)}-\mathcal{L}\mu(u,v)=&-\frac{u^2\sum p_{k,\ell}k^2+2uv\sum p_{k,\ell}k\ell+v^2\sum p_{k,\ell}\ell^2}{2}+o(u^2+v^2)\\
=&\frac{\sum k^2p_{k,\ell}}{2m^2}(u^2-2uv\cos\theta +v^2)+o(u^2+v^2),
\end{align*}
where in the last equality we used $\sum p_{k,\ell}k^2=\sum p_{k,\ell}\ell^2$ and $\sum p_{k,\ell}k\ell=-\cos \theta \sum p_{k,\ell}k^2$. Hence, we have for $x,y>0$ fixed and $m$ going to infinity
\begin{align*}
	K({e}^{-x/m},{e}^{-y/m}) &= {e}^{-(x+y)/m} - \sum p_{k,\ell}{e}^{-(x(-k+1)+y(-\ell+1))/m}\\
	&={e}^{-(x+y)/m} -\mathcal{L}\mu(x/m,y/m)\\
	& = -\frac{\sum k^2p_{k,\ell}}{2m^2}(x^2-2xy\cos\theta +y^2)+o\left(\frac{1}{m^2}\right).
\end{align*}
Going back to \eqref{comp_h_n}, we have for some non-zero constant $c_2$ that as $m\to \infty$,
\begin{align*}
    \mathcal{L}h_n(mx,my) \sim c_2 m^{n\pi/\theta+1}\frac{(x^{-\pi/\theta})^n - (-y^{-\pi/\theta})^n}{x^2-2xy\cos\theta+y^2}.
\end{align*}
We now move the Laplace transform of the continuous harmonic function $h_n^\sigma$:
\begin{align*}
	\mathcal{L}h_n^\sigma(x,y) &= \int_0^\infty \int_0^\infty h_n^\sigma(u,v){e}^{-(ux+vy)}dudv\\
	& = \int_0^\infty \int_0^\infty g_n(u/\sin\theta+v\cot\theta ,v){e}^{-(ux+vy)}dudv\\
	& =  \int_0^\infty \int_{v'\cot\theta}^\infty g_n(u',v'){e}^{-(u'\sin\theta - v'\cos\theta)x-v'y}\sin\theta du'dv'\\
	& =\sin\theta \int_0^\infty\int_0^{\theta}g_n(r\cos t,r\sin t){e}^{-r(x\sin (\theta-t)+y\sin t)}rdtdr\\
	& = \sin\theta \int_0^{\theta} \sin \left(n\frac{\pi}{\theta}t\right) \int_0^\infty r^{n\frac{\pi}{\theta} +1}{e}^{-r(x\sin (\theta-t)+y\sin t)}drdt\\
	& = \Gamma(n\pi/\theta+2)\sin\theta \int_0^{\theta} \frac{\sin (n\frac{\pi}{\theta}t)}{(x\sin (\theta-t)+y\sin t)^{n\frac{\pi}{\theta} +2}}dt\\
	& =  \left.\frac{\frac{\Gamma(n\pi/\theta+2)\sin\theta}{n\pi/\theta+1}}{x^2-2xy\cos\theta+y^2} \frac{-x\sin (\theta-(n\frac{\pi}{\theta}+1)t)-y\sin((n\frac{\pi}{\theta}+1)t)}{(x\sin (\theta-t)+y\sin t)^{n\frac{\pi}{\theta} +1}}\right|_0^\theta\\
	& = \frac{\Gamma(n\pi/\theta+2)}{(n\pi/\theta+1)(\sin\theta)^{n\pi/\theta-1}}\frac{(x^{-\pi/\theta})^n - (-y^{-\pi/\theta})^n}{x^2-2xy\cos\theta+y^2}.
\end{align*}
One can compare the above results and the proof is then complete.
\end{proof}

Remark that one cannot deduce from Proposition \ref{prop:main_intro-1:it3} the asymptotics of $h_n(i,j)$ as $n$ is fixed and $i+j\to\infty$. However, classically, Proposition \ref{prop:main_intro-1:it3} entails that $h_n$ is converging locally in the $L^1$-norm towards $h_n^\sigma$.


\subsection{Proof of Theorem \ref{thm:main_intro-2}}

As a first step, we need to prove that as $n$ increases, the harmonic function $h_n$ in \eqref{eq:expression_H_n} has more and more zero coefficients, in a sense which will be made precise in Lemma~\ref{lemma:triangle_lemma} (case $p_{1,1}=0$) and Lemma~\ref{lemma:rectangle_lemma} (case $p_{1,1}\neq 0$), see also Figure~\ref{fig:zeros_h_n}. Let us remark here that Lemma~\ref{lemma:rectangle_lemma} is an a posteriori justification of our choice of the polynomials $P_n$ in the definition \eqref{eq:def_P_n}. 

\begin{lemma}
\label{lemma:triangle_lemma}
Assuming $p_{1,1}=0$, then $h_n$ satisfies the following assertions:
    \begin{itemize}
        \item For all $i,j\geq 1$ such that $i+j\leq n$, we have $h_n(i,j)=0$;
        \item For all $i,j\geq 1$ such that $i+j= n+1$, we have $h_n(i,j)\neq 0$.
    \end{itemize}
\end{lemma}

\begin{figure}
\hspace{-10mm}
\begin{tikzpicture}[scale=2.5]
\draw[gray,very thin] (-0.1,-0.1) grid (1.1,1.1)
	 [step=0.25cm] (0,0) grid (1.1,1.1);
\draw[->] (-0.1,0) -- (1.1,0);
\draw[->] (0,-0.1) -- (0,1.1);


\put(-2.6,-2.6){\textcolor{black}{$\bullet$}}
\put(15.1,-2.6){\textcolor{black}{$\bullet$}}
\put(15.1,15.1){\textcolor{blue}{$\bullet$}}
\put(32.8,-2.6){\textcolor{black}{$\bullet$}}
\put(32.8,15.1){\textcolor{blue}{$\bullet$}}
\put(15.1,32.8){\textcolor{blue}{$\bullet$}}
\put(50.5,-2.6){\textcolor{black}{$\bullet$}}
\put(-2.6,15.1){\textcolor{black}{$\bullet$}}
\put(-2.6,32.8){\textcolor{black}{$\bullet$}}
\put(-2.6,50.5){\textcolor{black}{$\bullet$}}
\put(-2.6,68.2){\textcolor{black}{$\bullet$}}
\put(68.2,-2.6){\textcolor{black}{$\bullet$}}
\put(15.1,68.2){\textcolor{red}{$\bullet$}}
\put(32.8,32.8){\textcolor{blue}{$\bullet$}}
\put(68.2,15.1){\textcolor{red}{$\bullet$}}
\put(50.5,32.8){\textcolor{red}{$\bullet$}}
\put(32.8,50.5){\textcolor{red}{$\bullet$}}
\put(50.5,15.1){\textcolor{blue}{$\bullet$}}
\put(15.1,50.5){\textcolor{blue}{$\bullet$}}
\end{tikzpicture}\qquad\qquad\qquad
\begin{tikzpicture}[scale=2.5]
\draw[gray,very thin] (-0.1,-0.1) grid (1.1,1.1)
	 [step=0.25cm] (0,0) grid (1.1,1.1);
\draw[->] (-0.1,0) -- (1.1,0);
\draw[->] (0,-0.1) -- (0,1.1);


\put(-2.6,-2.6){\textcolor{black}{$\bullet$}}
\put(15.1,-2.6){\textcolor{black}{$\bullet$}}
\put(32.8,-2.6){\textcolor{black}{$\bullet$}}
\put(50.5,-2.6){\textcolor{black}{$\bullet$}}
\put(68.2,-2.6){\textcolor{black}{$\bullet$}}
\put(-2.6,68.2){\textcolor{black}{$\bullet$}}
\put(-2.6,15.1){\textcolor{black}{$\bullet$}}
\put(15.1,15.1){\textcolor{blue}{$\bullet$}}
\put(32.8,15.1){\textcolor{blue}{$\bullet$}}
\put(50.5,15.1){\textcolor{blue}{$\bullet$}}
\put(-2.6,32.8){\textcolor{black}{$\bullet$}}
\put(15.1,32.8){\textcolor{blue}{$\bullet$}}
\put(32.8,32.8){\textcolor{blue}{$\bullet$}}
\put(50.5,32.8){\textcolor{blue}{$\bullet$}}
\put(-2.6,50.5){\textcolor{black}{$\bullet$}}
\put(15.1,50.5){\textcolor{blue}{$\bullet$}}
\put(32.8,50.5){\textcolor{blue}{$\bullet$}}
\put(50.5,50.5){\textcolor{blue}{$\bullet$}}
\put(50.5,50.5){\textcolor{blue}{$\bullet$}}
\put(68.2,15.1){\textcolor{red}{$\bullet$}}
\put(15.1,68.2){\textcolor{red}{$\bullet$}}
\end{tikzpicture}
\caption{Illustrations of Lemmas~\ref{lemma:triangle_lemma} and~\ref{lemma:rectangle_lemma}. The harmonic function $h_n$ is always zero on the axes. Left: in the case $p_{1,1}=0$, the coefficients of $h_n$ are zero at the blue points of the triangle, and non-zero just above (the red points). Right: in the case $p_{1,1}\neq 0$, the coefficients of $h_n$ are zero at all blue points of a square, non-zero at the extremal red points.}
\label{fig:zeros_h_n}
\end{figure}
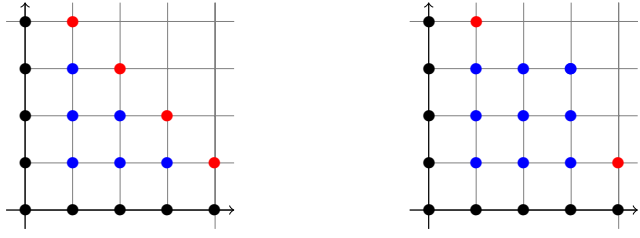

\begin{proof}
As in \eqref{eq:expression_H_n}, rewrite $H_n$ as 
\begin{align*}
    H_n(x,y) = \frac{\psi_{1}(x)^n-(-\psi_{1}(y))^n}{K(x,y)} &= \frac{\psi_{1}(x)+\psi_{1}(y)}{K(x,y)}\sum_{k=0}^{n-1}\psi_{1}(x)^k(-\psi_{1}(y))^{n-1-k}\\
    &=H_{1}(x,y)\sum_{k=0}^{n-1}\psi_{1}(x)^k(-\psi_{1}(y))^{n-1-k}.
\end{align*}
Recall from Lemma \ref{lem:psi_{1}(0)} that $\psi_{1}$ is analytic at $0$, with $\psi_{1}(0)=0$ and $\psi_{1}'(0)\not=0$. Hence, we may write
\begin{equation*}
   \psi_{1}(x)=\sum_{i\geq 1}c_{i}x^i,
\end{equation*}
with $c_1\not=0$. We then obtain, using the bivariate expansion of $H_1$, that
    \begin{equation}
    \label{eq:expansion_H_n_vanishing}
        H_n(x,y) =  \sum_{i,j\geq 1} h_1(i,j)x^{i-1}y^{j-1}\sum_{k=0}^{n-1}\left(\sum_{i\geq 1}c_i x^i\right)^k\left(-\sum_{j\geq 1}c_j y^j\right)^{n-1-k}.
    \end{equation}
By L'Hospital's rule, $\frac{\psi_{1}(x)}{K(x,0)}\rightarrow \frac{c_1}{p_{1,0}}\not =0$. Hence, $h_1(1,1)\neq 0$ and we deduce from \eqref{eq:expansion_H_n_vanishing} the proof of Lemma \ref{lemma:triangle_lemma}.
\end{proof}

\begin{lemma}
\label{lemma:rectangle_lemma}
If $p_{1,1}\neq 0$, then $h_{1}(1,1)\not=0$ and for $k\geq 1$ and $\epsilon\in\{0,1\}$, the harmonic function $h_{2k+\epsilon}$  satisfies $h_{2k+\epsilon}(i,j)=0$ for all $1\leq i,j\leq  k$. Moreover, the following matrix is invertible:
\begin{equation*}
   T_k=\begin{pmatrix}
h_{2k}(k+1,1) &h_{2k+1}(k+1,1)\\h_{2k}(1,k+1)&h_{2k+1}(1,k+1)
\end{pmatrix}.
\end{equation*}
\end{lemma}
 
\begin{proof}
First, the proof that $h_{1}(1,1)\not=0$ is similar to the end of the Lemma~\ref{lemma:triangle_lemma}. Let $k\geq 1$ and $\epsilon\in\{0,1\}$. Recall from Theorem~\ref{thm:main_intro-1}~\ref{thm:main_intro-1:it1} that $h_{2k+\epsilon}$ admits the generating function
\begin{equation*}
   H_{2k+\epsilon}(x,y)=\frac{P_{2k+\epsilon}\bigl(\psi_{1}(x)\bigr)-P_{2k+\epsilon}\bigl(-\psi_{1}(y)\bigr)}{K(x,y)},
\end{equation*}
with $P_{2k+\epsilon}(t)=t^\epsilon(t^2-p_{1,1}^2)^k=t^\epsilon(t^2-\psi_{1}(0)^2)^k$, see \eqref{eq:def_P_n}. Hence we have
    \begin{multline*}
        H_{2k+\epsilon}(x,y)
        =\frac{1}{K(x,y)}\Bigl(\psi_{1}(x)^\epsilon\bigl(\psi_{1}(x)-{\psi_{1}(0)}\bigr)^k\bigl(\psi_{1}(x)+{\psi_{1}(0)}\bigr)^k \\
        - \bigl(-\psi_{1}(y)\bigr)^\epsilon\bigl(-\psi_{1}(y)-{\psi_{1}(0)}\bigr)^k\bigl(-\psi_{1}(y)+\psi_{1}(0)\bigr)^k \Bigr).
    \end{multline*}
Since $K(0,0)\neq 0$, $\frac{1}{K(x,y)}$ is analytic at $(0,0)$. Moreover, since $\frac{1}{K(x,y)}$ satisfies the functional equation \eqref{eq:functional_equation}, then it is also a generating function of a harmonic function $h_0(i,j)$, i.e.,
\begin{equation*}
    \frac{1}{K(x,y)} = \sum_{i,j\geq 1} h_0(i,j)x^{i-1}y^{j-1}.
\end{equation*}
Since $\psi_{1}$ is analytic at $0$, which is in the interior of $\mathcal{S}_1^+$, then it has an expansion
\begin{equation*}
    \psi_{1}(x) = \psi_{1}(0) + \sum_{i\geq 1}c_i x^i,
\end{equation*}
where $c_1\neq 0$ since $\psi_{1}$ is a conformal map. Therefore,
\begin{multline*}
      H_{2k+\epsilon}(x,y) =  \left(\sum_{i,j\geq 1} h_0(i,j)x^{i-1}y^{j-1}\right)\times
         \\
\times\left[\left(\sum_{i\geq 1}c_i x^i\right)^k\left(\psi_{1}(0)+\sum_{i\geq 1}c_ix^i\right)^\epsilon\left(2{\psi_{1}(0)} + \sum_{i\geq 1}c_i x^i\right)^k\right.\\
-\left(-\sum_{j\geq 1}c_j y^j\right)^k\left.\left(-\psi_{1}(0)-\sum_{j\geq 1}c_jy^j\right)^\epsilon\left(-2{\psi_{1}(0)} - \sum_{j\geq 1}c_j y^j\right)^k\right].
    \end{multline*}
Since $h_0(1,1)\neq 0$ and $c_1\neq 0$, the bivariate power series expansion of the function above has zero coefficients for all monomials $x^iy^j$ with $i, j\leq k-1$, while the coefficient of the monomial $x^{k}$ is 
\begin{equation*}
   h_{2k+\epsilon}(k,1)=(2c_1)^kh_{0}(1,1)\psi_{1}(0)^{k+\epsilon}=(2c_1)^kh_{0}(1,1)p_{1,1}^{k+\epsilon}
\end{equation*}
and that of $y^{k}$ is 
\begin{equation*}
   h_{2k+\epsilon}(1,k)=(-1)^{\epsilon-1}(2c_1)^kh_{0}(1,1)p_{1,1}^{k+\epsilon}.
\end{equation*}
Hence,
\begin{equation*}
   \begin{pmatrix}
h_{2k}(k+1,1) &h_{2k+1}(k+1,1)\\h_{2k}(1,k+1)&h_{2k+1}(1,k+1)
\end{pmatrix}=(2c_1)^kh_{0}(1,1)p_{1,1}^{k}\begin{pmatrix}
\phantom{-}1& p_{1,1}\\-1&p_{1,1}
\end{pmatrix}.
\end{equation*}
The determinant of the latter matrix is $2^{k+1}c_1^kh_{0}(1,1)p_{1,1}^{k+1}\not=0$, which  implies the second statement of Lemma \ref{lemma:rectangle_lemma}.
\end{proof}

\begin{remark}
As the proof of Lemma \ref{lemma:rectangle_lemma} shows, choosing instead of $P_n$ in \eqref{eq:def_P_n} the family of polynomials
\begin{equation*}
   P_{m,n}=(X-p_{1,1})^m(X+p_{1,1})^n
\end{equation*}
for $m,n\geq 1$, would yield harmonic functions vanishing on the rectangle $\{(i,j):1\leq i\leq m, 1\leq j\leq n\}$. In particular, there is no uniqueness in the choice of $P_n$.
\end{remark}

\begin{lemma}
\label{lemma:consequence_triangle_lemma}
Given $p_{1,1}=0$ and a sequence $\{c_n\}_{n\geq1}$, there exists a unique sequence $\{a_n\}_{n\geq1}$ such that the (harmonic) function $\sum_{n\geq1} a_{n}h_n$ satisfies
\begin{equation}
\label{cond:consequence_triangle_lemma}
    \sum_{n\geq1} a_{n}h_{n}(i,1)=c_i,\quad \forall i\geq 1.
\end{equation}
Moreover, $\{a_n\}_{n\geq1}$ can be deduced from the (infinite) linear system of equations
\begin{equation*}
    M\cdot a = c,
\end{equation*}
where $a=(a_1,a_2,a_3,\ldots)^\top$, $c=(c_1,c_2,c_3,\ldots)^\top$ and $M$ is an infinite non-singular lower triangular matrix.
\end{lemma}
By construction, the generating function of the harmonic function given in Lemma~\ref{lemma:consequence_triangle_lemma} is the function 
\begin{equation*}
   H(x,y)=\frac{\sum_{n\geq1} a_{n}\psi_{1}(x)^n-\sum_{n\geq1} a_{n}(-\psi_{1}(y))^n}{K(x,y)}=\frac{F(\psi_{1}(x))-F(-\psi_{1}(y))}{K(x,y)},
\end{equation*}
with $F(t)=\sum_{n\geq1} a_{n}t^n$, in accordance with \eqref{eq:main-intro-2}.

\begin{proof}
By Lemma \ref{lemma:triangle_lemma}, $h_{n}$ vanishes on $(i,1)$ for $1\leq i\leq n-1$ and $h_{n}(n,1)\not=0$. Hence, the infinite matrix 
\begin{equation*}
   L=\bigl(h_{j}(i,1)\bigr)_{1\leq i,j\leq \infty}
\end{equation*}
is lower triangular, with non-zero diagonal coefficients. Hence, $L$ is invertible, and for any vector $c=(c_1,c_2,c_3,\ldots)^\top$, there exists a unique vector $a=(a_1,a_2,a_3,\ldots)^\top$ such that
    \begin{equation*}
    \begin{pmatrix}
        h_{1}(1,1) & 0 & 0 & \ldots  & 0&\ldots\\
        h_{1}(2,1) & h_{2}(2,1) & 0 &  \ldots  & 0&\ldots\\
         \vdots &  \vdots  &  \ddots  &  \ddots  &  \ddots &\ddots\\
        h_{1}(n,1) & h_{2}(n,1) &  \ldots  &  \ldots  & h_{n}(n,1)&\ddots\\
        \vdots&\vdots&\ddots&\ddots&\ddots&\ddots
    \end{pmatrix}
    \begin{pmatrix}
        a_1\\a_2\\ \vdots \\a_n\\\vdots
    \end{pmatrix}
    =
    \begin{pmatrix}
        c_1\\c_2\\ \vdots \\c_n\\\vdots
    \end{pmatrix}.
\end{equation*}
Hence, given $\{c_n\}_{n\geq1}$, there exists a unique sequence $\{a_n\}_{n\geq1}$ such that \eqref{cond:consequence_triangle_lemma} holds.
\end{proof}

\begin{lemma}
\label{lemma:consequence_rectangle_lemma}
Given $p_{1,1}\neq 0$ and two infinite sequences $\{c_n\}_{n\geq1}$ and $\{d_n\}_{n\geq2}$, there exists a unique sequence $\{a_n\}_{n\geq1}$ such that \eqref{cond:consequence_triangle_lemma} holds, as well as
\begin{equation}
\label{cond:consequence_rectangle_lemma}
    \sum_{n\geq1} a_{n}h_{n}(1,i)=d_i, \quad \forall i\geq 2.
\end{equation}
Moreover, $\{a_n\}_{n\geq1}$ can be deduced from the linear system
\begin{equation*}
   M\cdot a= b,
\end{equation*}
where $a=(a_1,a_2,a_3,\ldots)^\top$, $b=(c_1,c_2,d_2,c_3,d_3,\ldots)^\top$ and $M$ is an infinite block lower triangular matrix with invertible blocks of size $1$ or $2$ on the diagonal.
\end{lemma}

\begin{proof}
Let $\tau$ be the involution of $(\mathbb{N}\times \{1\})\cup (\{1\}\times \mathbb{N})$ switching coordinates. The only fixed point of $\tau$ is $(1,1)$. By Lemma \ref{lemma:rectangle_lemma}, for fixed values of $k\geq 1$ and $\epsilon\in\{ 0,1\}$, $h_{2k+\epsilon}$ vanishes at $(i,1)$ and $(1,i)$ for all $i\leq k-1$, and the matrix
\begin{equation*}
   T_k=\bigl(h_{2k+j-1}(\tau^{i-1}(k+1,1))\bigr)_{1\leq i,j\leq 2}
\end{equation*}
is invertible. Likewise, $h_{1}(1,1)\not=0$. Hence, 
\begin{equation*}
   M=\bigl(h_{j}(\tau^{i\,[2]}(\lfloor i/2\rfloor+1,1))\bigr)_{1\leq i,j\leq\infty}=\begin{pmatrix}
        h_{1}(1,1) &\begin{array}{cc} 0&0\end{array} & 0 & \ldots  \\
        \begin{array}{c}
         h_1(2,1)  \\
         h_1(1,2) 
        \end{array} & T_1 & 0 &  \ldots \\
         \vdots &  \vdots  & T_2  &  \ldots \\
        \vdots&\vdots&\vdots&\ddots
    \end{pmatrix},
\end{equation*}
with each $T_i$ invertible. Thus $M$ is invertible and for any vector $b=(c_1,c_2,d_2,c_3,d_3,\ldots)^\top$, there exists a unique vector $a=(a_1,a_2,a_3,\ldots)^\top$ such that $M\cdot a=b$. For such a vector $a$, we thus have conditions \eqref {cond:consequence_triangle_lemma} and \eqref{cond:consequence_rectangle_lemma} satisfied.
\end{proof}

Putting all the latter lemmas together yields the proof of Theorem \ref{thm:main_intro-2}, as follows:
\begin{proof}[Proof of Theorem \ref{thm:main_intro-2}]
Let $\Phi$ be the function as in the statement of Theorem \ref{thm:main_intro-2}. First notice that by Lemma \ref{lemma:triangle_lemma} when $p_{1,1}=0$ and Lemma \ref{lemma:rectangle_lemma} for $p_{1,1}\not=0$, the map $\Phi$ is well defined, since for any sequence $\{a_{n}\}_{n\geq 1}$ and all $i,j\geq 1$, $\sum_{n\geq 1} a_{n}h_{n}(i,j)$ is a finite sum and the harmonicity of $\sum a_n h_n$ is directly inherited by the harmonicity of each $h_n$. 

Suppose first that $p_{1,1}>0$. By Lemma \ref{lemma:consequence_rectangle_lemma}, $\Phi$ is injective, and for any pair of formal power series $F,G$ with $F(0)=G(0)$, there exists a sequence $\{a_{n}\}_{n\geq 1}$ such that 
\begin{equation*}
   F(x)=\sum_{i\geq 1}\sum_{n\geq 1}a_nh_n(i,1)x^{i-1}\quad \text{and}\quad G(y)=\sum_{j\geq 1}\sum_{n\geq 1}a_nh_n(1,j)y^{j-1}.
\end{equation*}
Since any harmonic function $h$ is uniquely determined by its sectional generating functions \begin{equation}
\label{eq:def_F-and-G}
   F(x)=\sum_{i\geq 1}h(i,1)x^{i-1} \quad \text{and}\quad G(y)=\sum_{j\geq 1}h(1,j)y^{j-1}
\end{equation}
through \eqref{eq:functional_equation}, this shows the surjectivity of $\Phi$.

The proof is more involved in the case $p_{1,1}=0$. By Lemma \ref{lemma:consequence_triangle_lemma}, $\Phi$ is injective. It remains to show that $\Phi$ is surjective. Let $h$ be a harmonic function and set $F$ as in \eqref{eq:def_F-and-G}. Then, by Lemma~\ref{lemma:consequence_triangle_lemma}, there exists $\{a_n\}_{n\geq 1}$ such that $\sum a_nh_{n}(i,1)=h(i,1)$ for all $i\geq 1$. Hence, the formal power series $F(x)$ and
\begin{equation*}
   \widetilde{F}(x)=K(x,0)\sum_{i\geq 1}\left(\sum_{n\geq1}a_nh_n(i,1)\right)x^{i-1}
\end{equation*}
coincide. Let $H$ be the generating series of $h$ and $\widetilde{H}$ the generating series of $\sum a_nh_{n}$. Then we have 
\begin{equation*}
   H(x,y)=\frac{F(x)+G(y)}{K(x,y)}\quad \text{and} \quad \widetilde{H}(x,y)=\frac{\widetilde{F}(x)+\widetilde{G}(y)}{K(x,y)}=\frac{F(x)+\widetilde{G}(y)}{K(x,y)},
\end{equation*}
with $G, \widetilde{G}\in \mathbb{R}[[t]]$. Since $K(0,0)=0$ and $\frac{\partial}{\partial x}K(0,0)\not=0$, by the Weierstrass preparation theorem (see Appendix~\ref{sec:app_W}) applied to $K(x,y)$ with respect to $x$ in the ring of formal power series $\mathbb{R}[[x,y]]$, one can write  
\begin{equation*}
   K(x,y)=u(x,y)(x-v(y)),
\end{equation*}
with $u$ invertible in $\mathbb{R}[[x,y]]$ and $v\in \mathbb{R}[[y]]$ satisfying $v(0)=0$. In particular, $F(x)+G(y)$ and $F(x)+\widetilde{G}(y)$ are both divisible by $(x-v(y))$ in $\mathbb{R}[[x,y]]$.

By the Weierstrass division theorem (see Appendix~\ref{sec:app_W}) applied to the division of $F(x)$ by the Weierstrass polynomial of degree one $x-v(y)$  in $\mathbb{R}[[x,y]]$, there is a unique Weierstrass polynomial $G$ of degree zero in $\mathbb{R}[[x,y]]$ (namely, $G\in \mathbb{R}[[y]]$ with $G(0)=0$) such that $G(y)+F(x)$ is divisible by $x-v(y)$. Hence, $G(y)=\widetilde{G}(y)$ and $H=\widetilde{H}$. This shows the surjectivity of $\Phi$.
\end{proof}

\subsection{Symmetric and anti-symmetric harmonic functions}
\label{sec:sym_antisym}

\begin{proposition}\label{proposition: symmetric, anti-symmetric harmonic function}
The harmonic function $h(i,j)$ is symmetric (resp.\ anti-symmetric) if and only if its characterizing series $F(t)=\Phi^{-1}(h)$ is odd (resp.\ even).
\end{proposition}

\begin{proof}
It is seen that $h(i,j)$ is symmetric (resp.\ anti-symmetric) if and only if $H(x,0)=H(0,x)$ (resp.\ $H(x,0)=-H(0,x)$).
    
First, in the case where $H$ is of type \eqref{eq: KH(x,0),KH(0,y)}, the identity $H(x,0)=H(0,x)$ is equivalent to $F(\psi_{1}(x)) = -F(-\psi_{1}(x))$. This implies that $h(i,j)$ is symmetric if and only if $F(t)$ is an odd function. The second statement in the anti-symmetric case would be proved similarly. 

By \eqref{eq:def_P_n}, $P_n$ is odd for $n$ odd and even for $n$ even, thus by the previous reasoning $h_n$ is symmetric for $n$ odd and anti-symmetric for $n$ even. We deduce then from Theorem \ref{thm:main_intro-2} that $\Phi(F)$ is symmetric if $F$ is odd and anti-symmetric if $F$ is even. Since the vector subspace of odd power series and the one of even power series are complementary in $\mathbb{R}_0[[t]]$, the result is deduced.
\end{proof}

\section{Various examples}
\label{sec:examples}

In this part, we provide a list of models for which one may compute the conformal map $\psi_{1}$ explicitly. We start with the example of simple random walks (Section~\ref{sec:SRW}). For this model, all generating functions happen to be rational functions, and most of the computations are rather easy to lead. We then move to arbitrary small step random walks (Section~\ref{sec:uniformization}) and obtain an expression for the conformal mapping in terms of generalized Chebychev polynomials. Finally, we construct of family of walks with arbitrary large jumps, in relation with plane bipolar orientations, for which one has a simple (algebraic) formula for the conformal mapping (Section~\ref{sec:larger_jumps}). We use this family of examples to construct harmonic functions for certain transition probabilities without second moment (Section~\ref{subsec:example_less_moments}), going beyond the classical literature on the topic.

\subsection{The simple random walk}
\label{sec:SRW}
We present here a first example, which is the simple random walk with transition probabilities $p_{1,0}=p_{0,1}=p_{-1,0}=p_{0,-1}=1/4$, see Figure \ref{fig:step_sets} and \eqref{eq:def_Laplacian_usual} for the associated Laplacian operator. The kernel takes the form
\begin{equation*}
   K(x,y) = xy\left(1-\frac{x+y+x^{-1}+y^{-1}}{4}\right)=-\frac{x(y-1)^2}{4}-\frac{y(x-1)^2}{4}.
\end{equation*}

\subsubsection*{Curve and conformal mappings}
Setting $(x,y)=(\eta s,\eta s^{-1})$ and solving (in $\eta$) the equation $K(x,y)=0$, one easily obtains 
\begin{equation*}
	\mathcal{S}_1 = \biggl\{{i}s\frac{s-{i}}{s+{i}}: s={e}^{{i}t},t\in[0,\pi)\biggr\},
\end{equation*}
which we may rewrite as
\begin{equation*}
   \mathcal{S}_1 = \{(1-\sin t,(1-\sin t)\tan t): t\in [0,\pi)\},
\end{equation*}
as announced in \eqref{eq:expression_S1_SRW}. See Figures \ref{fig:some_curves} and \ref{fig:def_theta}. One can choose a conformal mapping $\pi_{1}$ as
\begin{equation*}
	\pi_{1}(x) = \frac{x-(x-1)^2}{x+(x-1)^2}.
\end{equation*}
Indeed, $\pi_{1}$ is analytic in $\mathcal{S}_1^+$ and for all $x=\eta(e^{it})e^{it}\in\mathcal{S}_1$,
\begin{equation*}
	\pi_{1}(x) = \frac{-(x+\frac{1}{x})+3}{(x+\frac{1}{x})-1} = \frac{1+2i\frac{\sin^2t}{\cos t}}{1-2i\frac{\sin^2t}{\cos t}}\in\mathcal{C}.
\end{equation*}
With the conformal mapping $\phi$ defined in \eqref{eq:expression_phi}, we finally obtain
\begin{equation}
\label{eq:expression_psi_{1}_SRW}
	\psi_{1}(x) = \phi\circ\pi_{1}(x)= \frac{x}{(1-x)^2}.
\end{equation}

\subsubsection*{Polynomial harmonic functions}

Using Equations \eqref{eq:main-intro-2} and \eqref{eq:expression_psi_{1}_SRW}, one has
\begin{equation*}
   H(x,y)=\frac{F\bigl(\frac{x}{(1-x)^2}\bigr)-F\bigl(-\frac{y}{(1-y)^2}\bigr)}{K(x,y)}.
\end{equation*}
Writing $X=\frac{x}{(1-x)^2}$, $Y=\frac{y}{(1-y)^2}$ and $G(t)=-4F(t)$, one may rewrite the above equation more symmetrically, as
\begin{equation}
\label{eq:expressions_SRW_XY}
   H(x,y)=\frac{XY}{xy}\frac{G(X)-G(-Y)}{X+Y}.
\end{equation}
In particular, applying \eqref{eq:expressions_SRW_XY} with $G(t)= t$, we have
\begin{equation*}
	H_1(x,y)= \frac{1}{(1-x)^2(1-y)^2}=\sum_{i,j\geq 1}ijx^{i-1}y^{j-1}.
\end{equation*}
Recall that $h(i,j) = ij$ is the unique positive harmonic function for this model (unique up to multiplicative factors). 

\subsubsection*{Characterization of the positive harmonic function}
As explained in Section \ref{sec:solution_BVP}, the general Martin boundary theory implies that in the framework of this paper, there is a unique positive harmonic function, which in our construction corresponds to taking $F$ as a one-degree polynomial in Theorem \ref{thm:main_intro-2}. However, we don't have any direct proof of this general fact, except precisely for the simple random walk, for which explicit, and in our opinion instructive computations may be done.

More specifically, the question we would like to address here is the following: prove that for all polynomials $G(t)=\sum_{m=1}^{k}a_m t^m$ with degree $k\geq 2$, there exists at least one coefficient of $H(x,y)$ in \eqref{eq:expressions_SRW_XY} above which is negative (and in fact infinitely many coefficients are then negative). 

We first look at the case when $G$ is a monomial. If $G(t)=t^{2k}$, then one has $H(x,y)=-H(y,x)$ (see \eqref{eq:expressions_SRW_XY}) and thus $H$ must admit negative coefficients in its Taylor expansion, as the function itself takes negative values. The more interesting case is $G(t)=t^{2k+1}$. However, for a future use, we look at general (meaning non-necessarily odd) exponents, i.e., $G(t)=t^{k}$, for some $k\geq 1$. Then with \eqref{eq:expressions_SRW_XY} one has
\begin{equation*}
   H(x,y)=\frac{1}{xy}\sum_{\substack{i+j=k+1\\ i,j\geq 1}}(-1)^{j-1} X^i Y^j.
\end{equation*}
Moreover, as $n\to \infty$, one has 
\begin{equation*}
   [x^n]X^i \sim \frac{n^{2i-1}}{(2i-1)!}.
\end{equation*}
So for large values of $p$ and $q$,
\begin{align*}
   [x^{p+1} y^{q+1}]H(x,y)&\sim \sum_{\substack{i+j=k+1\\ i,j\geq 1}}(-1)^{j-1} \frac{p^{2i-1}}{(2i-1)!}\frac{q^{2j-1}}{(2j-1)!} \\&=pq\sum_{\substack{i+j=k-1\\ i,j\geq 0}}(-1)^{j} \frac{p^{2i}}{(2i+1)!} \frac{q^{2j}}{(2j+1)!}.
\end{align*}

Consider now the general case where $G(t)$ is a polynomial $\sum_{m=1}^{k}a_m t^m$, with $a_k\neq 0$, and take $(p,q)=(r_1,r_2)n$, where $r_1,r_2$ are positive integers and $n$ tends to infinity. Using the above estimate, we have
\begin{equation*}
   [x^{p+1} y^{q+1}]H(x,y)\sim n^{2k} r_1^{2k-1} r_2\sum_{\substack{i+j=k-1\\ i,j\geq 0}}(-1)^{j} \frac{(r_2/r_1)^{2j}}{(2i+1)!(2j+1)!}.
\end{equation*}
Replacing $r_2/r_1$ by $x$, our question is therefore equivalent to prove that for any fixed integer $k\geq 2$, there exists $x\in(0,\infty)$ such that 
\begin{equation*}
   \sum_{\substack{i+j=k-1\\ i,j\geq 0}}(-1)^{j} \frac{x^{2j}}{(2i+1)!(2j+1)!}<0.
\end{equation*}
To that purpose, we first observe that
\begin{align*}
   \sum_{\substack{i+j=k-1\\ i,j\geq 0}}(-1)^{j} \frac{x^{2j}}{(2i+1)!(2j+1)!}&=\frac{(1+ix)^{2k}-(1-ix)^{2k}}{2(2k)!ix}\\&=\frac{1}{(2k)!x}\Im ((ix+1)^{2k})\\&=\frac{(1+x^2)^k}{(2k)!x}\sin(2k\arctan x).
\end{align*}
Clearly, given $k\geq 2$ one may fix $x\in(0,\infty)$ such that the above is negative. More precisely, this function admits $k-1$ sign changes.

\subsubsection*{A related example: the king walk}
We continue with the king walk, which by definition (see the second example on Figure \ref{fig:step_sets}) admits the kernel
\begin{equation*}
   K(x,y) = xy\left(1-\frac{xy+x+xy^{-1}+y^{-1}+x^{-1}y^{-1}+x^{-1}+x^{-1}y+y}{8}\right).
\end{equation*}
As the simple walk, its unique positive harmonic function is given by $h(i,j)=ij$. However, this example is a bit different as we now have $p_{1,1}\neq 0$. A few computations starting from the kernel yield
\begin{equation*}  
    \mathcal{S}_1= \biggl\{\frac{u(t)-\sqrt{u(t)^2-4}}{2}{e}^{{i}t}:t\in[0,2\pi)\biggr\},
\end{equation*}
where $u(t) = -\cos t+\sqrt{12-3\cos^2 t}$. This curve is drawn on Figure~\ref{fig:def_theta}. Computations similar to that of the simple walk lead to the following rather simple expressions for the conformal mappings:
\begin{equation*}
    \pi_{1}(x) = \frac{3x}{x^2+x+1} \quad \text{and} \quad \psi_{1}(x) = \frac{x^2+4x+1}{(1-x)^2}.
\end{equation*}
Using Equation \eqref{eq:main-intro-2} with $F(t)=t/16$\
, then one can recover the generating function of the positive harmonic function:
\begin{equation*}
    H_1(x,y) = \frac{F(\psi_{1}(x))-F(-\psi_{1}(y))}{K(x,y)}=\frac{1}{(1-x)^2(1-y)^2} = \sum_{i,j\geq 1} ijx^{i-1}y^{j-1}.
\end{equation*}
Similar computations would lead to other harmonic functions.

\subsection{Symmetric small step random walks and comparison between the analytic approaches of \cite{CoBo-83} and \cite{FaIaMa-17,Ra-14}}
\label{sec:uniformization}

In this part, we look at the class of symmetric, small step random walks, meaning that $p_{k,\ell}=0$ as soon as $\vert k\vert\geq 2$ or $\vert \ell\vert\geq 2$. As illustrated by our bibliography, this class of models has been (and is still) widely studied in the literature, the main reason being that the zero set of the kernel is, in this case, a Riemann surface of genus $0$ or $1$, opening the way to explicit parametrizations in terms of rational or elliptic functions.

Here our main objective is twofold: first, we will derive an expression of the conformal mapping $\psi_{1}$ for small step random walks, see Proposition~\ref{prop:expression_conformal_map_small_jumps}. Doing so, we will introduce some tools from complex analysis, which are close to the analytic approach developed in \cite{FaIaMa-17,Ra-14}. As a second step, we will compare the analytic approach used in this paper (inspired by \cite{CoBo-83}) with the one of \cite{FaIaMa-17,Ra-14}.

\subsubsection*{Explicit expression for the conformal mapping $\psi_{1}$}

We first introduce a function $T_a(x)$ that generalizes  the classical Chebyshev polynomials of the first kind, obtained when $a\in\mathbb N$. This function is defined for $x\in\mathbb C\setminus (-\infty,-1]$ by
\begin{equation}
\label{eq:Chebyshev_polynomial}
T_a(x)
={_2F}_1 \Bigl(-a,a;\frac{1}{2};\frac{1-x}{2}\Bigr),
\end{equation}
where ${_2F}_1$ is the Gauss hypergeometric function. Then $T_a$
admits the following expansion, valid for $\vert x-1\vert<2$:
\begin{equation*}
  T_a(x)= \sum_{n\geq 0} \frac a{a+n} {a+n \choose 2n} 2^n (x-1)^n.
\end{equation*}
When $a\in \mathbb N$, then the above sum ranges from $0$ to $a$, and $T_a(\cos t)=\cos(at)$, by
definition of the classical Chebychev polynomial. 
Another useful formula, valid for $x$ in $\mathbb C\setminus (-\infty,-1)$, is
\begin{equation*}
  T_a(x)=\frac{1}{2}
  \Bigl(\bigl(x+\sqrt{x^2-1}\bigr)^a+\bigl(x-\sqrt{x^2-1}\bigr)^a\Bigr).
\end{equation*}

\begin{figure}
	\hspace{-10mm}
	\begin{tikzpicture}[scale=1.5]
	[step=0.25cm] (-1,-1) grid (1,1);
	\draw[->] (-1.1,0) -- (1.1,0)  node[right] {$1$};
	\draw[->] (0,-1.1) -- (0,1.1)  node[above] {$1$};
	
	\filldraw[thick,variable=\t,domain=0:180,samples=50,color=blue,fill=blue!5]
	plot ({1-sin(\t)},{tan(\t)-sin(\t)*tan(\t)});
	\draw[thick,variable=\t,domain=0:360,samples=50,magenta]
	plot ({cos(\t)},{sin(\t)});
	\draw[thick,variable=\t,domain=(3-2*sqrt(2)):1,samples=50,red]
	plot ({\t},{0});
	
	\fill[red] (0.17157,0) circle (1pt);

	\put(20,15){$\mathcal S_1$};
	\put(8,-10){$x_1$};

	\end{tikzpicture}
	\begin{tikzpicture}[scale=2.5]
	[step=0.25cm] (0,0) grid (1,1);
	\draw[->] (-0.1,0) -- (1.1,0)  node[right] {$\infty$};
	\draw[->] (0,-0.1) -- (0,1.1)  node[above] {$\infty$};

	\fill[blue!5,] (0,0) -- (0.7,0.7) -- (1,0) -- cycle;
	\draw[thick,variable=\t,domain=0:0.8,samples=50,blue]
	plot ({\t},{\t});
	\draw[thick,variable=\t,domain=0:1,samples=50,red]
	plot ({\t},{0});
	\draw[thick,variable=\t,domain=0:1,samples=50,magenta]
	plot ({0},{\t});

	\fill[red] (0.5,0) circle (0.75pt);

	\put(35,-13){$1$};
	
	\draw[->] (-0.4,0.5) -- (-0.1,0.5);
	\draw[->] (-0.1,0.4) -- (-0.4,0.4);
	\draw[->] (1,0.5) -- (1.3,0.5);
	
	\put(70,40){$\omega(s)$};
	\put(-30,40){$\sigma(x)$};
	\put(-30,15){$x(s)$};
	
	\end{tikzpicture}
	\begin{tikzpicture}[scale=1.5]
	[step=0.25cm] (0,1) grid (1,-1);
	\draw[->] (0,-1.1) -- (0,1.1)  node[above] {$\infty$};
	\draw[->] (-0.1,0) -- (1.1,0)  node[right] {$\infty$};
	
	\fill[blue!5,] (0,1) -- (0,-1) -- (1,-1) -- (1,1) -- cycle;
	\draw[thick,variable=\t,domain=-1:1,samples=50,blue]
	plot ({0},{\t});
	\draw[thick,variable=\t,domain=0.4:1,samples=50,red]
	plot ({\t},{0});

	\fill[red] (0.4,0) circle (0.75pt);
	
	\put(15,-10){$2$};
	
	\end{tikzpicture}
	\caption{The conformal mapping for the simple random walk: $\sigma(x)$ maps conformally $\mathcal{S}_1^+\setminus [x_1,1]$ (the light blue domain) onto the cone, and $\omega(s)$ maps conformally the cone onto the right half-plane $\mathcal{H}^+$.}
	\label{fig:uniformzation-simpleRW}
\end{figure}
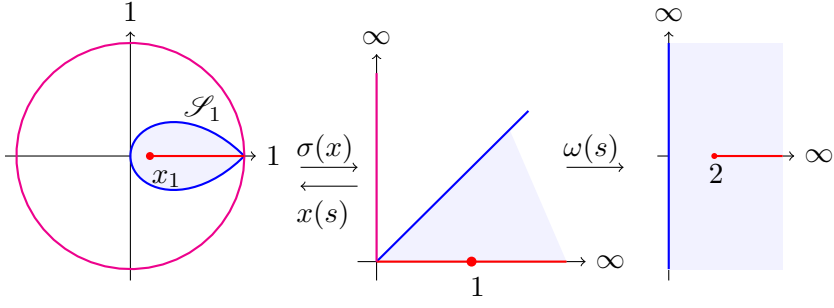

\begin{proposition}
\label{prop:expression_conformal_map_small_jumps}
Assume \ref{H1:jumps}--\ref{H5:jumps} and the small step hypothesis. Let $\theta$ as in \eqref{eq:angle_at_1} and
\begin{equation*}
   \mu(x) = \frac{\mu_0x-\mu_1}{2(x-1)},
\end{equation*}
with $\mu_0$ and $\mu_1$ defined by \eqref{eq:uniformization_points}. Then the conformal map $\psi_{1}$ may be chosen such that
\begin{equation*}
   \psi_{1}(x)=2
   T_{\pi/\theta}(\mu(x)),
\end{equation*}
where $T_{\pi/\theta}$ is the generalized Chebyshev polynomial \eqref{eq:Chebyshev_polynomial} with $a=\pi/\theta$.
\end{proposition}

We now prove Proposition \ref{prop:expression_conformal_map_small_jumps}.
The main idea of the proof is borrowed from \cite[Sec.~6.5]{FaIaMa-17}, see also \cite{Ra-14}. For driftless, small step random walks, the zero set of the kernel
\begin{equation*}
   \{(x,y)\in(\mathbb C\cup\{\infty\})^2: K(x,y)=0\}
\end{equation*}
is a Riemann surface of genus $0$, which can thus be parametrized with rational functions. As it turns out, the curves $\mathcal S_1$ and $\mathcal S_2$ become particularly simple in the uniformizing variable, from what we will deduce an expression for the conformal map.

Before stating the rational uniformisation of the above Riemann surface, we introduce a few notations. First, the kernel \eqref{eq:def_K} may be rewritten as $K(x,y) = a(x)y^2 + b(x)y + c(x)$,
where
\begin{equation*}
\left\{\begin{array}{lcl}
	a(x) &=& -(p_{-1,-1}x^2 + p_{0,-1}x + p_{1,-1}),\\
	b(x) &=& -(p_{-1,0}x^2 - x +p_{1,0}),\\
	c(x) &=& -(p_{-1,1}x^2 + p_{0,1}x + p_{1,1}).
\end{array}\right.
\end{equation*}
Let also $d=b^2-4ac$ denote the discriminant of $K(x,y)$ in $y$. It is seen that $d$ has degree $3$ or $4$, and that $1$ is a double root. In the case where $d$ has degree $3$ (resp.\ $4$), then the remaining root is denoted by $x_1$ (resp.\ the remaining roots are denoted by $x_1$ and $x_4$). It has been proved in \cite{FaIaMa-17} that $x_1\in [-1,1)$ and $x_4\in (1,\infty)\cup (-\infty,-1]$. In the case of $d$ having degree $3$, we will denote $x_4=\infty$. Now put
\begin{equation*}
\left\{\begin{array}{lcl}
	s_0 &=&\frac{2-(x_1+x_4)+2\sqrt{(1-x_1)(1-x_4)}}{x_4-x_1},\\
	s_1 &=& \frac{x_1+x_4-2x_1x_4+2\sqrt{x_1x_4(1-x_1)(1-x_4)}}{x_4-x_1},
\end{array}\right.
\end{equation*}
as well as (with $\theta$ defined in \eqref{eq:angle_at_1})
\begin{equation}
\label{eq:uniformization_points}
   \mu_0 = s_0 +\frac{1}{s_0},\quad \mu_1 = s_1 +\frac{1}{s_1},\quad\rho = e^{-i\theta}.
\end{equation}

We may now state the rational uniformization; for a proof, we refer to \cite[Sec.~2.3]{FaRa-11}.
\begin{lemma}
\label{lem:unif}
One has
\begin{equation*}
   \{(x,y)\in(\mathbb C\cup\{\infty\})^2: K(x,y)=0\}=\{(x(s),y(s)) : s\in\mathbb C\cup\{\infty\}\},
\end{equation*}
where
\begin{equation*}
	x(s) = \frac{(s-s_1)(s-\frac{1}{s_1})}{(s-s_0)(s-\frac{1}{s_0})} \quad \text{and} \quad
	y(s) = \frac{(\rho s-s_1)(\rho s-\frac{1}{s_1})}{(\rho s-s_0)(\rho s-\frac{1}{s_0})}.
\end{equation*}
Moreover, the above rational functions admit the involutions $x(s)=x(1/s)$ and $y(s)=y(1/(\rho^2s))$.
\end{lemma}
Then the set of complex points where $\vert x\vert= \vert y\vert$ (of which $\mathcal K$ in \eqref{eq:main_domain} is a subset) is very
simple, as in the $s$ variable, it becomes simply the line $e^{i\frac{\theta}{2}}\mathbb R$. More precisely, defining the
cone
    \begin{equation*}
	   \mathcal{E}^+=\{r{e}^{{i}\chi}:r>0 \text{ and } \chi\in (0,\frac{\theta}{2})\},
    \end{equation*}
one has the following result, which is illustrated on Figure \ref{fig:uniformzation-simpleRW}:
%
\begin{lemma}
    The function $x(s)$ is one-to-one from ${e}^{{i}\frac{\theta}{2}}\mathbb{R}_+$ onto $\mathcal{S}_1$. Moreover, $x(s)$ is a conformal mapping from $\mathcal{E}^+$ onto $\mathcal{S}_1^+\setminus [x_1,1]$.
\end{lemma}

\begin{proof}
The crucial fact is that $x(s)$ maps one-to-one from $\mathcal{H}^+$ onto the whole plane $\mathbb{C}$ cut along some segment. Depending on the value of $x_4$, $x(s)$ maps one-to-one from $\mathcal{H}^+$ onto
\begin{itemize}
   \item $\mathbb{C}\setminus [x_1,x_4]$ if $x_4>1$ or $x_4=\infty$;
   \item $\mathbb{C}\setminus ([x_1,\infty)\cup (-\infty,x_4])$ if $x_4<-1$.
\end{itemize}

We first prove that $x(s)$ maps one-to-one from ${e}^{{i}\theta/2}\mathbb{R}_+$ onto $\mathcal{S}_1$. Indeed, for a point $s=r{e}^{{i}\frac{\theta}{2}}$ with $r>0$, one has
\begin{equation*}
	\Im \bigl(\phi(x(s))\bigr)=\frac{2(r+\frac{1}{r})\cos\frac{\theta}{2}-(\mu_0+\mu_1)}{\mu_1-\mu_0}\geq \frac{4\cos\frac{\theta}{2}-(\mu_0+\mu_1)}{\mu_1-\mu_0} >0,
\end{equation*}
where the last inequality follows from a direct computation.
This means that $\vert x(s)\vert <1$. Combining with the facts that $x(s)=\overline{y(s)}$ and $(x(s),y(s))$ is a root of $K(x,y)$, we deduce that $x(s)\in\mathcal{S}_1$ for all $s\in {e}^{{i}\frac{\theta}{2}}\mathbb{R}_+$. The one-to-one property between the two curves then follows. To conclude, it is easy to check that $x(s)$ also maps one-to-one $[0,1]$ (resp.\ $[1,\infty)$) onto $[x_1,1]$. Hence, $x(s)$ maps one-to-one from $\mathcal{E}^+$ onto $\mathcal{S}_1\setminus [x_1,1]$.
\end{proof}

Let $\sigma(s)$ denote the inverse mapping of $x(s)$:
\begin{equation}
\label{eq:inverse_mapping_of_uniformization}
	\sigma(x) = \frac{\mu_0x-\mu_1}{2(x-1)} +\sqrt{\left (\frac{\mu_0x-\mu_1}{2(x-1)}\right )^2-1}.
\end{equation}
The branch cut of the square root in \eqref{eq:inverse_mapping_of_uniformization} is chosen such that $\sigma(s)$ maps conformally $\mathcal{S}_1^+\setminus [x_1,1]$ onto $\mathcal{E}^+$.

Now consider the following mapping
\begin{equation*}
	\omega(s)=s^{\pi/\theta} + s^{-\pi/\theta},
\end{equation*}
which maps conformally $\mathcal{E}^+$ onto $\mathcal{H}^+\setminus [2,\infty)$. It is seen that $\omega\circ\sigma(s)$ maps conformally $\mathcal{S}_1^+\setminus [x_1,1]$ onto $\mathcal{H}^+\setminus[2,\infty)$ (see Figure \ref{fig:uniformzation-simpleRW}). In fact, in the following lemma, a stronger statement can be deduced.

\begin{lemma}
\label{lem:bef_conc}
The function $\omega\circ \sigma$ maps conformally $\mathcal{S}_1^+$ onto $\mathcal{H}^+$.
\end{lemma}

\begin{proof}
By the specific form of $\omega\circ\sigma$, we know that $\omega\circ\sigma$ is analytic in $\mathcal{S}_1^+$. Moreover, since $\omega\circ\sigma$ is univalent on $\mathcal{S}_1^+\setminus [x_1,1]$, then it remains to prove that $\omega\circ\sigma$ is injective on $[x_1,1)$. This is true since $\sigma$ (with suitable branch) maps one-to-one $[x_1,1)$ onto $[1,\infty)$ and $\omega$ maps one-to-one $[1,\infty)$ onto $[2,\infty)$. Therefore, $\omega\circ\sigma$ is univalent on $\mathcal{S}_1^+$. The proof is complete.
\end{proof}

\begin{proof}[End of the proof of Proposition \ref{prop:expression_conformal_map_small_jumps}]
We use Lemma \ref{lem:bef_conc} together with the expressions for $\omega(s)$ and $s(x)$.
\end{proof}

\subsubsection*{Comparison between the analytic approaches of \cite{CoBo-83} and \cite{FaIaMa-17,Ra-14}}

Both approaches start with the same functional equation \eqref{eq:functional_equation}, and, as a second step, introduce subsets of $\mathbb C^2$ where this equation may be evaluated. The approaches differ in the choice of these subsets: 
\begin{itemize}
  \item In \cite{CoBo-83} (which we choose to follow in the present work), this subset is chosen to be $\mathcal K$ in \eqref{eq:main_domain} (where we recall that $\vert x\vert=\vert y\vert\leq 1$);
  \item On the other hand, the set in \cite{FaIaMa-17} is $x\in [x_1,1]$ (then $y=Y(x)$ is a solution to the kernel equation, and $x_1$ is the branch point introduced in the previous section).
\end{itemize}
In both cases, the method continues by stating (and solving) a BVP for the generating functions on curves obtained from the subsets above.

This short recap shows that the unique, but major difference in the two approaches lies in the choice of the domain of evaluation. Both choices are equally natural for small step random walks. However, the main advantage of the choice of \cite{CoBo-83} is that the domain $\mathcal K$ may be defined without any difficulty for models admitting arbitrary big negative jumps, as in our paper. On the contrary, we did not find any canonical way to extend the definition of the segment $[x_1,1]$ for large step models\footnote{For small step models, there is only one branch point interior to the unit disk (see again the previous section), so $x_1$ appears as the only possible choice. However, the number of branch points being increasing with the amplitude of the big jumps, it is not clear at all in general what segment, or what union of segments might replace $[x_1,1]$ in general.}. Let us also underline the $x\leftrightarrow y$ symmetry of the domain $\mathcal K$, while this symmetry is broken when taking $x\in [x_1,1]$ (indeed $x$ is then real and $y$ becomes non-real).

Finally, as shown above, in the case of small steps, the two approaches are very similar: our curve $\mathcal S_1$ corresponds to the line $e^{i\frac{\theta}{2}}\mathbb R_+$, while the curve $[x_1,1]$ would correspond to $\mathbb R_+$; one passes from one curve to the other simply by multiplying by a complex number.

\subsection{A family of random walks with larger steps}
\label{sec:larger_jumps}

Consider the model whose jumps and weights are given by
\begin{equation}
\label{eq:jumps_expression_mapping_fam_ex}
   p_{k,\ell}=\left\{\begin{array}{ll}
   z & \text{if } (k,\ell)=(1,1),\\
   z_r & \text{if } k+\ell+r=0,
   \end{array}\right.
\end{equation}
where the $z_r$ satisfy (so that the $p_{k,\ell}$ are transition probabilities summing to $1$)
\begin{equation}
\label{eq:condition1_zr}
   z+\sum_{r} (r+1)z_r=1
\end{equation}
and (so as to have a zero drift, see our hypothesis \ref{H4:jumps} in Section~\ref{sec:introduction})
\begin{equation}
\label{eq:condition2_zr}
   z=\sum_{r} z_r \frac{r(r+1)}{2}.
\end{equation}
See Figure~\ref{fig:large_jumps_example}. For example, choosing $z=z_1=\frac{1}{3}$ and all other $z_r=0$ leads to Kreweras' step set $\{(1,1), (-1,0), (0,-1)\}$ with uniform weights.

Let us remark that the model \eqref{eq:jumps_expression_mapping_fam_ex} is not always irreducible. More precisely, it is reducible if and only if all steps $(k,\ell)$ have even size (by which we mean that all coordinate sums $k+\ell$ are even). Although our irreducibility assumption \ref{H3:jumps} is not satisfied in general, we will show how to apply our main results in this slightly modified framework.

Our motivation to look at this particular family of model comes from bipolar orientations on planar maps, which, as shown in the papers \cite{KeMiSh-19} and \cite{BMFuRa-20}, are in close correspondence with the model of walks confined to the first quadrant as on the right of Figure \ref{fig:large_jumps_example}. Then our model (on the left on the same picture) is just obtained through a horizontal symmetry, so as in particular to have a model symmetric in the first diagonal. Because of the connection with bipolar orientations, both models represented on Figure~\ref{fig:large_jumps_example} admit a very strong structure, which the results in this section will also illustrate.

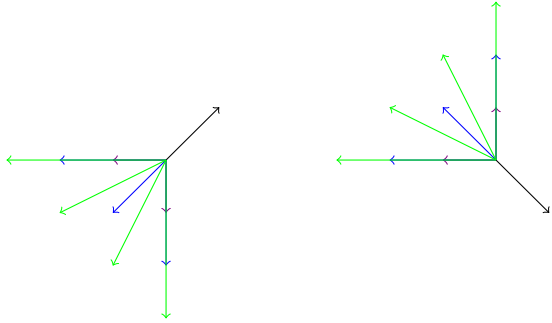
\begin{figure}
\begin{center}
   \begin{tikzpicture}[scale=.7] 
    \draw[->,white] (-3,-3) -- (-3,3);
    \draw[->,white] (3,-3) -- (-3,3);
    \draw[->,black] (0,0) -- (1,1);
    \draw[->,violet] (0,0) -- (0,-1);
    \draw[->,violet] (0,0) -- (-1,0);
    \draw[->,blue] (0,0) -- (-2,0);
    \draw[->,blue] (0,0) -- (0,-2);
    \draw[->,blue] (0,0) -- (-1,-1);
    \draw[->,green] (0,0) -- (-3,0);
    \draw[->,green] (0,0) -- (-2,-1);
    \draw[->,green] (0,0) -- (-1,-2);
    \draw[->,green] (0,0) -- (-0,-3);
  \end{tikzpicture}
     \begin{tikzpicture}[scale=.7] 
    \draw[->,white] (-3,-3) -- (-3,3);
    \draw[->,white] (3,-3) -- (-3,3);
    \draw[->,black] (0,0) -- (1,-1);
    \draw[->,violet] (0,0) -- (0,1);
    \draw[->,violet] (0,0) -- (-1,0);
    \draw[->,blue] (0,0) -- (-2,0);
    \draw[->,blue] (0,0) -- (0,2);
    \draw[->,blue] (0,0) -- (-1,1);
    \draw[->,green] (0,0) -- (-3,0);
    \draw[->,green] (0,0) -- (-2,1);
    \draw[->,green] (0,0) -- (-1,2);
    \draw[->,green] (0,0) -- (-0,3);
  \end{tikzpicture}
  \end{center}
  \caption{On the left: the model studied in Section \ref{sec:larger_jumps}. It is inspired by a similar (non-symmetric) model analyzed in \cite{KeMiSh-19,BMFuRa-20}, displayed on the right picture.}
  \label{fig:large_jumps_example}
\end{figure}

Introduce
\begin{equation}
\label{eq:expression_rho}
   \rho(x)=z\sqrt{\frac{I_0(x)-I_0(t)}{I_0(x)-I_0(1)}},
\end{equation}
where the function $I_0$, introduced in \cite{BMFuRa-20}, takes the value
\begin{equation*}
   I_0(x)=x+\frac{z}{x}-\sum z_r x^{r+1},
\end{equation*}
where $t$ is the unique real point of $\mathcal S_1$ apart from $1$, and where the branch cut of the square root function in \eqref{eq:expression_rho} is chosen on $\mathbb{R}_-$. Equivalently, $t$ is characterized by 
\begin{equation}
\label{eq:characterization_t}
   I_0'(t)=0 \quad \text{and} \quad t\in (-1,1),
\end{equation}
as it easily follows from \eqref{eq:K(x,x)} below. Our main result is the following:
\begin{proposition}
\label{prop:expression_mapping_fam_ex}
The map $\rho$ in \eqref{eq:expression_rho} is a conformal map from $\mathcal S_1^+$ to $\mathcal{H}^+$, sending $0$ to~$z$.
\end{proposition}

The remainder of this section has the following structure: we will first provide the proof of Proposition \ref{prop:expression_mapping_fam_ex}, and then give a few consequences of it.

\subsubsection*{Proof of Proposition \ref{prop:expression_mapping_fam_ex}}

We start by noticing that for $x\not=y$,
\begin{align*}
K(x,y)&=xy-z-xy\sum z_r\frac{y^{r+1}-x^{r+1}}{y-x}\\
&=\frac{xy}{y-x}\left(y-x-\frac{z}{x}+\frac{z}{y}-\sum z_r(y^{r+1}-x^{r+1})\right)\\&=\frac{xy}{y-x}\bigl(I_0(y)-I_0(x)\bigr),
\end{align*}
and for $x=y$
\begin{equation}
\label{eq:K(x,x)}
   K(x,x)=x^2-z-\sum (r+1)z_rx^{r+2}=x^2I_0'(x).
\end{equation}
Denote $\mathbb{P}^1=\mathbb C\cup\{\infty\}$.

\begin{lemma}
\label{lem:prelim_prop:expression_mapping_fam_ex}
If $(x,y)\in\mathcal{K}$, then $I_0(x)=I_0(y)\in\mathbb{R}$, and the map $(x,y)\mapsto I_0(x)$ is two-to-one from $\mathcal{K}$ to an interval $[a,b]$ of $\mathbb{R}$, with $a=I_0(t)<0$ and $b=I_0(1)>0$. Moreover, $I_0$ maps the curve $\mathcal S_1^+$ onto $\mathbb{P}^1\setminus [a,b]$.
\end{lemma}

\begin{proof}
Note first that if $x=y$, then the condition $x=\overline{y}$ implied by the symmetry of $K$ yields $x=\pm1$, which implies $I_0(x)\in\mathbb{R}$. The point $(1,1)$ always belongs to $\mathcal{K}$, and so does $(-1,-1)$ if and only if the walk is reducible (which, we recall, happens when all steps have even size).

If $x\not=y$, the above expression of $K$ yields $I_0(x)=I_0(y)$. Since $I_0$ has real coefficients and $y=\overline{x}$, we also have $I_0(x)=I_0(\overline{x})=\overline{I_0(x)}$, and $I_0(x)\in \mathbb{R}$.

By the equality $K(x,x)=x^{2}I_0'(x)$, see \eqref{eq:K(x,x)}, we get that $I'_0(x)=0$ if and only if $K(x,x)=0$. By Rouch\'e's theorem, for $h\in(0,1)$, the map
\begin{equation*}
   K_{h}(x)=x^2-h\bigl(z+\sum (r+1)z_rx^{r+2}\bigr)
\end{equation*}
has exactly two zeros in the unit disk, which are real. Hence, letting $h$ go to one yields that $x\mapsto K(x,x)$ has at most two real zeros in the unit disk, plus eventual additional zeros on the unit circle. Moreover, $K(x,x)$ cannot vanish on the unit circle except at $-1,1$; this follows from \ref{H3:jumps} in the irreducible case, and would follow from a similar argument in the reducible case. By the study of $x\mapsto K(x,x)$ on $[-1,1]$ done in Section \ref{sec:func_eq}, there are only two zeros at the two real points of $\mathcal S_1$, one being $1$ and the other one being strictly negative (since $p_{1,1}>0$). Hence, $I'_0(x)$ vanishes only at the two real points of $\mathcal S_1$. Since $I_0(x)=I_0(\overline{x})$, we deduce that $I_0$ is two-to-one from $\mathcal S_1$ to an interval $[a,b]$ of $\mathbb{R}$. Since $z>\sum z_r$, we deduce that $I_0(1)>0$ and $I_0(t)<0$, where $t$ is the unique real point of $\mathcal S_1$ apart from $1$ (recall that $t$ is negative by Lemma \ref{lemma: properties of S1 S2'}).

Since $I_0$ is analytic on $\mathcal S_1^+$ and continuous on the closure $\overline{\mathcal S_1^+}$, 
\begin{equation*}
   \partial \bigl(I_0(\mathcal S_1^+)\bigr)\subset I_0(\partial \mathcal S_1^+)=I_0(\mathcal S_1)=[I_0(t),I_0(1)]=[a,b].
\end{equation*}
Hence, $\mathbb{P}^1\setminus [a,b]\subset I_0(\mathcal S_1^+)$. Let $x\in \mathcal S_1^+\setminus \mathbb{R}$. Then, $I_0(x)\not\in[a,b]$, for otherwise $K(x,\overline{x})=0$ which would contradict the fact that $x\not\in \mathcal S_1$. On $\mathbb{R}\cap \mathcal S_1^+$, $x^2I_0'(x)=K(x,x)<0$, thus $I_0$ is increasing. Since $I_0(0)=\infty$, we deduce that $I_0(\mathbb{R}\cap \mathcal S_1^+)=(\mathbb{R}\setminus[a,b])\cup\{\infty\}$, and thus $I_0(\mathcal S_1^+)=\mathbb{P}^1\setminus [a,b]$.
\end{proof}

\begin{proof}[End of the proof of Proposition \ref{prop:expression_mapping_fam_ex}]
First, we directly check that $\rho(0)=z\sqrt{1}=z$. Further, since
\begin{equation*}
   z\mapsto \sqrt{\frac{z-a}{z-b}}
\end{equation*}
maps conformally $\mathbb{P}^1\setminus[a,b]$ to the right half-plane $\mathcal{H}^+$, we deduce from Lemma \ref{lem:prelim_prop:expression_mapping_fam_ex} that $\rho$ in \eqref{eq:expression_rho} is analytic from $\mathcal S_1^+$ to $\mathcal{H}^+$. As a second step, we will prove that $\rho$ is continuous and one-to-one from $\mathcal S_1$ to $[a,b]$. As a direct consequence, $\rho$ will be conformal from $\mathcal S_1^+$ to $\mathcal{H}^+$, thereby concluding the proof of Proposition \ref{prop:expression_mapping_fam_ex}.

We thus show that $\rho$ extends by continuity to an injective map from $\mathcal S_1$ to $[a,b]\subset\mathbb{R}$. Suppose that $\{x_n\}_{n\geq 1}$ converges to $\xi\in \mathcal S_1$ with $\Im\xi>0$. Then, by continuity of the function $I_0$, $I_0(x_n)\rightarrow I_0(\xi)$. Since $I_0'(x)<0$ for $x\in [t,1]$ and $I_0([t,1])\subset\mathbb{R}$, we deduce that $\Im I_0(x)<0$ for $x$ in a neighbourhood of $[1,t]$ in $i\mathcal{H}^+$. Taking into account that $I_0(\mathcal S_1^+\cap i\mathcal{H}^+)\subset \mathbb{C}\setminus \mathbb{R}$ (for otherwise it would meet $\mathcal S_1$), we have $I_0(\mathcal S_1^+\cap i\mathcal{H}^+)\subset -i\mathcal{H}^+$. Hence, $I_0(x_n)\rightarrow I_0(\xi)\in [a,b]$ while staying in $-i\mathcal{H}^+$. This implies that 
\begin{equation*}
   \frac{I_0(x_n)-a}{I_0(x_n)-b} =1+\frac{b-a}{I_0(x_n)-b}\to  \frac{I_0(\xi)-a}{I_0(\xi)-b},
\end{equation*}
while being after some rank in $i\mathcal{H}^+$. Thus, $\rho(x_n)$ goes towards $iz\sqrt{\frac{I_0(\xi)-a}{b-I_0(\xi)}}$, where $\sqrt{\frac{I_0(\xi)-a}{b-I_0(\xi)}}$ is the unique positive root of $X^2=\frac{I_0(\xi)-a}{b-I_0(\xi)}$.

Similarly, if $\{x_n\}_{n\geq 1}$ converges to $\xi\in \mathcal S_1$ with $\Im\xi<0$, then $\rho(x_n)$ goes to $-iz\sqrt{\frac{I_0(\xi)-a}{b-I_0(\xi)}}$. Hence, we can extend $\rho$ by continuity to $\mathcal S_1$ with the value 
\begin{equation*}
   \rho(\xi)=\pm iz \sqrt{\frac{I_0(\xi)-a}{b-I_0(\xi)}}\qquad \text{if }\xi\in \pm i\mathcal{H}^+.
\end{equation*}
Notice that the above formula still holds when $\xi\in\mathbb{R}$, with $\rho(1)=\infty$ and $\rho(t)=0$. Since $I_0$ is two-to-one from $\mathcal S_1$ to $[a,b]$ except at $t$ and $1$, we deduce that $\rho$ is injective on $\mathcal S_1$.
\end{proof}

\subsubsection*{Applications of Proposition \ref{prop:expression_mapping_fam_ex}}
Theorem \ref{thm:main_intro-1} directly yields the following explicit expression for the harmonic functions $h_n$ as introduced in \eqref{eq:expression_H_n}.
\begin{proposition}
\label{prop:expression_HF_family}
Let $t$ as in \eqref{eq:characterization_t}. Then, for each $n\geq 1$, the power series expansion at $(0,0)$ of the bivariate series
\begin{equation*}
   H_n(x,y)=\frac{(y-x)z^n\left(\left(\frac{I_0(x)-I_0(t)}{I_0(x)-I_0(1)}\right)^{n/2}+(-1)^{n-1}\left(\frac{I_0(y)-I_0(t)}{I_0(y)-I_0(1)}\right)^{n/2}\right)}{xy\bigl(I_0(y)-I_0(x)\bigr)}
\end{equation*}
defines a harmonic function $h_n$ for the Laplacian operator associated to the model \eqref{eq:jumps_expression_mapping_fam_ex}.
\end{proposition}
Observe that for all even values of $n\geq 1$, $H_n$ is a rational function (when the jumps are bounded).

For Kreweras' model, $I_0(x)=x+\frac{1}{3x}-\frac{x^2}{3}$. Hence, $I_0'(x)=1-\frac{1}{3x^2}-\frac{2x}{3}$, which admits the unique root $t=-\frac{1}{2}$ in $(-1,1)$. Hence, after some computations, $\rho$ in \eqref{eq:expression_rho} simplifies into 
\begin{equation}
\label{eq:rho_example-1}
   \rho(x)=\frac{1+2x}{3}\sqrt{\frac{1-x/4}{(1-x)^3}},
\end{equation}
and we have for example
\begin{align*}
   H_1(x,y)&=\frac{(1+2x)\sqrt{\frac{1-x/4}{(1-x)^3}}+(1+2y)\sqrt{\frac{1-y/4}{(1-y)^3}}}{3xy\bigl(1-\frac{1}{3}(\frac{1}{xy}+x+y)\bigr)}\\&=-\frac{1}{18}\left(1+\frac{27}{16}x+\frac{27}{16}y+\frac{567}{256}x^2+3xy+\frac{567}{256}y^2+\cdots\right).
\end{align*}
We also have 
\begin{equation*}
   H_2(x,y)=-\frac{9}{4}\frac{x-y}{(1-x)^3(1-y)^3}=-\frac{9}{4}\left(x-y+3x^2-3y^2+\cdots\right).
\end{equation*}

As a second example, consider $z=\frac{1}{2}$, $z_2=\frac{1}{6}$ and all other $z_r=0$. Then the curve admits the parametrization 
\begin{equation}
\label{eq:curve_S1_big_jumps_Example}
   \mathcal{S}_1=\left\lbrace\frac{1}{\sqrt{1+\frac{2}{\sqrt{3}}\sin t}}e^{it}: t\in[0,2\pi)\right\rbrace,
\end{equation}
see Figure \ref{fig:curve_S1_example}. In this case, the conformal mapping takes the form
\begin{equation}
\label{eq:rho_example-2}
   \rho(x)=\frac{1}{2}\sqrt{\frac{(3-x)(1+x)^3}{(3+x)(1-x)^3}}.
\end{equation}

\subsubsection*{Some universality results}

The following result shows that the family studied in this section has a universal behavior:
\begin{lemma}
\label{lem:uniform_2pi/3}
For any choice of parameters $z$ and $z_r$ in \eqref{eq:jumps_expression_mapping_fam_ex} satisfying to \eqref{eq:condition1_zr}, \eqref{eq:condition2_zr} and $\sum_{r} r^4z_{r}<\infty$, we have  $\theta=\frac{2\pi}{3}$.
\end{lemma}
Observe that the hypothesis $\sum_{r} r^4z_{r}<\infty$ is equivalent to our moment assumption \ref{H4:jumps}. Moreover,  Lemma~\ref{lem:uniform_2pi/3} is equivalent to the following statement: under the exact same assumptions (as in Lemma~\ref{lem:uniform_2pi/3}), the conformal mapping $\rho$ introduced in \eqref{eq:expression_rho} admits for $\vert x\vert<1$ close to $1$ the expansion
\begin{equation}
\label{eq:rho_example-univ}
   \rho(x)=(1-x)^{-3/2}\cdot \bigl(c+o(1)\bigr),
\end{equation}
where $c\neq 0$. See \eqref{eq:rho_example-1} and \eqref{eq:rho_example-2} for two examples. The above expansion comes from \eqref{eq:jumps_expression_mapping_fam_ex} together with the fact that in the neighbourhood of $1$,
\begin{equation}
\label{eq:expansion_I0_1}
   I_0(x)=I_0(1)+(x-1)^3\cdot \bigl(c'+o(1)\bigr),
\end{equation}
where $c'$ is a non-zero constant.

Using classical singularity analysis starting from Proposition \ref{prop:expression_HF_family} and Equation~\eqref{eq:rho_example-univ}, a consequence on the (conjecturally) positive harmonic function $h_1(i,j)$ is that as $i$ goes to infinity,
\begin{equation}
\label{eq:behavior_h_1(i,1)}
   h_1(i,1)\sim c \cdot i^{1/2},
\end{equation}
where $c$ is some positive constant. This holds under the assumptions of Lemma~\ref{lem:uniform_2pi/3}.

Using \cite{DeWa-15} and assuming that $\sum_r r^{4+\epsilon}z_r<\infty$ for some $\epsilon>0$, we know that as both $i$ and $j$ go to $\infty$ in some angular direction $j/i\to\tan\gamma$ with $\gamma\in(0,\pi/2)$, we have
\begin{equation}
\label{eq:behavior_h_1(i,1)_DW}
   h_1(i,j)\sim c_\gamma\cdot \bigl(\sqrt{i^2+j^2}\bigr)^{3/2},
\end{equation}
with $c_\gamma>0$, but we are not able to deduce this joint asymptotics from Proposition \ref{prop:expression_HF_family} and bivariate singularity analysis.

\begin{proof}[Proof of Lemma~\ref{lem:uniform_2pi/3}]
First, one easily computes
\begin{equation*}
   \sum_{k,\ell}k^2p_{k,\ell}=z+\sum_{r\geq 1}z_r \sum_{k= 1}^{r}k^2=z+\sum_{r\geq 1}\frac{r(r+1)(2r+1)}{6}z_r. 
\end{equation*}
Similarly, one has
\begin{equation*}
   \sum_{k,\ell}k\ell p_{k,\ell}=z+\sum_{r\geq 1}\frac{r(r-1)(r+1)}{6}z_r. 
\end{equation*}
Using \eqref{eq:condition2_zr}, one deduces that $\sum_{k,\ell}k^2p_{k,\ell}=2\sum_{k,\ell}k\ell p_{k,\ell}$, so that $\theta=\arccos(-\frac{1}{2})=\frac{2\pi}{3}$ by \eqref{eq:angle_at_1}.  Similar computations are performed in \cite[Lem.~8.1]{BMFuRa-20}.
\end{proof}

\subsection{An example with less moments}
\label{subsec:example_less_moments}
We first recall the general condition under which the positive harmonic function is constructed in \cite{DeWa-15}: in dimension $2$, it is assumed that the transition probabilities admit moments of order $2+\delta$ (with $\delta>0$) if $p=\frac{\pi}{\theta}\leq 2$ and of order $p$ if $p>2$. In this part, we would like to explore random walks with jumps as in \eqref{eq:jumps_expression_mapping_fam_ex}, but without second moment, meaning that $\sum_{r} r^4z_{r}=\infty$.

We first introduce our parameters $z$ and $z_r$. Given $\epsilon>0$, we define 
\begin{equation}
\label{eq:constant_Ce}
   C(\epsilon) =\frac{2}{\zeta(1+\epsilon)+3\zeta(2+\epsilon)+2\zeta(3+\epsilon)},
\end{equation}
where $\zeta$ denotes the classical Zeta Riemann function. We further define
\begin{equation}
\label{eq:constant_zz}
z = \frac{C(\epsilon)}{2} \bigl(\zeta(1+\epsilon)+\zeta(2+\epsilon)\bigr)
\qquad \text{and} \qquad  z_r = \frac{C(\epsilon)}{r^{3+\epsilon}}.
\end{equation} 
It is easy to verify that for this choice of weights, \eqref{eq:condition1_zr} and \eqref{eq:condition2_zr} are satisfied. Moreover, we check that for any $a<1+\epsilon$, the random walk has moments of order $a$. In particular, if $\epsilon>1$ then we have moments of order $2$, which is the classical framework of the paper (notice that in this case,  $p=\frac{3}{2}=\frac{\pi}{\theta}$).

The new, interesting case corresponds to $\epsilon\in(0,1]$, for which the second moment is infinite. Our first remark is that it is still possible to construct the harmonic functions, following the exact same steps leading to Proposition~\ref{prop:expression_HF_family}. To our knowledge, this is the first time that harmonic functions are constructed beyond the classical hypothesis of \cite{DeWa-15}. We thank the referee for suggesting us this possibility.

We can actually go further, and give some properties related to the growth at infinity of these harmonic functions. Indeed, in the classical case, the polynomial growth $\frac{3}{2}$ of the harmonic function (see \eqref{eq:behavior_h_1(i,1)_DW}) is directly related to the power $3$ in the expansion \eqref{eq:expansion_I0_1}. It is thus natural to see what now replaces \eqref{eq:expansion_I0_1} when the second moment does not exist. We will prove the following:
\begin{equation}
\label{eq:expansion_1_NO2}
   I_0(x)=I_0(1)+\left\{\begin{array}{ll} (x-1)^3\log(1-x)\cdot\bigl(c_1+o(1)\bigr) & \text{if } \epsilon=1,\smallskip\\
    (1-x)^{2+\epsilon}\cdot\bigl(c_\epsilon+o(1)\bigr)& \text{if } \epsilon \in(0,1),
    \end{array}\right.
\end{equation} 
where $c_\epsilon$, $\epsilon\in(0,1]$, is a non-zero constant. The above equation exhibits a phase transition when the moment assumption varies. We are not able to deduce from \eqref{eq:expansion_1_NO2} a joint bivariate asymptotics for the associated harmonic function $h(i,j)$ as both $i$ and $j$ tend to infinity, but a classical univariate singularity analysis gives the following estimate, which generalizes \eqref{eq:behavior_h_1(i,1)}:
\begin{equation}
\label{eq:behavior_h_1(i,1)_less-moments}
   h_1(i,1)\sim c_\epsilon \left\{\begin{array}{ll} i^{1/2}\cdot \log^{1/2} i & \text{if } \epsilon=1,\smallskip\\
    i^{\epsilon/2}& \text{if } \epsilon \in(0,1),
    \end{array}\right.
\end{equation}
where again $c_\epsilon$ denotes some positive constant.

\begin{proof}[Proof of Equation~\eqref{eq:expansion_1_NO2}]
The function $I_0$ may be expressed in terms of the classical polylogarithm function 
\begin{equation*}
   \Li_s(x) = \sum_{k\geq 1} \frac{x^k}{k^s},
\end{equation*}
the behavior of which near $1$ being well known. Precisely, using our notation \eqref{eq:constant_Ce}, one may write
\begin{equation*}
   I_0(x)=x+\frac{z}{x}-C(\epsilon)x\Li_{3+\epsilon}(x),
\end{equation*}
from which \eqref{eq:expansion_1_NO2} follows.
\end{proof}


\begin{figure}
\begin{tikzpicture}[scale=2.5]
\draw[gray,very thin] (-1.1,-1.1) grid (1.1,1.1)
	 [step=0.25cm] (-1,-1) grid (1,1);
\draw[->] (-1.1,0) -- (1.1,0)  node[right] {$1$};
\draw[->] (0,-1.1) -- (0,1.1)  node[above] {$1$};

  \textcolor{red}{\qbezier(64,-11)(49,0)(64,11)}

\draw[thick,variable=\t,domain=0:180,samples=50,blue]
  plot ({cos(\t)/sqrt(1+2/sqrt(3)*sin(\t))},{sin(\t)/sqrt(1+2/sqrt(3)*sin(\t))});
  \draw[thick,variable=\t,domain=0:180,samples=50,blue]
  plot ({cos(\t)/sqrt(1+2/sqrt(3)*sin(\t))},{sin(\t)/sqrt(1+2/sqrt(3)*sin(\t))*(-1)});
  

\put(22,5){$\theta=\frac{2\pi}{3}$}  

\put(28,23){$\mathcal S_1$}  
  
\end{tikzpicture}
\caption{The curve $\mathcal S_1$ in \eqref{eq:curve_S1_big_jumps_Example}, for the model \eqref{eq:jumps_expression_mapping_fam_ex} with jumps $z=\frac{1}{2}$, $z_2=\frac{1}{6}$ and all other $z_r=0$. As proved in Lemma~\ref{lem:uniform_2pi/3}, the value $\frac{2\pi}{3}$ of the angle at $1$ is a general fact for the family studied in this section, under some moment assumptions.}
\label{fig:curve_S1_example}
\end{figure}

\appendix
\section{Weierstrass' preparation and division theorems}
\label{sec:app_W}
We recall here two important results in the study of analytic functions or formal power series in several variables. We will only state these results in the case where two variables are involved, even if they hold for an arbitrary (finite) number of variables. Let $\mathbb{K}[[x,y]]$ denote the formal ring of power series in two variables, with coefficients in the field $\mathbb{K}$. Recall that $f\in \mathbb{K}[[x,y]]$ is invertible in $\mathbb{K}[[x,y]]$ if and only if the constant term of $f$ is non-zero. We will denote by $\mathbb K[[x]][y]$ the set of polynomials in $y$ with coefficients being power series in $x$.

\begin{definition}
A Weierstrass polynomial with respect to $y$ of degree $n$ is a polynomial $P\in \mathbb K[[x]][y]$ of degree $n$ with leading coefficient equal to $1$ and with non-invertible lower coefficients. In other words, $P$ can be written as
\begin{equation*}
   P(x,y)=y^n+F_{n-1}(x)y^{n-1}+\cdots+F_{0}(x),
\end{equation*}
with $F_{0},\ldots,F_{n-1}\in \mathbb{K}[[x]]$ such that $F_{i}(0)=0$ for all $0\leq i\leq n-1$. 
\end{definition}
Then, we have the following main results, respectively called Weierstrass preparation theorem and Weierstrass division theorem.

\begin{theorem}[Weierstrass preparation theorem]
\label{thm:preparation_theorem}
Suppose that $f\in \mathbb{K}[[x,y]]$ is non-invertible and that for some $k\geq 1$, the coefficient in $y^k$ of $f$ is non-zero. Then there exist a unique Weierstrass polynomial $P$ with respect to $y$ and an invertible element of $\mathbb{K}[[x,y]]$ such that 
\begin{equation*}
   f=hP.
\end{equation*}
If $f$ is the germ of an analytic function at $(0,0)$, then $h$ and $P$ are also germs of analytic functions at $(0,0)$.
\end{theorem}

\begin{theorem}[Weierstrass division theorem]\label{thm:division_theorem}
Suppose that $f\in \mathbb{K}[[x,y]]$ and $P$ is a Weierstrass polynomial of degree $n$ with respect to $y$. Then there exist a unique pair $(q,R)$ with $q\in \mathbb{K}[[x,y]]$ and $R$ a Weierstrass polynomial with respect to $y$ of degree less than $n$ such that 
\begin{equation*}
   f=qP+R.
\end{equation*}
If $f$ and $P$ are germs of analytic functions at $(0,0)$, then $q$ and $R$ are also germs of analytic functions at $(0,0)$.
\end{theorem}


\begin{thebibliography}{10}

\bibitem{Bi-92}
Ph. Biane.
\newblock \'{E}quation de {C}hoquet-{D}eny sur le dual d'un groupe compact.
\newblock {\em Probab. Theory Related Fields}, 94(1):39--51, 1992.

\bibitem{Bi-91}
Philippe Biane.
\newblock Quantum random walk on the dual of {${\rm SU}(n)$}.
\newblock {\em Probab. Theory Related Fields}, 89(1):117--129, 1991.

\bibitem{BiBoOC-05}
Philippe Biane, Philippe Bougerol, and Neil O'Connell.
\newblock Littelmann paths and {B}rownian paths.
\newblock {\em Duke Math. J.}, 130(1):127--167, 2005.

\bibitem{BoBMMe-18}
Alin Bostan, Mireille Bousquet-M\'{e}lou, and Stephen Melczer.
\newblock Counting walks with large steps in an orthant.
\newblock {\em J. Eur. Math. Soc. (JEMS)}, 23(7):2221--2297, 2021.

\bibitem{BoMuSi-15}
Aymen Bouaziz, Sami Mustapha, and Mohamed Sifi.
\newblock Discrete harmonic functions on an orthant in {$\Bbb Z^d$}.
\newblock {\em Electron. Commun. Probab.}, 20:no. 52, 13, 2015.

\bibitem{BMFuRa-20}
Mireille Bousquet-M\'{e}lou, \'Eric Fusy, and Kilian Raschel.
\newblock Plane bipolar orientations and quadrant walks.
\newblock {\em S\'{e}m. Lothar. Combin.}, 81:Art. B81, 2020.

\bibitem{BMMi-10}
Mireille Bousquet-M\'{e}lou and Marni Mishna.
\newblock Walks with small steps in the quarter plane.
\newblock In {\em Algorithmic probability and combinatorics}, volume 520 of
  {\em Contemp. Math.}, pages 1--39. Amer. Math. Soc., Providence, RI, 2010.

\bibitem{ChFuRa-20}
Fran{\c{c}}ois Chapon, {\'E}ric Fusy, and Kilian Raschel.
\newblock Polyharmonic functions and random processes in cones.
\newblock In Michael Drmota and Clemens Heuberger, editors, {\em 31st
  International Conference on Probabilistic, Combinatorial and Asymptotic
  Methods for the Analysis of Algorithms (AofA 2020)}, volume 159 of {\em
  Leibniz International Proceedings in Informatics (LIPIcs)}, pages 9:1--9:19,
  Dagstuhl, Germany, 2020. Schloss Dagstuhl--Leibniz-Zentrum f{\"u}r
  Informatik.

\bibitem{Co-92}
J.~W. Cohen.
\newblock {\em Analysis of random walks}, volume~2 of {\em Studies in
  Probability, Optimization and Statistics}.
\newblock IOS Press, Amsterdam, 1992.

\bibitem{CoBo-83}
Jacob~Willem Cohen and O.~J. Boxma.
\newblock {\em Boundary value problems in queueing system analysis}, volume~79
  of {\em North-Holland Mathematics Studies}.
\newblock North-Holland Publishing Co., Amsterdam, 1983.

\bibitem{CoMeMiRa-17}
J.~Courtiel, S.~Melczer, M.~Mishna, and K.~Raschel.
\newblock Weighted lattice walks and universality classes.
\newblock {\em J. Combin. Theory Ser. A}, 152:255--302, 2017.

\bibitem{DeWa-15}
Denis Denisov and Vitali Wachtel.
\newblock Random walks in cones.
\newblock {\em Ann. Probab.}, 43(3):992--1044, 2015.

\bibitem{DrHaRoSi-18}
Thomas Dreyfus, Charlotte Hardouin, Julien Roques, and Michael~F. Singer.
\newblock On the nature of the generating series of walks in the quarter plane.
\newblock {\em Invent. Math.}, 213(1):139--203, 2018.

\bibitem{Du-56}
R.~J. Duffin.
\newblock Basic properties of discrete analytic functions.
\newblock {\em Duke Math. J.}, 23:335--363, 1956.

\bibitem{DuRaTaWa-20}
Jetlir Duraj, Kilian Raschel, Pierre Tarrago, and Vitali Wachtel.
\newblock Martin boundary of random walks in convex cones.
\newblock {\em Ann. H. Lebesgue}, to appear.

\bibitem{EiKo-08}
Peter Eichelsbacher and Wolfgang K\"{o}nig.
\newblock Ordered random walks.
\newblock {\em Electron. J. Probab.}, 13:no. 46, 1307--1336, 2008.

\bibitem{Pr-12}
David Ellwood, Charles Newman, Vladas Sidoravicius, and Wendelin Werner,
  editors.
\newblock {\em Probability and statistical physics in two and more dimensions},
  volume~15 of {\em Clay Mathematics Proceedings}. American Mathematical
  Society, Providence, RI; Clay Mathematics Institute, Cambridge, MA, 2012.

\bibitem{Ev-66}
M.~A. Evgrafov.
\newblock {\em Analytic functions}.
\newblock W. B. Saunders Co., Philadelphia, Pa.-London, 1966.
\newblock Translated from the Russian by Scripta Technica, Inc, Translation
  edited by Bernard R. Gelbaum.

\bibitem{FaRa-11}
G.~Fayolle and K.~Raschel.
\newblock Random walks in the quarter-plane with zero drift: an explicit
  criterion for the finiteness of the associated group.
\newblock {\em Markov Process. Related Fields}, 17(4):619--636, 2011.

\bibitem{FaIa-79}
Guy Fayolle and Roudolf Iasnogorodski.
\newblock Two coupled processors: the reduction to a {R}iemann-{H}ilbert
  problem.
\newblock {\em Z. Wahrsch. Verw. Gebiete}, 47(3):325--351, 1979.

\bibitem{FaIaMa-17}
Guy Fayolle, Roudolf Iasnogorodski, and Vadim Malyshev.
\newblock {\em Random walks in the quarter plane}, volume~40 of {\em
  Probability Theory and Stochastic Modelling}.
\newblock Springer, Cham, second edition, 2017.
\newblock Algebraic methods, boundary value problems, applications to queueing
  systems and analytic combinatorics.

\bibitem{FaRa-15}
Guy Fayolle and Kilian Raschel.
\newblock About a possible analytic approach for walks in the quarter plane
  with arbitrary big jumps.
\newblock {\em C. R. Math. Acad. Sci. Paris}, 353(2):89--94, 2015.

\bibitem{Fel-08}
William Feller.
\newblock {\em An introduction to probability theory and its applications.
  {V}ol. {II}}.
\newblock John Wiley \& Sons, Inc., New York-London-Sydney, second edition,
  1971.

\bibitem{Fe-44}
Jacqueline Ferrand.
\newblock Fonctions pr\'{e}harmoniques et fonctions pr\'{e}holomorphes.
\newblock {\em Bull. Sci. Math. (2)}, 68:152--180, 1944.

\bibitem{Gu-65}
R.C. Gunning and H.~Rossi.
\newblock {\em Analytic Functions of Several Complex Variables}.
\newblock Ams Chelsea Publishing. Prentice-Hall, 1965.

\bibitem{He-49}
H.~A. Heilbronn.
\newblock On discrete harmonic functions.
\newblock {\em Proc. Cambridge Philos. Soc.}, 45:194--206, 1949.

\bibitem{IR-08}
Irina Ignatiouk-Robert.
\newblock Martin boundary of a killed random walk on a half-space.
\newblock {\em J. Theoret. Probab.}, 21(1):35--68, 2008.

\bibitem{IR-20}
Irina Ignatiouk-Robert.
\newblock Harmonic functions of random walks in a semigroup via ladder heights.
\newblock {\em J. Theoret. Probab.}, 21, to appear.

\bibitem{IRLo-10}
Irina Ignatiouk-Robert and Christophe Loree.
\newblock Martin boundary of a killed random walk on a quadrant.
\newblock {\em Ann. Probab.}, 38(3):1106--1142, 2010.

\bibitem{Is-52}
Rufus Isaacs.
\newblock Monodiffric functions. {C}onstruction and applications of conformal
  maps.
\newblock In {\em Proceedings of a symposium}, National Bureau of Standards,
  Appl. Math. Ser., No. 18, pages 257--266. U. S. Government Printing Office,
  Washington, D. C., 1952.

\bibitem{KeMiSh-19}
Richard Kenyon, Jason Miller, Scott Sheffield, and David~B. Wilson.
\newblock Bipolar orientations on planar maps and {${\rm SLE}_{12}$}.
\newblock {\em Ann. Probab.}, 47(3):1240--1269, 2019.

\bibitem{KoSc-10}
Wolfgang K\"{o}nig and Patrick Schmid.
\newblock Random walks conditioned to stay in {W}eyl chambers of type {C} and
  {D}.
\newblock {\em Electron. Commun. Probab.}, 15:286--296, 2010.

\bibitem{KuMa-98}
I.~A. Kurkova and V.~A. Malyshev.
\newblock Martin boundary and elliptic curves.
\newblock {\em Markov Process. Related Fields}, 4(2):203--272, 1998.

\bibitem{KuSu-03}
I.~A. Kurkova and Y.~M. Suhov.
\newblock Malyshev's theory and {JS}-queues. {A}symptotics of stationary
  probabilities.
\newblock {\em Ann. Appl. Probab.}, 13(4):1313--1354, 2003.

\bibitem{KuRa-12}
Irina Kurkova and Kilian Raschel.
\newblock On the functions counting walks with small steps in the quarter
  plane.
\newblock {\em Publ. Math. Inst. Hautes \'{E}tudes Sci.}, 116:69--114, 2012.

\bibitem{LeRa-16}
C\'{e}dric Lecouvey and Kilian Raschel.
\newblock {$t$}-{M}artin boundary of killed random walks in the quadrant.
\newblock In {\em S\'{e}minaire de {P}robabilit\'{e}s {XLVIII}}, volume 2168 of
  {\em Lecture Notes in Math.}, pages 305--323. Springer, Cham, 2016.

\bibitem{Ma-72}
V.~A. Maly\v{s}ev.
\newblock An analytic method in the theory of two-dimensional positive random
  walks.
\newblock {\em Sibirsk. Mat. \v{Z}.}, 13:1314--1329, 1421, 1972.

\bibitem{Mu-64}
B.~H. Murdoch.
\newblock A theorem on harmonic functions.
\newblock {\em J. London Math. Soc.}, 39:581--588, 1964.

\bibitem{Mu-06}
Sami Mustapha.
\newblock Gaussian estimates for spatially inhomogeneous random walks on {${\bf
  Z}^d$}.
\newblock {\em Ann. Probab.}, 34(1):264--283, 2006.

\bibitem{PeWi-13}
Robin Pemantle and Mark~C. Wilson.
\newblock {\em Analytic combinatorics in several variables}, volume 140 of {\em
  Cambridge Studies in Advanced Mathematics}.
\newblock Cambridge University Press, Cambridge, 2013.

\bibitem{PiWo-92}
Massimo~A. Picardello and Wolfgang Woess.
\newblock Martin boundaries of {C}artesian products of {M}arkov chains.
\newblock {\em Nagoya Math. J.}, 128:153--169, 1992.

\bibitem{Po-92}
Ch. Pommerenke.
\newblock {\em Boundary behaviour of conformal maps}, volume 299 of {\em
  Grundlehren der Mathematischen Wissenschaften [Fundamental Principles of
  Mathematical Sciences]}.
\newblock Springer-Verlag, Berlin, 1992.

\bibitem{Ra-11}
Kilian Raschel.
\newblock Green functions for killed random walks in the {W}eyl chamber of
  {${\rm Sp}(4)$}.
\newblock {\em Ann. Inst. Henri Poincar\'{e} Probab. Stat.}, 47(4):1001--1019,
  2011.

\bibitem{Ra-14}
Kilian Raschel.
\newblock Random walks in the quarter plane, discrete harmonic functions and
  conformal mappings.
\newblock {\em Stochastic Process. Appl.}, 124(10):3147--3178, 2014.
\newblock With an appendix by Sandro Franceschi.

\bibitem{Va-99}
N.~Th. Varopoulos.
\newblock Potential theory in conical domains.
\newblock {\em Math. Proc. Cambridge Philos. Soc.}, 125(2):335--384, 1999.

\bibitem{Va-09}
Nicolas~Th. Varopoulos.
\newblock The discrete and classical {D}irichlet problem.
\newblock {\em Milan J. Math.}, 77:397--436, 2009.

\bibitem{Va-15}
Nicolas~Th. Varopoulos.
\newblock The discrete and classical {D}irichlet problem: {P}art {II}.
\newblock {\em Milan J. Math.}, 83(1):1--20, 2015.

\end{thebibliography}
\end{document}